\documentclass[11pt, a4paper]{article}
\usepackage{amsfonts, amsmath, amsthm, amssymb, enumitem, mathtools, float}
\usepackage{array, multirow, bbm}

\PassOptionsToPackage{hyphens}{url}\usepackage[hidelinks]{hyperref}

\usepackage{pifont}

\usepackage{makecell} 

\usepackage{xfrac}

\usepackage{tikz-cd}

\theoremstyle{plain}
\newtheorem{theorem}{Theorem}[section]

\newtheorem{corollary}[theorem]{Corollary}
\newtheorem{lemma}[theorem]{Lemma}
\newtheorem{proposition}[theorem]{Proposition}
\newtheorem{computation}[theorem]{Computation}
\newtheorem*{claim*}{Claim}
\newtheorem*{problem*}{Problem}
\newtheorem*{conjecture*}{Conjecture}
\newtheorem*{theorem*}{Theorem}

\theoremstyle{definition}
\newtheorem{definition}[theorem]{Definition}

\newtheorem{remark}[theorem]{Remark}

\newcommand\al{\alpha}
\newcommand\bt{\beta}
\newcommand\gm{\gamma}
\newcommand\dl{\delta}
\newcommand\lm{\lambda}

\newcommand\cF{\mathcal{F}}
\newcommand\cJ{\mathcal{J}}
\newcommand\cM{\mathcal{M}}

\newcommand\la{\langle}
\newcommand\ra{\rangle}
\newcommand\lla{\langle\!\langle}
\newcommand\rra{\rangle\!\rangle}

\newcommand\ad{\mathrm{ad}}

\DeclareMathOperator{\Aut}{Aut}
\DeclareMathOperator{\ch}{char}

\DeclareMathOperator{\Sym}{Sym}
\DeclareMathOperator{\Ann}{Ann}

\newcommand\R{\mathbb{R}}
\newcommand\Q{\mathbb{Q}}
\newcommand\N{\mathbb{N}}
\newcommand\Z{\mathbb{Z}}
\newcommand\FF{\mathbb{F}}

\newcommand{\A}{\mathrm{A}}
\newcommand{\B}{\mathrm{B}}
\newcommand{\C}{\mathrm{C}}
\newcommand{\J}{\mathrm{J}}
\newcommand{\Y}{\mathrm{Y}}
\newcommand{\IY}{\mathrm{IY}}

\DeclareMathOperator{\Miy}{Miy}

\newcommand{\Cl}{\widehat{S}(2)^\circ} 

\newcommand{\cH}{\mathcal{H}}
\newcommand{\hatH}{\hat \cH}

\newcommand{\half}{{\tau_{\sfrac{1}{2}}}} 

\renewcommand{\phi}{\varphi}
\renewcommand{\epsilon}{\varepsilon}
\newcommand{\1}{\mathbbm{1}}

\newcommand{\rt}{\xi}
\newcommand{\cAM}{\widetilde{\cM}}

\newcommand{\V}{\mathbf{V}}
\newcommand{\I}{\mathbf{I}}

\setlist[enumerate,1]{label={\upshape (\arabic*)}}
\setlist[enumerate,2]{label={\upshape (\alph*)}}
\setlist[enumerate,3]{label={\upshape (\roman*)}}

\usepackage{ltablex} 
\usepackage{tabularx}
\keepXColumns

\newcolumntype{C}[1]{>{\centering\arraybackslash}m{#1}}
\newcolumntype{Y}{>{\centering\arraybackslash}X}

\title{A guide to the $2$-generated axial algebras of Monster type}
\author{J.~M\textsuperscript{c}Inroy\footnote{Department of Mathematics, University of Chester, Exton Park, Parkgate Rd, Chester, CH1 4BJ, UK, and School of Mathematics, University of Bristol, Fry Building, Woodland Road, Bristol, BS8 1UG, UK, email: j.mcinroy@chester.ac.uk}
\and
A.W.~Mir\footnote{Department of Mathematics, University of Chester, Exton Park, Parkgate Rd, Chester, CH1 4BJ, UK, email: w.mir@chester.ac.uk}
 }

\date{\today}

\begin{document}
\maketitle

\begin{abstract}
Axial algebras of Monster type are a class of non-associative algebras which generalise the Griess algebra, whose automorphism group is the largest sporadic simple group, the Monster.  The $2$-generated algebras, which are the building blocks from which all algebras in this class can be constructed, have recently been classified by Yabe; Franchi and Mainardis; and Franchi, Mainardis and M\textsuperscript{c}Inroy.  There are twelve infinite families of examples as well as the exceptional Highwater algebra and its cover, however their properties are not well understood.

In this paper, we detail the properties of each of these families, describing their ideals and quotients, subalgebras and idempotents in all characteristics.  We also describe all exceptional isomorphisms between them.  We give new bases for several of the algebras which better exhibit their axial features and provide code for others to work with them.
\end{abstract}

\section{Introduction}

Axial algebras are a class of commutative non-associative algebras introduced by Hall, Rehren and Shpectorov \cite{Axial1}.  One of the main motivations for their introduction was to axiomatise some properties found in the $196,884$-dimensional Griess algebra which was used to define the largest sporadic simple group, the Monster.  However, almost all Jordan algebras are also examples of axial algebras.  Recently examples have been found in some unexpected areas, outside algebra and group theory, such as Hsiang algebras in the theory of non-linear PDEs \cite{Tkachev} and algebras of vector flows on manifolds \cite{Fox}.

Many classes of algebra are defined by the elements having to satisfy some identities, for example Lie algebras satisfying the Jacobi identity.  However, axial algebras are different.  An axial algebra $A$ is defined to be generated by a distinguished set of elements $X$, called \emph{axes}, whose multiplication is partially controlled by a so-called fusion law $\cF$.  For graded fusion laws, we get the key feature of axial algebras, which is that we can associate an automorphism $\tau_a$ to every axis $a$.  These generate a distinguished subgroup $\Miy A$ of $\Aut A$ called the \emph{Miyamoto group}.  For the motivating example of the Griess algebras, its Miyamoto group is indeed the Monster.

Axes are semisimple idempotents $a$ whose eigenspaces $A_\lm(a)$ multiply according to a \emph{fusion law} $\cF = (\cF, \star)$.  That is, $A_\lm(a) A_\mu(a) = A_{\lm \star \mu}(a)$, where $\lm \star \mu$ is a set of eigenvalues.  One of the key goals of axial algebras is to classify the algebras with a given fusion law $\cF$ and which groups occur as Miyamoto groups.

The smallest interesting fusion law is the Jordan fusion law $\cJ(\eta)$, which has three eigenvalues $1$, $0$ and $\eta$ (see Table \ref{tab:fusion laws}).  It is well known that idempotents in a Jordan algebra have the so-called Peirce decomposition which gives the fusion law $\cJ(\frac{1}{2})$.  Matsuo algebras $M_\eta(G,D)$, which are defined from a $3$-transposition group $G = (G, D)$, are another important family of examples which have fusion law $\cJ(\eta)$.  Axial algebras of Jordan type $\eta \neq \frac{1}{2}$ have been fully classified and the $2$-generated algebras have been classified for $\eta = \frac{1}{2}$ too \cite{Axial2}.

The other fusion law of most interest is the Monster fusion law $\cM(\al,\bt)$ which has four eigenvalues, $1$, $0$, $\al$ and $\bt$.  Indeed, the motivating example of the Griess algebra has fusion law $\cM(\frac{1}{4}, \frac{1}{32})$.  Many examples of axial algebras of Monster type $(\frac{1}{4}, \frac{1}{32})$ have been constructed, but far fewer of type $(\al, \bt)$ are known.  Of those which are known, the most important examples are the $2$-generated algebras -- those which can be generated by two axes.  These are the smallest examples and are the building blocks from which all other Monster type algebras can be constructed.  The theoretical background for constructing a completion of a so-called shape is described in \cite{forbidden} and there are two different algorithms with implementations by Pfieffer and Whybrow \cite{Maddycode} and by McInroy and Shpectorov \cite{axialconstruction} to achieve this.

The $2$-generated axial algebras of Monster type $(\al,\bt)$ have been classified (under a symmetry assumption on the generating axes) and comprise several $1$- and $2$-parameter families of examples and also the exceptional Highwater algebra and its cover.

\begin{theorem*}[\cite{yabe, highwater5, HWquo}]
A symmetric $2$-generated axial algebra of Monster type $(\alpha, \beta)$ is a quotient of one of the following:
\begin{enumerate}
\item an axial algebra of Jordan type $\alpha$, or $\beta$;
\item an algebra in one of the following families:

\begin{enumerate}
\item $3\A(\al,\bt)$, $4\A(\frac{1}{4}, \bt)$, $4\J(2\bt, \bt)$, $4\B(\al, \frac{\al^2}{2})$, $4\Y(\al, \frac{1-\al^2}{2})$, $4\Y(\frac{1}{2}, \bt)$, $5\A(\al, \frac{5\al-1}{8})$, $6\A(\al, \frac{-\al^2}{4(2\al-1)})$, $6\J(2\bt, \bt)$ and $6\Y(\frac{1}{2}, 2)$;
\item $\IY_3(\al, \frac{1}{2}, \mu)$ and $\IY_5(\al, \frac{1}{2})$;
\end{enumerate}

\item the Highwater algebra $\cH$, or its characteristic $5$ cover $\hatH$.
\end{enumerate}
\end{theorem*}

The number in the notation denotes the number of axes, where $\mathrm{I}$ denotes generically infinitely many axes (although this can be finite for certain parameter values).

Even though these algebras have been classified, many of their properties are either not easily seen from the description given, or not known.  For example, it was only after the family $\IY_3(\al,\frac{1}{2},\mu)$ was described in a different way as a split spin factor algebra that the number of axes and idempotents were known \cite{splitspin}.  Whilst the $2$-generated algebras of Jordan type are well understood and the quotients of $\cH$ and $\hatH$ were classified in \cite{HWquo}, it remains to fully understand the algebras in the twelve families in item (2) above.

In this paper, we detail the properties of the twelve families of Monster type algebras in all characteristics (except two where the fusion law is not interesting).  We also provide code in {\sc Magma} for all of the algebras \cite{githubcode, magma}.  Some of our results are by computation using this code (in particular finding idempotents is often completely infeasible without using Gr\"obner basis calculations on a computer) and we also use the code for routine but fiddly computations for some theoretical results.  It is also hoped that the {\sc Magma} code will be a valuable resource for other researchers.

We give new bases for six of the twelve algebras which better exhibit their axial structure.  In particular, this makes their axial subalgebras clear in the new basis and for $5\A(\al, \frac{5\al-1}{8})$ we can see an additional automorphism acting naturally, giving an automorphism group being the Frobenius group of order $20$ rather than just $D_{10}$.  For each algebra, we classify their axial subalgebras.  In particular, for two of the algebras, $4\Y(\frac{1}{2},\bt)$ and $6\A(\al, \frac{-\al^2}{4(2\al-1)})$, one of the axial subalgebras degenerates and drops a dimension for certain values of the parameters.

From the above classification theorem, it is not clear whether the list contains duplicates for some parameter values.  We find all isomorphisms between the algebras on this list.  We say that two axial algebras $(A_1, X_2)$ and $(A_2,X_2)$ are \emph{axially isomorphic} if there exists an algebra isomorphism $\phi \colon A_1 \to A_2$ which induces a bijection of the axes $X_1$ and $X_2$.

\begin{theorem*}
The exceptional axial isomorphisms between algebras in the twelve families of symmetric $2$-generated axial algebras of Monster type are
\[
\begin{gathered}
4\A(\tfrac{1}{4}, \tfrac{1}{8}) \cong 4\J(\tfrac{1}{4}, \tfrac{1}{8}), \qquad 4\J(\tfrac{1}{2}, \tfrac{1}{4}) \cong 4\Y(\tfrac{1}{2}, \tfrac{1}{4}), \qquad 4\B(\tfrac{1}{2}, \tfrac{1}{8}) \cong 4\Y(\tfrac{1}{2}, \tfrac{1}{8}) \\
4\B(\pm \tfrac{1}{\sqrt{2}}, \tfrac{1}{4}) \cong 4\Y(\pm \tfrac{1}{\sqrt{2}}, \tfrac{1}{4}), \qquad 6\A(\tfrac{2}{5}) \cong 6\J(\tfrac{2}{5})
\end{gathered}
\]
In addition, for all $\al$, we have $\IY_3(\al,\frac{1}{2},-\frac{1}{2}) \cong 3\A(\al, \frac{1}{2})$ and, if $\ch \FF = 5$, then $\IY_5(\al, \frac{1}{2}) \cong 5\A(\al, \frac{1}{2})$.
\end{theorem*}

Note that $\IY_3(\al,\frac{1}{2}, \mu)$ and $\IY_5(\al,\frac{1}{2})$ generically have infinitely many axes but they have finitely many axes for certain values of the parameter $\mu$ and in certain characteristics.  The algebra $\IY_3(\al,\frac{1}{2}, \mu)$ has $3$ axes precisely when $\mu = -1$ and $\IY_5(\al,\frac{1}{2})$ has $5$ axes precisely in characteristic $5$.

Interestingly, we also find that two of the algebras are isomorphic as algebras, but not as axial algebras.  If $\al \neq 2$, then $4\Y(\al, \frac{1-\al^2}{2}) \cong 4\B(\tilde{\al}, \frac{\tilde{\al}^2}{2})$, where $\tilde{\al} = 1 - \al$.  This isomorphism maps an axis $a_i$ in $4\Y(\al, \frac{1-\al^2}{2})$ to $\1 - b_i$, where $b_i$ is an axis in $4\B(\tilde{\al}, \frac{\tilde{\al}^2}{2})$.  Intriguingly, this is another example of a $2$-generated algebra having a symmetry between $\al$ and $\tilde{\al} = 1 - \al$; this has also been observed for the split spin factor algebras $\IY_3(\al,\frac{1}{2}, \mu)$ \cite{splitspin}.

Since quotients of $2$-generated axial algebras of Monster type are also $2$-generated axial algebras of Monster type (or the $1$-dimensional $1\A$), it is important to classify all these.  Table $2$ of \cite{yabe} lists some quotients of Monster (but not Jordan) type and this is claimed to be a complete list for characteristic not $5$.  In this paper, we explicitly find all the ideals and quotients including those of Jordan type for all characteristics.  In particular, we find some of Monster but not Jordan type missing from Yabe's table, including two for $6\A(\al, \frac{-\al^2}{4(2\al-1)})$ and two for $\IY_5(\al, \frac{1}{2})$.  This includes a family of algebras $W_a(\al, \frac{1}{2},1)$ originally discovered by Whybrow \cite[Theorem 5.6]{whybrowgraded} and subsequently by Afanasev in \cite{axial3evals}.  However, most interestingly, in Proposition \ref{4Bideals} we find an infinite family of non-symmetric quotients of $4\B(\al, \frac{\al^2}{2})$.

Since axes are idempotents, it is natural to ask what other idempotents there are in our algebras and whether any of these also have Monster type, or have some other interesting fusion law.  For the Norton-Sakuma algebras ($2$-generated axial algebras of Monster type $(\frac{1}{4}, \frac{1}{32})$), the idempotents over $\R$ were found by Castillo-Ramirez using numerical root finding methods \cite{Alonsoidempotents}.  We investigate the idempotents for general $(\al, \bt)$.  The families $4\A(\frac{1}{4}, \bt)$, $4\Y(\frac{1}{2}, \bt)$, $6\Y(\frac{1}{2}, 2)$, $\IY_3(\al, \frac{1}{2}, \mu)$ and $\IY_5(\al, \frac{1}{2})$ have infinitely many idempotents whilst the remaining families have finitely many, which is generically $2^{\dim A}$.  We find all the idempotents in each of the families of algebras over a function field with the exception of $6\A(\al, \frac{-\al^2}{4(2\al-1)})$, $6\J(2\bt, \bt)$ and $\IY_5(\al, \frac{1}{2})$.  For $6\A(\al, \frac{-\al^2}{4(2\al-1)})$ and $6\J(2\bt, \bt)$ we show that there are finitely many idempotents and find almost all of them.  For $\IY_5(\al, \frac{1}{2})$ there is a $2$-dimensional variety of idempotents.  For the other families with infinitely many idempotents, we identify the idempotents and parametrise them in a natural way.

It remains to identify the idempotents for any exceptional values of the parameters, we will allow us to identify the automorphism group in all cases.  We expect generically the Miyamoto group to be the full automorphism group except for $5\A(\al, \frac{5\al-1}{8})$, which has a Frobenius group of order $20$.

We also identify the algebras with so-called double axes; these are just $4\A(\frac{1}{4}, \bt)$ and $4\J(2\bt, \bt)$.  For both of these we investigate the subalgebra generated by the double axes and other properties.

The paper is structured as follows.  Section \ref{sec:background} contains a brief overview of axial algebras and the background material needed and Section \ref{sec:tech} describes the techniques developed and used, particularly for finding ideals and idempotents.  The remaining sections are then for algebras with $3$, $4$, $5$, $6$ and generically infinitely many axes.

\medskip

\noindent \textbf{Acknowledgement} We would like to thank the London Mathematical Society for the support given to the second author from an Undergraduate Research Bursary, URB-2024-74.

\section{Background}\label{sec:background}

We give a brief overview of axial algebras; for a full introduction, see \cite{survey}.

\subsection{Axial algebras}
First we need fusion laws.

\begin{definition}
A \emph{fusion law} is a set $\cF$ together with a symmetric function $\star \colon \cF \times \cF \to 2^\cF$, where $2^\cF$ denotes the power set of $\cF$.
\end{definition}

For the fusion laws of interest, the set $\cF$ is typically small, so we can represent the fusion law easily in a table.  In these we drop the set notation and in each cell we just list the elements of the set $\lm \star \mu$ for $\lm, \mu \in \cF$.  In particular, an empty cell represents the empty set.  In Table \ref{tab:fusion laws}, we list the Jordan type $\eta$ fusion law, $\cJ(\eta)$, and the Monster type fusion law $\cM(\al, \bt)$.

\begin{figure}[!ht]
\begin{center}
\begin{minipage}[t]{0.25\linewidth}
{\renewcommand{\arraystretch}{1.5}
\begin{tabular}[t]{c||c|c|c}
$\star$ & $1$ & $0$ & $\eta$ \\
\hline\hline
$1$ & $1$ & & $\eta$ \\
\hline
$0$ & & $0$ & $\eta$ \\
\hline 
$\eta$ & $\eta$ & $\eta$ & $1,0$
\end{tabular}
}
\end{minipage}\hspace{20pt}
\begin{minipage}[t]{0.33\linewidth}
{\renewcommand{\arraystretch}{1.5}
\begin{tabular}[t]{c||c|c|c|c}
$\star$ & $1$ & $0$ & $\alpha$ & $\beta$ \\
\hline\hline
$1$ & $1$ & & $\alpha$ & $\beta$ \\
\hline
$0$ & & $0$ & $\alpha$ & $\beta$ \\
\hline 
$\alpha$ & $\alpha$ & $\alpha$ & $1,0$ & $\beta$ \\
\hline 
$\beta$ & $\beta$ & $\beta$ & $\beta$ & $1,0,\alpha$
\end{tabular}
}
\end{minipage}
\end{center}
\caption{Fusion laws $\cJ(\eta)$, and $\cM(\al,\bt)$.}
\label{tab:fusion laws}
\end{figure}

Let $A$ be a commutative, not necessarily associative algebra over a field $\FF$.  For $a \in A$, let $\ad_a \colon x \mapsto ax$ denote the adjoint map with respect to $a$ and $A_\lm(a)$ denote the $\lm$-eigenspace with respect to $\ad_a$.  We will write $A_S(a) = \bigoplus_{\lm \in S} A_\lm(a)$, for $S \subseteq \cF$.

\begin{definition}
Let $\cF \subset \FF$ be a fusion law, $A$ be a commutative algebra over $\FF$ and $a \in A$.  We say $a$ is an \emph{$\cF$-axis} if 
\begin{enumerate}
\item $a$ is an idempotent;
\item $a$ is semisimple, namely $A = A_\cF = \bigoplus_{\lambda \in \cF} A_\lambda(a)$;
\item for all $\lambda, \mu \in \cF$, $A_\lambda(a) A_\mu(a) \subseteq A_{\lambda \star \mu}(a)$.
\end{enumerate}
If, in addition, $A_1(a)$ is $1$-dimensional, then we say $a$ is \emph{primitive}.
\end{definition}

Where the fusion law is understood from context, we will just call $a$ an axis.  Similarly, if the axis $a$ being considered is clear, we will just write $A_\lm$ for $A_\lm(a)$.

\begin{definition}
Let $\cF \subset \FF$ be a fusion law.  An \emph{$\cF$-axial algebra} is a commutative algebra $A$ over $\FF$ together with a distinguished set $X$ of $\cF$-axes which generate $A$.
\end{definition}

Similarly to above, if $\cF$ is clear from context, we will drop it and if $X$ is obvious then we just say that $A$ is an axial algebra with the $X$ being implicit.

The fusion laws in Table \ref{tab:fusion laws} are of particular interest to us in this paper.  We say that $A$ is an axial algebra of \emph{Monster type $(\al, \bt)$} if it has the fusion law $\cM(\al, \bt)$ and of \emph{Jordan type $\eta$} if it has the $\cJ(\eta)$ fusion law.  Note that in order for these to define a set with $4$ and $3$ elements respectively, we require that $\al, \bt \neq 1,0$ and $\al \neq \bt$ for $\cM(\al, \bt)$ and that $\eta \neq 1,0$ for $\cJ(\eta)$.

The algebras we are interested in typically have a bilinear form which associates with the algebra product.  In particular, all the $2$-generated algebras of Monster type admit a Frobenius form.

\begin{definition}
Let $A$ be an axial algebra.  A \emph{Frobenius form} is a bilinear form $(\cdot, \cdot) \colon A \times A \to \FF$ such that, for all $a,b,c \in A$,
\[
(a,bc)=(ab,c)
\]
\end{definition}

Note that a Frobenius form is necessarily symmetric.

\subsection{Gradings and the Miyamoto group and axets}

We are most interested in algebras which have a naturally associated automorphism group.  This comes from a grading of the fusion law.  Following \cite{DPSV}, we define these via homomorphisms of fusion laws.

\begin{definition}
Let $(\cF_1, \star_1)$ and $(\cF_2, \star_2)$, be two fusion laws.  A \emph{homomorphism} of fusion laws is a map $\xi \colon \cF_1 \to \cF_2$ such that
\[
\xi(\lm \star_1 \mu) = \xi(\lm) \star_2 \xi(\mu)
\]
\end{definition}

For every abelian group $T$, we get a fusion law $(T, \star_T)$ by defining $s \star_T t = \{ st\}$.  We say that a fusion law $\cF$ is \emph{$T$-graded} if there is a homomorphism from $\cF$ to $(T, \star_T)$.  Looking back at the Monster and Jordan type fusion laws in Table \ref{tab:fusion laws}, we see that they are both $C_2$-graded.

Using this, we can now define an automorphism of the algebra.  We do this just for a $C_2$-graded fusion law, which is what we are mostly concerned with in this paper.  For a more general definition, see \cite{axialstructure}.  Let $\xi \colon \cF \to C_2$ be the grading map.  Since an axis $a$ is semi-simple, $A$ decomposes as $A = \bigoplus_{\lm \in \cF} A_\lm$ and so we can define $A_\epsilon :=  \bigoplus_{\lm \in \cF, \xi(\lm) = \epsilon} A_\lm(a)$, the sum of the eigenspace which are $\epsilon$-graded.  Now define $\tau_a\colon A \to A$ by
\[
x \mapsto \begin{cases}
x & \text{if } x \in A_+ \\
-x & \text{if } x \in A_- 
\end{cases}
\]
Since multiplication of eigenspaces obeys the fusion law, this is an automorphism that we call a \emph{Miyamoto involution}.

%

\begin{definition}
For an $\cF$-axial algebra $A = (A, X)$ with a $C_2$-graded fusion law, we define the \emph{Miyamoto group} as the group
\[
\Miy(A) = \Miy(A, X) = \la \tau_a : a \in X \ra
\]
\end{definition}


The Miyamoto group is subgroup of the automorphism group, which is often all, or almost all of $\Aut A$.

Given a set $Y$ of axes, we define the \emph{closure} $\overline{Y} = Y^{\Miy(A, Y)}$ of $Y$.  By \cite{axialstructure}, $\Miy(A, \overline{Y}) = \Miy(A, Y)$ and  so we say a set $Y$ of axes is \emph{closed} if $\overline{Y} = Y$.  Hence we may always enlarge a set of axes so that it is closed.

Given a (closed) set of axes $X$, the Miyamoto group acts by permuting the elements of $X$.  In \cite{forbidden}, M\textsuperscript{c}Inroy and Shpectorov define an \emph{axet} which roughly speaking is the set of axes $X$ with a map $\tau_a \colon X \to G := \Sym(X)$ satisfying some natural conditions.  For our purposes, we may just think of the set of axes together with the Miyamoto map as being the axet.

In this paper, we concentrate on symmetric $2$-generated axial algebras.  By \cite[Theorem 1.1]{forbidden}, every such algebra has axet $X := X(n)$, where $n  \in \N \cup \{ \infty\}$.  The axet $X = X(n)$ has $n$ elements which correspond to an $n$-gon and $\tau_x$, for $x \in X$, corresponds to the reflection of the $n$-gon through the vertex $x$.  The Miyamoto group $\Miy(X) = \la \tau_x : x \in X \ra$ for a $2$-generated axet/algebra is a dihedral group.

As such, the subaxets of $X(n)$ are $X(k)$, where $k \mid n$.  For each subaxet $Y = X(k)$ of $X = X(n)$, $\lla  Y \rra$ is an axial subalgebra of $A$.  Note that, due to other multiplications in the algebra $A$, this may or may not be proper.

\subsection{$2$-generated algebras of Jordan type}

In this paper, we will encounter the $2$-generated axial algebras of Jordan type which were classified by Hall, Rehren and Shpectorov \cite{Axial1}.  The axes here have three distinct eigenvalues $1$, $0$ and $\eta$ and satisfy the fusion law given in Table \ref{tab:fusion laws}. Here we give brief details, for further details see \cite{Axial1}, or \cite[Section 5]{forbidden}. 

There are two algebras which occur for all values of $\eta$: $2\B$ and $3C(\eta)$.  The algebra $2\B \cong \FF^2$ is generated by two axes $a$ and $b$ where $ab = 0$.  The algebra $3\C(\eta)$ has a basis of axes $a,b,c$ and the multiplication is given by $xy = \frac{\eta}{2}(x+y+z)$ for $\{x,y,z\} = \{a,b,c\}$.  This algebra is simple unless $\eta = -1, 2$ and for $\eta = -1$, there is a $2$-dimensional quotient $3\C(-1)^\times$.

For $\eta = \frac{1}{2}$, there are additional algebras.  Let $V$ be a $2$-dimensional vector space with a non-degenerate symmetric bilinear form $(\cdot, \cdot)$ with gram matrix $\begin{psmallmatrix} 2 & \dl \\ \dl & 2 \end{psmallmatrix}$.  Set $A = \la \1 \ra \oplus V$ and define multiplication by
\[
(\lm \1 + u)(\mu \1 + v) = \big(\lm \mu + \tfrac{1}{2}b(u,v)\big)\1 + \lm v + \mu u
\]
This is called a \emph{spin factor} algebra and the axes have the form $x = \frac{1}{2}(\1 + u)$, where $b(u,u) = 2$.  If $u,v$ is a basis for $V$ and $\dl := b(u,v)$, then the algebra generated by $x=\frac{1}{2}(\1 + u)$ and $y=\frac{1}{2}(\1 + v)$ is denoted $S(\dl)$.

There are two additional exceptions $S(2)^\circ$ and its cover $\Cl$.  The algebra $\Cl$ has basis $x,y,z$, where $x$ and $y$ are axes, $z \in \Ann(\Cl)$ and $xy = \frac{1}{2}(x+y) + z$.  We then define $S(2)^\circ = \Cl/\Ann(\Cl)$.  The algebras $S(\dl)$, $\Cl$ and $S(2)^\circ$ usually have infinitely many axes, but can have finitely many for certain values of $\dl$, or when the field is finite.  Hall, Rehren and Shpectorov showed that these algebras are the only examples.

\begin{theorem}[\cite{Axial1}]\label{2gen jordan}
	A symmetric $2$-generated axial algebra of Jordan type $\eta$ is isomorphic to one of:
	\begin{enumerate}
		\item $2\B$,
		\item $3\C(\eta)$,
		\item $3\C(-1)^{\times}$ and $\eta = -1$,
		\item $S(\delta)$, $\delta \neq 2$ and $\eta = \tfrac{1}{2}$,
		\item $S(2)^\circ$ and $\eta = \tfrac{1}{2}$,
		\item $\Cl$ and $\eta = \tfrac{1}{2}$.
	\end{enumerate}
\end{theorem}

We note that the only possible isomorphisms between the above algebras are between $3\C(\frac{1}{2})$ and $S(\dl)$ for some specific values of $\dl$ (for further details see \cite[Remark 5.9]{forbidden}).  In addition, in characteristic $3$, $\frac{1}{2} = 2 = -1$ and $3\C(-1) \cong \Cl$ and $3\C(-1)^\times \cong S(2)^\circ$.\footnote{These are missing from \cite{forbidden}}

When we encounter a $2$-generated axial algebra of Jordan type, we want to be able to identify which one it is.  Suppose that $A = \lla x,y \rra$ is a $2$-generated axial algebra of Jordan type $\eta$.  In Table \ref{tab:2gen Jordan Table}, we list the dimension and details of the identity in the algebra and we see that these are enough to identify the algebra.

\begin{table}[h!tb]
	\centering
	\renewcommand{\arraystretch}{1.2}
	\begin{tabular}{c|c|c|c}
		Algebra & Dimension & Axet & Identity \\
		\hline
		$2\B$ & $2$ &  $X(2)$ & $a+b$ \\ 
		$3\C(\eta)$ & 3 & $X(3)$ & if $\eta \neq -1$, $\1 = \frac{1}{\eta+1}(a+b+c)$ \\ 
		$3\C(-1)^\times$ & $2$ &  $X(3)$ & no \\ 
		\hline
		\multicolumn{4}{c}{In addition, if $\eta = \frac{1}{2}$} \\
		\hline
		$S(\delta)$ & $3$ & $X(\infty)$ & yes \\ 
		$S(2)^\circ$ & $2$ & $X(\infty)$  &  no \\ 
		$\Cl$ & $3$ & $X(\infty)$ & no \\ 
		\hline
	\end{tabular}
	\caption{Properties of 2-generated algebras of Jordan type.}
	\label{tab:2gen Jordan Table}
\end{table}

\begin{remark}\label{idJordan}
The above details above and in Table \ref{tab:2gen Jordan Table} are enough to identify a $2$-generated Jordan type $\eta$ algebra $A = \lla x, y \rra$.  If $\eta \neq \frac{1}{2}$, then by considering the dimension and whether the algebra has an identity, we see it is either $2\B$, $3\C(\eta)$, or $3\C(-1)^\times$.  If $\eta = \frac{1}{2}$ and the algebra is $2$-dimensional, then it is either $2\B$, or $S(2)^\circ$ and these are easily distinguished by checking if $xy$ is $0$, or $\frac{1}{2}(x+y)$.  If the algebra is $3$-dimensional and has no identity, then it is $\Cl$.  Finally, if the algebra is $3$-dimensional, then it is either $3\C(\frac{1}{2})$, or $S(\dl)$, for some $\dl \in \FF$.   If $c := a+b - 4ab$ is an axis, then $A \cong 3\C(\frac{1}{2})$, otherwise it is $S(\dl)$, where $\dl$ can be found from
\[
xy = \tfrac{1}{4}\big( (1 + \tfrac{1}{2}b(u, v))\1 + u + v \big)= \tfrac{1}{2}(x + y) + \tfrac{1}{8}(\dl - 2)\1
\]
\end{remark}

\subsection{Symmetric $2$-generated algebras of Monster type}

Let $A = \lla a_0, a_1 \rra$ be a symmetric $2$-generated algebra of Monster type.  Then its set of axes is $X = \{ a_0\}^{\Miy(A)} \cup \{ a_1\}^{\Miy(A)}$.  Let $\rho := \tau_0 \tau_1$.  We can enumerate the axes by
\[
a_{2i} := {a_0}^{\rho^i}, \quad a_{2i+1} := {a_1}^{\rho^i}
\]
for all $i \in \Z$.  Note that if the complete set of closed axes (axet) is finite of size $n$, then $a_i = a_j$ if and only if $i \equiv j \mod n$.  We will write $\tau_i$ for $\tau_{a_i}$. Note that this acts by permuting the axes as a reflection $\tau_i \colon a_k \mapsto a_{2i-k}$. We say that a $2$-generated algebra $A$ is \emph{symmetric} if there exists an involutory automorphism $\phi$ which switches the generating axes.  We will often write $\half$ for $\phi$.

Rehren originally extended the definitions of the Norton-Sakuma algebras (the $2$-generated axial algebras of Monster type $(\frac{1}{4}, \frac{1}{32})$) to a general Monster fusion law \cite{gendihedral}.  These became infinite families of $\cM(\al,\bt)$ examples with respect to $\al$ and $\bt$.  Joshi then found two further infinite families \cite{JoshiMRes}.  The Highwater algebra is an infinite dimensional example which was found by Franchi, Mainardis and Shpectorov in \cite{highwater}.  In a significant breakthrough, Yabe classified the $2$-generated symmetric axial algebras of Monster type in most cases.  The remaining cases of characteristic $5$ were completed by Franchi and Mainardis, and the quotients of the Highwater algebra by Franchi, Mainardis and M\textsuperscript{c}Inroy.

\begin{theorem}[\cite{yabe, highwater5, HWquo}]\label{2gen sym}
A symmetric $2$-generated axial algebra of Monster type $(\alpha, \beta)$ is a quotient of one of the following:
\begin{enumerate}
\item an axial algebra of Jordan type $\alpha$, or $\beta$;
\item an algebra in one of the following families:

\begin{enumerate}
\item $3\A(\al,\bt)$, $4\A(\frac{1}{4}, \bt)$, $4\J(2\bt, \bt)$, $4\B(\al, \frac{\al^2}{2})$, $4\Y(\al, \frac{1-\al^2}{2})$, $4\Y(\frac{1}{2}, \bt)$, $5\A(\al, \frac{5\al-1}{8})$, $6\A(\al, \frac{-\al^2}{4(2\al-1)})$, $6\J(2\bt, \bt)$ and $6\Y(\frac{1}{2}, 2)$;
\item $\IY_3(\al, \frac{1}{2}, \mu)$ and $\IY_5(\al, \frac{1}{2})$;
\end{enumerate}

\item the Highwater algebra $\cH$, or its characteristic $5$ cover $\hatH$.
\end{enumerate}
\end{theorem}

Our notation for these is from \cite{forbidden}.  Each algebra is is $n\mathrm{L}(\al,\bt)$, where $n$ is the size of the axet, $\mathrm{L}$ is letter identifying the algebra and $\al$ and $\bt$ are the values in the fusion law $\cM(\al,\bt)$.  If the axet is infinite, we write $\mathcal{I}$ instead of $n$.  Note that several of the algebras do not exist for all values of $\al$ and $\bt$, but instead for a $1$-dimensional variety of values $(\al, \bt)$.

Note that we have algebras with an axet of size $3$, $4$, $5$, $6$, or (generically) infinity.  From \cite{forbidden}, only the axets of size $4$, $6$, or infinity may have axial subalgebras.

\subsection{Double axes}

Double axes were introduced by Joshi in \cite{JoshiMRes}.  Suppose that $\cF$ is a Seress fusion law and $a,b$ are two axes such that $ab = 0$.  Then, we call $x := a+b$ a \emph{double axis} and we may refer to $a$ and $b$ as \emph{single axes}.

Since $b \in A_0(a)$ and $\cF$ is Seress, we have $bA_\lm(a) \subseteq A_\lm(a)$ for all $\lm \in \cF$ and so $A_\lm(a)$ is invariant under multiplication by $b$.  So we can decompose $A_\lm(a) = \sum_{\mu \in \cF} A_{\lm,\mu}$, where $A_{\lm, \mu} := A_\lm(a) \cap A_\mu(b)$.  Thus $A =  \sum_{\lm,\mu \in \cF} A_{\lm,\mu}$.

\begin{lemma}\label{doubleespace}
Let $x = a+b$ be a double axis.  Then $A_\nu(x) := \sum_{\lm+\mu =\nu} A_{\lm,\mu}$.
\end{lemma}


For the $2$-generated axial algebras of Monster type, only $4\A(\frac{1}{4}, \bt)$ and $4\J(2\bt, \bt)$ have axes which are orthogonal.

\section{Theory and techniques}\label{sec:tech}

In this paper, we have two main goals, classifying all ideals and quotients, and finding all idempotents.

\subsection{Ideals}

The theory of ideals in axial algebras was developed in \cite{axialstructure}.  Since ideals are invariant under multiplication by axes, they decompose into the direct sum of eigenspaces.  So if $I \unlhd A$ and $a$ is an axis, then
\[
I = \bigoplus_{\lm \in \cF} I_\lm
\]
where $I_\lm := I \cap A_\lm(a)$.

\begin{proposition}\textup{\cite[Corollary 3.11]{axialstructure}}
Any ideal $I$ in an axial algebra is invariant under the action of the Miyamoto group $\Miy(A)$.
\end{proposition}

In \cite{axialstructure}, ideals are split into two types, those which contain axes, and those which do not.

\begin{proposition}\textup{\cite[Lemmas 4.14 and 4.17]{axialstructure}}
Let $A$ be a primitive axial algebra which has a Frobenius form.  Suppose that $I \unlhd A$ which contains an axis $a$.  If $b$ is an axis such that $(a,b) \neq 0$, then $b \in I$.
\end{proposition}

For our situation, this gives an easy corollary by combining the above two propositions.

\begin{corollary}\label{axisideal}
	Let $A = \lla a, b \rra$ be a $2$-generated primitive axial algebra with a $C_2$-graded fusion law. If either
	\begin{enumerate}
		\item there is one orbit of axes under $\Miy(A)$, or 
		
		\item there are two orbits of axes under $\Miy(A)$ and there exists $c \in a^{\Miy(A)}$ and $d \in b^{\Miy(A)}$ such that $(c, d) \neq 0$		
	\end{enumerate}
	then there are no proper ideals containing an axis from the generating set. \qed
\end{corollary}

We use this corollary to show that almost all algebras in this paper do not have proper ideals which contain axes.

We now turn to ideals which do not contain any axes.  The \emph{radical} $R(A,X)$ is the unique largest ideal that does not contain any axes in $X$.

\begin{theorem}\textup{\cite[Theorem 4.9]{axialstructure}}
Let $(A,X)$ be an axial algebra which admits a Frobenius form.  Then the radical $A^\perp$ of the Frobenius form equals the radical $R(A,X)$ of the algebra if and only if $(a,a) \neq 0$ for all $a \in X$.
\end{theorem}

Since all the $2$-generated algebras have a Frobenius form, we can use the above theorem to easily find the radical $R$.  We then make heavy use of the fact that ideals are $\Miy(A)$-invariant and representation theory to decompose $R$ into irreducible $\Miy(A)$-modules.  We now recall the following.

\begin{lemma}
Let $M$ be a $G$-module and $M = \bigoplus_{i \in I} M_i$ be an isotypic decomposition of $M$, where each $M_i$ is the direct sum of isomorphic irreducible modules.  If $N$ is a submodule of $M$, then $N$ decomposes as $N = \bigoplus_{i \in I} N_i$, where $N_i = M_i \cap N$.
\end{lemma}

This allows us to find all subideals of $R$.

We note that we must pay careful attention to the characteristic of the field (and other parameters) in two places particularly.  Firstly, in finding the radical.  Instead of just checking for when the determinant of the Gram matrix $M$ of the Frobenius form is zero, we instead find the characteristic polynomial $\chi(M) \in \FF[t]$.  We can then analyse when $\chi(M)$ has a factor $t^d$ and how the multiplicity $d$ is effected by the characteristic and values of other parameters such as $\al$ and $\bt$.  This give an upper bound for the dimension of the radical.  In almost all cases, this is enough to identify the radical.  However, occasionally, we do need to do Echelon row operations taking care not to divide by $p = \ch \FF$.  Secondly, when $\ch \FF$ divides the order of the Miyamoto group, the module structure of $R$ changes and we must use modular rather than ordinary representation theory.

\subsection{Idempotents}

In general, finding idempotents in an arbitrary algebra is hard.  One na\"ive approach and sometimes the only one is to use Gr\"obner bases.  Suppose that $A$ is an algebra over a field $\FF$ with basis $a_1, \dots, a_n$.  By extending the field to a polynomial ring $P := \FF[x_1, \dots, x_n]$, we can write a general element of the algebra as $x \in A$ as $x = \sum_{i=1}^n x_i a_i$, where the coefficients are the unknowns.  If $x^2 = x$, then we get $n$ quadratic equations $f_i$ in the $n$ unknowns $x_i$.  We can form the ideal $I = (f_1, \dots, f_n)$.  Any idempotents correspond to simultaneous solutions of the polynomials $f_i$; in other words, elements in the variety $\V(I)$.  The variety is finite (and so there are only finitely many idempotents) over an algebraically closed field if and only if the ideal is $0$-dimensional.  If $\dim I = 0$, we can use Magma's \texttt{Variety} function to find the variety.

A variety $V$ is the set of simultaneous solutions of a set of polynomial equations $f_1, \dots, f_n$.  So given a variety $V$, we can also form the ideal $\I(V) = (f_1, \dots, f_n)$.  For any variety $V$, $\V(\I(V)) = V$.  However, for an ideal $I$ over an algebraically closed field, $\I(\V(I)) = \sqrt{I}$, the radical of $I$.  Hence, if $I = \sqrt{I}$, then the functors $\V$ and $\I$ are inverses of one another.  A variety $V$ is called \emph{irreducible} if it cannot be written as the union $V_1 \cup V_2$ of two varieties $V_1$ and $V_2$.  A variety $V$ is irreducible if and only if the corresponding ideal $\I(V)$ is prime.  So a decomposition $V = V_1 \cup \dots \cup V_k$ into irreducible varieties corresponds to a radical decomposition $\I(V) = I =  I_1\cap \dots \cap I_k$ where $I_i$ are prime ideals corresponding to the irreducible varieties $V_i$.

So to find all idempotents, we can form the ideal $I$ coming from the relations given by $x^2 = x$.  It is more efficient and quicker computationally to use Magma's \texttt{RadicalDecomposition} function to first decompose $I = I_1 \cap \dots \cap I_k$ as the intersection of prime ideals $I$ and then to find the variety associated to each $I_i$ using $\texttt{Variety}$ than it is to try to find the variety $\V(I)$ directly.  Note, sometimes the ideal $I$ decomposes into some $0$-dimensional and some higher dimensional prime ideals.  In some cases, we are able to describe the idempotents corresponding to a higher-dimensional ideal.

Even doing this, for algebras over function fields $\FF(\al)$ for example, this is too slow for dimension $6$ and above.  We also have two tricks which we employ which speed things up considerably.

For the first, note that over a function field $\FF(\al)$, all our algebras have an identity.  In this case, if $x$ is an idempotent, then $\1-x$ is also an idempotent.  Set $y = x -\frac{1}{2} \1$.  Then $x$ is an idempotent if and only if $y^2 = ( x-\frac{1}{2} \1)^2 = x^2 - x + \frac{1}{4} \1 = \frac{1}{4} \1$.  So, we can instead search for $y$ such that $y^2 = \frac{1}{4}\1$.  This has the advantage that the polynomials are more symmetric and so the Gr\"obner basis calculation in finding the radical decomposition and variety is quicker.  More importantly, we can use the Frobenius form to get another relation; $(y,y) = (y^2, \1) = \frac{1}{4}(\1,\1)$.  Since the identity can be found easily (it is a linear requirement), this easily gives another relation for the ideal.

We list orbits of idempotents under the action of $G := \la \Miy(A), \half\ra \leq \Aut(A)$.  If $u$ is an idempotent in an orbit of size less than $|G|$, then it has a non-trivial stabiliser $G_u$.  If $g \in G_u$, then $ug = u$.  But as $g$ acts non-trivially on $A$, this gives linear relations on the coefficients of $u$ and so greatly speeds up the Gr\"obner basis calculation.  So we proceed by finding the non-trivial conjugacy classes in $G$ and for each representative $g$, forming the ideal $I$ including these linear relations.  This allows us to find all idempotents not in orbits of size $|G|$.

For the largest two algebras, $6\A(\al, \frac{-\al^2}{4(2\al-1)})$ and $6\J(2\bt, \bt)$, using the group action trick is the only way finding idempotents is possible and so in these cases, we can only find idempotents with orbit size not equal to $12 = |G|$.  There are $208$ of these.  However, as Magma's \texttt{IsZeroDimensional} does complete and show the ideal is $0$-dimensional, by B\'ezout's Theorem, over an algebraically closed field, there are $2^8 = 256$ idempotents.  Hence we know that there are $4 \times 12 = 256-208$ idempotents in $4$ orbits of size twelve which we could not find.

For some of the idempotents, we use the \texttt{AxialTools} package \cite{AxialTools} to find various properties including the fusion law and grading.  In addition to finding some idempotents which have the Jordan type fusion law, we find some with what we call the \emph{almost Monster} fusion law $\cAM(\al, \bt)$ given in Table \ref{tab:Almost Monster}.  This is also $C_2$-graded and differs only in the $\al \star \al$ part.

\begin{figure}[!ht]
\begin{center}
{\renewcommand{\arraystretch}{1.5}
\begin{tabular}[t]{c||c|c|c|c}
$\star$ & $1$ & $0$ & $\alpha$ & $\beta$ \\
\hline\hline
$1$ & $1$ & & $\alpha$ & $\beta$ \\
\hline
$0$ & & $0$ & $\alpha$ & $\beta$ \\
\hline 
$\alpha$ & $\alpha$ & $\alpha$ & $1,0, \al$ & $\beta$ \\
\hline 
$\beta$ & $\beta$ & $\beta$ & $\beta$ & $1,0,\alpha$
\end{tabular}
}
\end{center}
\caption{The Almost Monster fusion law $\cAM(\al,\bt)$.}
\label{tab:Almost Monster}
\end{figure}


\section{Algebras with axet $X(3)$}

There is only one algebra of Monster type with axet $X(3)$ which is not also of Jordan type and this is $3\A(\al, \bt)$.

\begin{table}[h!tb]
\setlength{\tabcolsep}{4pt}
\renewcommand{\arraystretch}{1.5}
\centering
\begin{tabular}{c|c|c}
Type & Basis & Products and Form \\ \hline
$3\A(\al,\bt)$ & \begin{tabular}[t]{c} $a_{-1}$, $a_0$, \\ $a_1$, $z$ \end{tabular} &
\begin{tabular}[t]{c}
$a_0 \cdot a_1 = \frac{\al+\bt}{2}(a_0 + a_1) + \frac{\al-\bt}{2}a_{-1} + z$ \\
$a_0 \cdot z = \frac{-\al(3\al^2 + 3\al\bt - \bt - 1)}{4(2\al-1)}a_0$, $z^2 = \frac{-\al(3\al^2 + 3\al\bt - \bt - 1)}{4(2\al-1)} z$ \\
$(a_0, a_1) = \frac{3\al^2 + 3\al\bt - \al - \bt}{4(2\al-1)}$, $(a_0, z) = \frac{-\al(3\al^2 + 3\al\bt - \bt - 1)}{4(2\al-1)}$\\
$(z,z) = \frac{\al^2(9\al+\bt-5)(3\al^2 + 3\al\bt -\bt-1)}{16(2\al-1)^2}$
\vspace{4pt}
\end{tabular}
\end{tabular}
\caption{$2$-generated $\cM(\alpha, \beta)$-axial algebras on $X(3)$}\label{tab:2genMonsterX3}
\end{table}

\begin{proposition}
If $\ch \FF \neq 3$, then, $3\A(\al, \frac{1}{2}) \cong \IY_3(\al, \frac{1}{2}, -\frac{1}{2})$.
\end{proposition}
\begin{proof}
Let $e = \frac{2}{3}(2a_0-a_1-a_{-1})$, $f = \frac{2}{3}(2a_1-a_0-a_{-1})$, $z_1 := -\frac{2}{3}(a_0 + a_1 + a_{-1} + \frac{4}{\al}z)$.  If $\al \neq -1$, set $z_2 := \frac{2}{3}(a_0 + a_1 + a_{-1} + \frac{4}{\al+1}z)$ and if $\al = -1$, then set $n := \frac{8}{3}z$.  Then, one can check that these multiply according to the definition of $\IY_3(\al, \frac{1}{2}, -\frac{1}{2})$ in \cite{splitspin}.
\end{proof}

The computations in this section can be found in \cite[\texttt{X3 algebras.m}]{githubcode}.

\subsection{$3\A(\al,\bt)$}

For $3\A(\al, \bt)$, in addition to requiring that $\{1,0,\al,\bt\}$ do not coincide, we also need an additional restriction that $\al \neq \frac{1}{2}$ for the multiplication to be defined.

It is clear from the multiplication in Table \ref{tab:2genMonsterX3}, than $\ad_z$ acts by multiplication on the algebra by $\frac{-\al(3\al^2 + 3\al\bt - \bt - 1)}{4(2\al-1)}$.  So, the algebra has an identity if $3\al^2 + 3\al\bt - \bt - 1 \neq 0$ and we have $\1 = \frac{-4(2\al -1)}{\al(3\al^2 + 3\al\bt - \bt - 1)} z$.  When $3\al^2 + 3\al\bt - \bt - 1 = 0$, the algebra has no identity and we see that $z$ is a nilpotent element contained in the annihilator $\Ann A$ of the algebra.  In fact, as we will see later, in this case $\Ann A = \la z \ra$. Otherwise, $\Ann A = 0$.

\begin{lemma}
For the axis $a_0 \in 3\A(\al,\bt)$, we have
\begin{align*}
A_1(a_0) &= \la a_0 \ra \\	
A_0(a_0) &= \la (3\al^3 + 3\al^2  \bt - \al \bt - \al)a_0 + 4(2\al - 1) z \ra \\
A_\al(a_0) &= \la \al(\al +\bt -1)a_0 + 2\al(2\al  -1) (a_1+a_{-1}) + 4(2\al - 1) z \ra \\
A_\bt(a_0) &= \la a_1 - a_{-1} \ra 
\end{align*}
\end{lemma}

\begin{lemma}
There are no ideals of $A := 3\A(\al,\bt)$ which contain axes.
\end{lemma}
\begin{proof}
Since the axes in $A$ are in one orbit under the action of $\Miy(A)$, by Corollary \ref{axisideal}, $3\A(\al,\bt)$ has no proper ideals which contain axes.
\end{proof}

\begin{computation}\label{3Adet}
The determinant of the Frobenius form is
\[
-\frac{\al^2  (3\al - \bt - 1)  (3\al^2 + 3\al\bt - \bt - 1) (3\al^2 + 3\al\bt - 9\al - 2\bt + 4)^3}{2^9  (2\al - 1)^5}
\]
\end{computation}

So the algebra has a non-trivial radical when one of $\bt = 3\al - 1$ , $3\al^2 + 3\al\bt - \bt - 1 = 0$, or $3\al^2 + 3\al\bt - 9\al - 2\bt + 4 = 0$.  We need to examine each of these cases in turn, find the radical and then investigate whether there are any additional ideals that are proper subspaces of the radical.  In order to get a clean case distinction, we first see when more than one of these conditions can be satisfied at once.

\begin{lemma}\label{3Acases}
No two of the conditions $\bt = 3\al - 1$ , $3\al^2 + 3\al\bt - \bt - 1 = 0$, and $3\al^2 + 3\al\bt - 9\al - 2\bt + 4 = 0$ are simultaneously satisfied unless $\ch \FF = 3$ and $\bt = -1$.  In this case all three are simultaneously satisfied.
\end{lemma}
\begin{proof}
We just do the most involved case; the other two are similar.  Suppose that both $3\al^2 + 3\al\bt - \bt - 1 = 0$ and $3\al^2 + 3\al\bt - 9\al - 2\bt + 4 = 0$.  Then subtracting, we get $9\al + \bt - 5 = 0$ and so $\bt = 5 - 9\al$.  Substituting this back into the first equation we obtain
\[
0 = 3\al^2 + 3\al(5-9\al) -(5-9\al) -1 = -24\al^2 + 24\al - 6 = -6(2\al-1)^2
\]
and so either $\al = \frac{1}{2}$, a contradiction, or $\ch \FF = 3$.  Now from $\bt = 5-9\al$, we obtain $\bt = -1$.  We obtain the same conditions for any two pairs of equations being simultaneously satisfied.
\end{proof}

We now find the radical and any further ideals in each case.  In order to find the radical, we calculate the nullspace of Frobenius form using a computer.  Full details of the computations for this case can be found in \cite[X3.m]{githubcode}.

\begin{proposition}\label{3Aclass}
The ideals and quotients of $3\A(\al, \bt)$ are given in Table $\ref{tab:3Aprops}$.  In particular, when $\ch \FF = 3$ and $\bt = -1$, all four cases occur for $3\A(\al, -1)$.
\end{proposition}
\begin{proof}
We first assume that $\ch \FF \neq 3$.  Then, by Lemma \ref{3Acases}, each of the cases are distinct.  When $3\al^2 + 3\al\bt - \bt - 1 = 0$, the radical is $1$-dimensional and as discussed above, contains the nilpotent element $z$ which annihilates the algebra.  So, we see that the radical is the annililator and is spanned by $z$.  Hence, $z$ is the common $0$-eigenvector for all the axes and so the quotient has Monster type with no $0$-part; this is $3\A(\al, \tfrac{1-3\al^2}{3\al-1})^\times$.

If $\bt = 3\al -1$, then the radical is $1$-dimensional and is spanned by a common $\al$-eigenvector.  Hence the quotient of $3\A(\al, 3\al -1)$ is $3\C(3\al-1)$.

If $3\al^2 + 3\al\bt - 9\al - 2\bt + 4 = 0$,  then note that $\alpha \neq \frac{2}{3}$, and so we can rearrange to get $\bt = \frac{3\al^2-9\al+4}{2-3\al}$.  By computation, the radical is $3$-dimensional and is spanned by the elements in Table \ref{tab:3Aprops}.  Since ideals of axial algebras are invariant under the action of the Miyamoto group, which here is $\Miy(A) \cong S_3$, the radical is the sum of a $1$-dimensional trivial module spanned by the sum of axes and a scalar multiple of $z$, and a $2$-dimensional module spanned by the difference of axes.  So these are the only candidates for ideals which are proper subideals of the radical.  Easy computations show that in fact neither are.  Hence, the radical is the only ideal in this case.

Finally, assume that $\ch \FF = 3$.  By Computation \ref{3Adet}, the determinant of the Frobenius form is zero if and only if $\bt = -1$.  In this case, $(a_i, a_j) = 1$ and $(a_i,z) = 0 = (z,z)$ and so the algebra is baric and the radical is spanned by $a_0-a_1, a_1-a_{-1}, z$.  This decomposes as the direct sum of two indecomposable modules $\la a_0 - a_1, a_1 - a_{-1} \ra$ and $\la z \ra$.  However, as $\ch \FF = 3$ divides $|S_3| = 6$, the $2$-dimensional module is no longer irreducible and sum of the axes is now a $1$-dimensional trivial submodule.

The $2$-dimensional module is not an ideal as $a_0(a_0-a_1) = \frac{1}{2}\big( (1-\al) (a_0 + a_{-1}) - (\al+1) a_1\big) - z \notin \la a_0-a_1, a_1-a_{-1}\ra$.  The only other candidates for a proper subideal of the radical are subalgebras of the sum of the two trivial modules $\la a_0+a_1+a_{-1}, z \ra$.  By Lemma \ref{3Acases}, all three conditions are simultaneously satisfied.  So, the axes all have a common $0$-eigenspace, which is also $\Ann A = \la z \ra$, and a common $\al$-eigenspace spanned by $u_\al := \al(a_0+a_1+a_{-1}) + 2z$. Since an ideal must also decompose into eigenspaces, these are the only two $1$-dimensional ideals.  Since $3\A(\al, -1)/\la u_\al\ra \cong 3\C(-1)$, this has a $1$-dimensional ideal with quotient $3\C(-1)^\times$ and hence $\la a_0+a_1+a{-1}, z \ra$ is also an ideal.
\end{proof}

\begin{table}[h!tb]
\renewcommand{\arraystretch}{1.5}
\centering
\footnotesize
\begin{tabular}{c|c|c|c}
\hline
Condition & Ideals & Quotients & Dimension \\
\hline
 $\bt = \tfrac{1-3\al^2}{3\al-1}$ & $\la z \ra$  & $3\A(\al, \tfrac{1-3\al^2}{3\al-1})^\times$ & 3 \\
$\bt = 3\al-1$ & $\la \al(a_0+a_1+a_{-1}) +2z \ra$ & $3\C(3\al-1)$ & 3\\
$\bt = \tfrac{3\al^2-9\al+4}{2-3\al}$ & $\la a_i + \tfrac{2(3\al - 2)}{3\al(2\al-1)}z \ra$ & $1\A$ & 1 \\
$\ch \FF = 3$, $\bt=-1$ & $\la a_0+a_1+a_{-1}, z \ra$ &$ 3\C(-1)^\times$ & 2
\end{tabular}
\caption{Ideals and quotients of $3\A(\al, \bt)$}\label{tab:3Aprops}
\end{table}

We now turn to examine the idempotents in $3\A(\al,\bt)$ over an algebraically closed field.  If $3\al-\bt-1$ and $2\al\bt-\al-\bt-1$ are both non-zero, then let $\rt \in \FF$ be such that
\[
\rt^2 = -\frac{2\al-1}{(3\al-\bt-1)(2\al\bt-\al-\bt-1)}
\]
and set 
\begin{align*}
	x &:=  \rt(a_0+a_1+a_{-1}) + \frac{2(2\al - 1)(3\al + \bt + 1)}{\al(3\al^2 + 3\al\bt - \bt - 1)} \rt z\\
	y &:=  \frac{\bt-\al}{2\al-1}\rt a_0 + \rt(a_1+a_{-1}) + \frac{2(4\al\bt + \al - \bt - 1)}{\al(3\al^2 + 3\al\bt - \bt - 1)} \rt z
\end{align*}

\begin{table}[h!tb]
\renewcommand{\arraystretch}{1.5} 
\begin{tabular}{cccc}
\hline
Representative & orbit size & length & comment \\
\hline

$\1$ & $1$ &  $\tfrac{9\al + \bt - 5}{3\al^2 + 3\al\bt -\bt - 1}$ & $3\al^2 + 3\al\bt -\bt - 1 \neq 0$\\
$\frac{1}{2}\1\pm x$ & $2$ &  $\tfrac{\pm \left( (9\al + \bt - 5) - (3\al - \bt - 1)^2\right)  \rt}{2(3\al^2 + 3\al\bt - \bt - 1)}$  &
{\renewcommand{\arraystretch}{1} 
 \begin{tabular}[t]{c}
	$3$ eigenspaces,\\ not graded in general, \\ when $\bt = \tfrac{1}{2}$, $x$ obeys $\cJ(\tfrac{1}{2})$
\end{tabular} }\\
$a_0$ & $3$ & $1$ & $\cM(\al, \bt)$-axes\\
$\1-a_0$ & $3$ & $\tfrac{-3\al^2 - 3\al\bt+9\al+2\bt - 4}{3\al^2 + 3\al\bt -\bt - 1}$  & {\renewcommand{\arraystretch}{1} 
	\begin{tabular}[t]{c}
		 $\cM(1-\al,1- \bt)$-axes, \\ not primitive
\end{tabular} } \\
$\frac{1}{2}\1\pm y$ & $2\times3$ & $\tfrac{\pm\left( (9\al + \bt - 5) - (3\al - \bt - 1)^2\right)  \rt}{2(3\al^2 + 3\al\bt - \bt - 1)}$ & {\renewcommand{\arraystretch}{1} %
\begin{tabular}[t]{c}
	$\cAM(\lm,\mu )$ in general, \\ when $\bt = \tfrac{1}{2}$, $y$ obeys $\cM(\lm, \mu)$\\
	for $\lm = \tfrac{1}{2} \pm \tfrac{3\al-\bt-1}{2}\rt$,\\ and $\mu = \tfrac{1}{2} \mp\tfrac{(2\bt-1)(3\al-\bt-1)}{4\al-2}\rt$
\end{tabular}
}
\end{tabular}
\caption{Idempotents of $3\A(\al, \bt)$}\label{tab:3Aidempotents}
\end{table}

\begin{computation}
Suppose $\FF$ is a field such that $\bt \neq \frac{1}{2}$, $3\al-\bt-1$ and $2\al\bt-\al-\bt-1$ are both non-zero and let $\rt \in \FF$ be defined as above.  Then, $3\A(\al,\bt)$ has $15$ non-trivial idempotents as given in Table $\ref{tab:3Aidempotents}$.
\end{computation}


\section{Algebras with axet $X(4)$}

For the five $2$-generated algebras of Monster type with axet $X(4)$, we propose new bases for three of them which better exhibit their axial properties.   To describe the rationale for these, we first find the axial subalgebras.

For an algebra $A$ with axet $X(4)$, there are two orbits of axes under the action of the Miyamoto group: $\{ a_0, a_2 \}$ and $\{ a_1, a_{-1} \}$.  These generate $2$-generated axial subalgebras $\lla a_0, a_2 \rra$ and $\lla a_1, a_{-1} \rra$, which are necessarily isomorphic since they are switched by the automorphism $\half$.  By computation, we have the following.

\begin{proposition}
The axial subalgebras of the algebras $4\A(\frac{1}{4}, \bt)$, $4\J(2\bt, \bt)$, $4\B(\al, \frac{\al^2}{2})$, $4\Y(\al, \frac{1-\al^2}{2})$ and $4\Y(\frac{1}{2}, \bt)$ are given in Table $\ref{tab:X4subalg}$.
\begin{table}[h!tb]
\centering
\begin{tabular}{ccc}
\hline
Algebra & $\lla a_0, a_2 \rra$ & conditions \\
\hline
$4\A(\frac{1}{4}, \bt)$ & $2\B$ & \\
$4\J(2\bt, \bt)$  & $2\B$ & \\
$4\B(\al, \frac{\al^2}{2})$ & $3\C(\al)$ & \\
$4\Y(\al, \frac{1-\al^2}{2})$ & $3\C(\al)$ & \\
$4\Y(\frac{1}{2}, \bt)$ & $S(\dl)$ & $\bt \neq \frac{1}{4}$ \\
$4\Y(\frac{1}{2}, \bt)$ & $3\C(\frac{1}{2})$ & $\bt = \frac{1}{8}, \frac{3}{8}$ \\
$4\Y(\frac{1}{2}, \bt)$ & $2\B$ & $\bt = \frac{1}{4}$ \\
\hline
\end{tabular}
\caption{Subalgebras in algebras with axet $X(4)$}\label{tab:X4subalg}
\end{table}
\end{proposition}
\begin{proof}
These follow by computation from the structure constants and using Remark \ref{idJordan}.
\end{proof}

Since the axial subalgebra $\lla a_0, a_2 \rra$ is an axial algebra of Jordan type $\al$, it is normalised by the Miyamoto group.  Moreover, when this subalgebra is $3\C(\al)$, or $S(\dl)$, the Miyamoto group acts by permuting the generating axes and fixing a $1$-dimensional subspace.  This subspace is fixed by $\half$ and so $\lla a_0, a_2 \rra \cap \lla a_1, a_{-1} \rra$ is $1$-dimensional.  So in these cases, we can chose the fifth basis vector to be in this $1$-dimensional subspace.  Moreover, it turns out that we can scale these to be an idempotent.  So, for $4\B(\al, \frac{\al^2}{2})$, the fifth basis vector is is the third $\cJ(\al)$-type axis in the $3\C(\al)$ axial subalgebra.  For $4\Y(\al, \frac{1-\al^2}{2})$, it is $\1_{\lla a_0, a_2 \rra} - c'$, where $c'$ is the third $\cJ(\al)$-type axis in the $3\C(\al)$ axial subalgebra (see Theorem \ref{4B4Yiso} for details).  For $4\Y(\frac{1}{2}, \bt)$, it is an idempotent in the axial subalgebra which is generically isomorphic to $S(\dl)$.  For $4\J(2\bt, \bt)$, we use the same basis as in \cite{doubleMatsuo} and we do not change the basis for $4\A(\frac{1}{4}, \bt)$.

The computations in this section can be found in \cite[\texttt{X4 algebras.m}]{githubcode}.

\begin{table}[h!tb]
\setlength{\tabcolsep}{4pt}
\renewcommand{\arraystretch}{1.5}
\centering
\footnotesize
\begin{tabular}{c|c|c}
Type & Basis & Products \& form \\ \hline
$4\A(\frac{1}{4}, \bt)$ & \begin{tabular}[t]{c} $a_{-1}$, $a_0$, \\ $a_1$, $a_2$, $e$ \end{tabular} &
\begin{tabular}[t]{c}
$a_0 \cdot a_1 = \frac{1+4\bt}{8}(a_0 + a_1) + \frac{1-4\bt}{8}(a_2+a_{-1}) + e$ \\
$a_0 \cdot a_2 = 0$, $a_i \cdot e = \frac{2\bt-1}{8}a_i$, $e^2 = \frac{2\bt-1}{8}e$ \\
$(a_0, a_1) = \beta$, $(a_0, a_2) = 0$\\
$(a_0, e) = \frac{2\bt-1}{8}$, $(e, e) = \frac{(2\bt-1)^2}{16}$
\vspace{4pt}
\end{tabular}\\
\hline
$4\J(2\bt, \bt)$ & \begin{tabular}[t]{c} $a_{-1}$, $a_0$, \\ $a_1$, $a_2$, $w$ \end{tabular} &
\begin{tabular}[t]{c}
$a_0 \cdot a_1 = \frac{\bt}{2}(2a_0 + 2a_1 -w)$, $a_0 \cdot a_2 = 0$\\
 $a_i \cdot w = \bt(2a_i - (a_{i-1}+a_{i+1}) + w)$, $w^2 = w$ \\
$(a_0, a_1) = \beta$, $(a_0, a_2) = 0$\\
$(a_0, w) = 2\bt$, $(w, w) = 2$
\vspace{4pt}
\end{tabular}\\
\hline
$4\B(\al, \frac{\al^2}{2})$ & \begin{tabular}[t]{c} $a_{-1}$, $a_0$, \\ $a_1$, $a_2$, $c$ \end{tabular} &
\begin{tabular}[t]{c}
$a_0 \cdot a_1 = \frac{\bt}{2}(a_0+a_1 - a_2 - a_{-1} +c)$ \\
$a_0 \cdot a_2 = \frac{\al}{2}(a_0+a_2 -c)$ \\
$a_i \cdot c =  \frac{\al}{2}(a_i-a_{i+2} +c)$, $c^2 = c$ \\
$(a_0, a_1) = \frac{\bt}{2}$, $(a_0, a_2) = \frac{\al}{2}= (a_0, c)$, $(c,c) = 1$
\vspace{4pt}
\end{tabular}\\
\hline
$4\Y(\al, \frac{1-\al^2}{2})$ & \begin{tabular}[t]{c} $a_{-1}$, $a_0$, \\ $a_1$, $a_2$, $c$ \end{tabular} &
\begin{tabular}[t]{c}
$a_0 \cdot a_1 = \frac{\bt}{2}(a_0+a_1 - a_2 - a_{-1}) + \frac{(\al+1)^2}{4}c$ \\
$a_0 \cdot a_2 = \frac{\al-1}{2}(a_0+a_2) + \frac{\al+1}{2}c$ \\
$a_i \cdot c = \frac{\al-1}{2}(a_{i+2}-a_i) + \frac{\al+1}{2}c$, $c^2 = c$ \\
$(a_0, a_1) = \frac{(2-\al)(\al+1)}{4}$, $(a_0, a_2) = \frac{\al}{2}$ \\
$(a_0, c) = \frac{2-\al}{2}$, $(c,c) = \frac{2-\al}{\al+1}$
\vspace{4pt}
\end{tabular}\\
\hline
$4\Y(\frac{1}{2}, \bt)$ & \begin{tabular}[t]{c} $a_{-1}$, $a_0$, \\ $a_1$, $a_2$, $z$ \end{tabular} &
\begin{tabular}[t]{c}
$a_0 \cdot a_1 = \frac{\bt}{2}(a_0+a_1 - a_2 - a_{-1} + 8\bt z)$ \\
$a_0 \cdot a_2 = \frac{1-4\bt}{2}(a_0+a_2 -8\bt z)$ \\
$a_i \cdot z = \frac{1}{4}(a_i -a_{i+2}) + 2\bt z$, $z^2 = z$ \\
$(a_0, a_1) = 4\bt^2$, $(a_0, a_2) = (4\bt-1)^2$ \\
$(a_0, z) = 2\bt$, $(z,z) = 1$
\vspace{4pt}
\end{tabular}\\
\end{tabular}\\
\caption{$2$-generated $\cM(\alpha, \beta)$-axial algebras on $X(4)$}\label{tab:2genMonsterX4}
\end{table}

\subsection{$4\A(\tfrac{1}{4}, \bt)$}

In $4\A(\tfrac{1}{4}, \bt)$, for $\{1,0,\al,\bt\}$ not to coincide, we require that $\ch \FF \neq 3$ and $\bt \neq 0,1, \frac{1}{4}$.

From the multiplication in Table \ref{tab:2genMonsterX4}, it is clear that $ad_e$ acts by multiplication on the algebra by $\tfrac{2\bt-1}{8}$. So, the algebra has an identity if $\bt \neq \tfrac{1}{2}$ and we have $ \1 = \tfrac{8}{2\bt -1}e$. When $\bt = \tfrac{1}{2}$, the algebra has no identity and we see that $e$ is a nilpotent element contained in the annihilator $\Ann A = \la e \ra$. For any other $\bt$, the annihilator is trivial.  It is also clear from the multiplication that $\lla a_0, a_2 \rra \cong 2\B$.

\begin{lemma}\label{4Aparts}
	For the axis $a_0 \in 4\A(\frac{1}{4}, \bt)$, we have
	\begin{align*}
		A_1(a_0) &= \la a_0 \ra \\	
		A_0(a_0) &= \la (2\bt-1)a_0 - 8e, a_2 \ra \\
		A_{\frac{1}{4}}(a_0) &= \la (4\bt-1)(a_0 + a_2) - (a_1 + a_{-1}) - 8e \ra \\
		A_\bt(a_0) &= \la a_1 - a_{-1} \ra 
	\end{align*}
\end{lemma}

\begin{lemma}
	There are no ideals of $A := 4\A(\frac{1}{4}, \bt)$ which contain axes.
\end{lemma}
\begin{proof}
	Since the axes in $A$ are in two orbits $\{a_0,a_2\}$ and $\{a_1,a_{-1}\}$ under the action of $\Miy(A)$, and $(a_0, a_1) =  \bt \neq 0$, by Corollary \ref{axisideal}, $4\A(\frac{1}{4}, \bt)$ has no proper ideals that contain axes.
\end{proof}

\begin{computation}
	The determinant of the Frobenius form is
	
	\[
	-\tfrac{1}{8}  \beta  (2\beta - 1)^3
	\]
\end{computation}

So, the algebra has a non-trivial radical when $ \beta = \frac{1}{2} $. Let us examine this case.

\begin{proposition}
	The ideals and quotients of $4\A(\frac{1}{4}, \bt)$ are given in Table $\ref{tab:4Aprops}$.
\end{proposition}

\begin{proof}
When $\bt = \frac{1}{2}$, the radical is $R := \la (a_0 + a_2) - (a_1 + a_{-1}),  e\ra$, which is $2$-dimensional.  It is easy to see that $e$ annihilates the algebra and, by computation, it spans $\Ann A$, which is an ideal.  By Lemma \ref{4Aparts}, $A := 4\A(\frac{1}{4},\frac{1}{2})$ has a common $\frac{1}{4}$-eigenspace, $A_{\frac{1}{4}} = \la (a_0 + a_2) - (a_1 + a_{-1}) -8e \ra$ and so $R$ decomposes as the sum of $A_{\frac{1}{4}}$ and $\Ann A$.  Since an ideal of $A$ must decompose into eigenspaces with respect to every axis, the only other possible subideal of $R$ is the common $\al = \frac{1}{4}$ part.  However, a quotient by $A_{\frac{1}{4}}$ would give a $2$-generated axial algebra of Jordan type $\frac{1}{2}$ of dimension $4$, a contradiction.  Hence the only ideals are $\Ann A$ and $R$.

In $A/\Ann A$, the image of the axes still have non-trivial $1$-, $0$-, $\frac{1}{4}$-, and $\frac{1}{2}$-parts which satisfy the Monster fusion law, hence the quotient is of Monster type and it is called $4\A(\frac{1}{4}, \frac{1}{2})^\times$.

Finally, let $B := A/R$.  Similarly to above, $B$ is a $2$-generated axial algebra of Jordan type $\frac{1}{2}$, that has dimension $3$.  By computation, $B$ has an identity, which is the image $\bar{a}_0 + \bar{a}_2$ of $a_0 + a_2$ in $B$.  From Table \ref{tab:2gen Jordan Table}, we can see that $B$ is either $3\C(\frac{1}{2})$ or $S(\delta)$.  Since $ \bar{a}_0 \cdot \bar{a}_1 - \frac{\bar{a}_0 + \bar{a}_1}{2} = -\frac{1}{4}\1_B$, by Remark \ref{idJordan}, we see that $\delta = 0$ and $B  \cong S(0)$.
\end{proof}

\begin{table}[h!tb]
	\renewcommand{\arraystretch}{1.5}
	\centering
	\footnotesize
	\begin{tabular}{c|c|c|c}
		\hline
		Condition & Ideals & Quotients & Dimension \\
		\hline
		 $\bt = \frac{1}{2}$ & $\la a_0 - a_1 + a_2 - a_{-1},  e\ra$  & $S(0)$ & 3 \\
		 $\bt = \frac{1}{2}$ & $\la e \ra$  & $4\A(\frac{1}{4}, \frac{1}{2})^\times
		 $ & 4\\
	\end{tabular}
	\caption{Ideals and quotients of $4\A(\frac{1}{4}, \bt)$}\label{tab:4Aprops}
\end{table} 

Since $\lla a_0, a_2 \rra \cong 2\B$, we can consider double axes for $A = 4\A(\frac{1}{4}, \bt)$.

\begin{lemma}\label{4Adouble}
In $A = 4\A(\frac{1}{4}, \bt)$, $d_0 := a_0+a_2$ and $d_1 := a_1+a_{-1}$ are double axes with
\begin{align*}
A_1(d_0) &= \la a_0, a_2 \ra \\
A_0(d_0) &= \la (2\bt-1)d_0 -8e \ra \\
A_{\frac{1}{2}}(d_0) &= \la (1-4\bt)d_0 + d_1 + 8e \ra \\
A_{2\bt}(d_0) &= \la a_1 - a_{-1} \ra
\end{align*}
If $\bt \neq \frac{1}{2}$, then they are of Monster type $(\frac{1}{2}, 2\beta)$.
\end{lemma}
\begin{proof}
First note that $a_0$ and $a_2$ are orthogonal axes, so we may form the double axis $d_0 = a_0+a_2$.  By Lemma \ref{doubleespace}, for the double axis $d_0$, $A_\nu(d_0) = \bigoplus_{\lambda + \mu = \nu} A_\lambda(a_0) \cap A_\mu(a_2)$.  Observe that the $\frac{1}{4}$- and $\bt$-eigenspaces in Lemma \ref{4Aparts} are invariant under $\tau_1$, which switches $a_0$ and $a_2$.  Hence, $a_0$ and $a_2$ share common $\frac{1}{4}$- and $\bt$-eigenspaces.  So, for $d_0$, these are the $\frac{1}{2}$- and $2\bt$-eigenspaces.  By observation, we see that $A_1(d_0) \supseteq \la a_0, a_2 \ra$ and $A_0(d_0) \supseteq \la (2\bt-1)d_0 -8e \ra$.  Counting the dimensions, we see that these are all the eigenspaces for $d_0$.  Note that for $\bt = \frac{1}{2}$, $2\bt = 1$ and so $A_1(d_0)$ and $A_{2\bt}(d_0)$ merge to become a $3$-dimensional $1$-eigenspace.

Suppose now that $\bt \neq \frac{1}{2}$.  Since the $\frac{1}{4}$- and $\bt$-eigenspaces are common between $a_0$ and $a_2$ and $A_{\{1,0\}}(a_0) = A_{\{1,0\}}(a_2) = A_{\{1,0\}}(d_0)$, for $d_0$ we have $\frac{1}{2} \star \frac{1}{2} \subseteq \{0,1\}$, $2\bt \star 2\bt \subseteq \{0,1,\frac{1}{2}\}$ and $\frac{1}{2} \star 2\bt = 2\bt$.  Observe that $A_0(d_0) = A_0(a_0) \cap A_0(a_2)$.  So, in the fusion law for $d_0$, we have $1 \star \lambda = \lambda$ and $0 \star \lambda = \lambda$, for $\lambda = 0, \frac{1}{2}, 2\bt$.  Hence $d_0$ has the  $\cM(\frac{1}{2}, 2\bt)$ fusion law and by symmetry so does $d_1$.
\end{proof}

\begin{proposition}\label{4Adoublesub}
The subalgebra $B := \lla d_0, d_1 \rra = \la d_0,d_1,e \ra$ is an axial algebra of Jordan type $\frac{1}{2}$
\[
\lla d_0, d_1 \rra \cong \begin{cases}
S(\delta), \mbox{where } \delta = 2(4\beta-1) & \mbox{if } \bt \neq \frac{1}{2} \\
\Cl & \mbox{if } \bt = \frac{1}{2}
\end{cases}
\]
Moreover, it is the fixed subalgebra under the action of the Miyamoto group of $4\A(\frac{1}{4}, \bt)$.
\end{proposition}
\begin{proof}
We have that $\tau_0 = \tau_2$ fixes $a_0$, $a_2$ and $e$ and swaps $a_1$ and $a_{-1}$, and $\tau_1 = \tau_{-1}$ fixes $a_1$, $a_{-1}$ and $e$ and swaps $a_0$ and $a_0$.  Thus we see that the fixed subalgebra of $A$ is $\la d_0,d_1,e\ra$.

We now calculate the product $d_0d_1$.  Observe that, by Lemma \ref{4Adouble}, we may decompose $d_1 = (d_1)_1 + (d_1)_0 + (d_1)_{\frac{1}{2}}$ with respect to $d_0$ as
\[
d_1 = 2\bt d_0 + \big((2\bt-1)d_0 -8e\big) + \big( (1-4\bt)d_0 + d_1 + 8e \big)
\]
So $d_0d_1 = 2\bt d_0 + \frac{1}{2}\big( (1-4\bt)d_0 + d_1 + 8e \big) = \frac{1}{2}(d_0+d_1) + 4e$.  Hence $\la d_0,d_1,e \ra \subseteq \lla d_0, d_1 \rra$ and since the subalgebra fixed under the Miyamoto group is $3$-dimensional, $\lla d_0,d_1 \rra = \la d_0,d_1,e \ra$.

If $\beta \neq \frac{1}{2}$, then $\1 := \frac{8}{2\bt -1}e$ is the identity for $4\A(\al, \bt)$ and so also for $\lla d_0, d_1 \rra$.  Hence, $d_0d_1 = \frac{1}{2}(d_0+d_1) + \frac{2\bt -1}{2}\1$.  Thus we see that $B \cong S(\delta)$, where $\delta := 2(4\beta-1)$.  If $\bt =\frac{1}{2}$, then $e \in \Ann A$ and $B \cong \Cl$.
\end{proof}

We now turn to examine the idempotents in $4\A(\frac{1}{4}, \bt)$.  Suppose $\bt \neq \frac{1}{2}, -\frac{3}{2}$ and that $\FF$ contains roots $\rt_1$, $\rt_2$, $\rt_3$ such that
\[
\rt_1^2 = -\bt(\bt-\tfrac{1}{2}), \qquad (\bt+ \tfrac{3}{2})\rt_2^2 = \bt, \qquad \rt_3^2 = \tfrac{3}{2}
\]
Define
\begin{align*}
	v_1 &:=  \tfrac{1}{\rt_1}\left ( (\tfrac{1}{4}-\bt)(a_0+a_2) + \tfrac{1}{4}(a_1+a_{-1}) + 2 e \right) \\
	v_2 &:= \tfrac{1}{4\bt}\rt_2 \sum_{i=-1}^2 a_i - \tfrac{1}{2 \rt_3} (a_0- a_2 +a_1 - a_{-1}) - 2\tfrac{2\bt+1}{\bt(2\bt-1)}\rt_2 e
\end{align*}

\begin{table}[h!tb]
\renewcommand{\arraystretch}{1.5}
	\begin{tabular}{cccc}
		\hline
		Representative & orbit size & length & comment \\
		\hline
		$\1$ & $1$ &  $4$ & \\
		$\tfrac{1}{2}\1\pm v_1$ & $2 \times 2$ & $2$ & $\cAM(\tfrac{1}{2}, \tfrac{1}{2} \pm 2\rt_1)$-axis, $\tau_{\frac{1}{2}\1\pm v_1} = \tau_0$\\
		$a_0$ & $4$ & $1$ &  $\cM(\frac{1}{4}, \bt)$\\
		$\1 - a_0$ & $4$ & $3$ & $\cM(\frac{3}{4}, 1-\bt)$, not primitive \\
		$\tfrac{1}{2}\1\pm v_2$ & $2\times4$ & $2(1\mp \rt_2)$ & {
		\renewcommand{\arraystretch}{1}
		 \begin{tabular}[t]{c}
		$5$ eigenspaces, $C_2$-graded fusion law,\\
		$\tau_{\frac{1}{2}\1\pm v_2} = \half$
		\end{tabular} } \\
		$x(\lm,\mu)$ & - & $2$  & infinite family  \\
	\end{tabular}
	\caption{Idempotents of $4\A(\frac{1}{4}, \bt)$}\label{tab:4Aidempotents}
\end{table}

\begin{proposition}\label{4Aidems}
Suppose that $\bt \neq \frac{1}{2}, -\frac{3}{2}$ and $\FF$ is a field containing the roots $\rt_1, \rt_2, \rt_3$.  The idempotent ideal decomposes into a $0$- and a $1$-dimensional ideal.  The $21$ non-trivial idempotents which are in the $0$-dimensional ideal are the first five rows listed in Table $\ref{tab:4Aidempotents}$.  The $1$-dimensional ideal gives an infinite family of idempotents
\[
x(\lambda, \mu) = \lambda d_0 + \mu d_1 + \tfrac{1}{2}(1 - \lambda - \mu) \tfrac{8}{2\bt - 1}e
\]
arising from the double axis subalgebra, one for each solution of $\lambda^2 + \delta \lambda \mu + \mu^2 = 1$, where $\dl = 2(4\bt-1)$.
\end{proposition}
\begin{proof}
By computation, we find that the idempotent ideal decomposes into a $0$-dimensional and $1$-dimensional ideal and the idempotents in the $0$-dimensional ideal are as claimed.

By Proposition~\ref{4Adoublesub}, $B \cong S(\delta)$ with $\delta = 2(4\bt - 1)$ and (over an algebraically closed field) this has infinitely many idempotents which are also idempotents of $A$. The spin factor $S(\delta) = \FF \1 \oplus V$, where $V = \la e, f\ra$ is a $2$-dimensional vector space with a bilinear form $b$ satisfying $b(e,e) = b(f,f) = 2$ and $b(e,f) = \delta$. Every idempotent in $S(\delta)$ has the form $x = \tfrac{1}{2}(\1 + u)$ with $b(u,u) = 2$.  If $u = \lambda e + \mu f$ we get
\[
2 = b(u,u) = \lambda^2 b(e,e) + 2\lambda\mu b(e,f) + \mu^2 b(f,f) = 2\lambda^2 + 2\delta\lambda\mu + 2\mu^2
\]
which gives $\lambda^2 + \delta \lambda \mu + \mu^2 = 1$. Over (an algebraically closed) $\FF$, this conic has infinitely many solutions. 

Under the isomorphism $B = \lla d_0, d_1 \rra \cong S(\delta)$, $d_0 \mapsto \frac{1}{2}(\1+e)$, $d_0 \mapsto \frac{1}{2}(\1+e)$ and $\1 \mapsto \1$, so we get
\begin{align*}
x &=  \tfrac{1}{2}\big( \1 + \lm(2d_0-\1) + \mu(2d_1-\1) \big) \\
  &=  \lm d_0 + \mu d_1 +  \tfrac{1}{2}(1 - \lambda - \mu)\tfrac{8}{2\bt - 1}e.
\end{align*}
By computation, we see that every element of the $1$-dimensional ideal has precisely this form.
\end{proof}

\begin{remark}
	If $\bt = -\frac{3}{2}$, we can take an algebraically closed field $\FF$ containing $\rt_1$ and $\rt_3$ as above.  Now $v_1$ is still defined and in fact one can check that for $\bt = -\frac{3}{2}$, the complete list of idempotents is as above, except for $v_2$ and $\1-v_2$. 
\end{remark}

\subsection{$4\J(2\bt, \bt)$}

For $4\J(2\bt, \bt)$, we require $\bt  \neq \{0,1, \frac{1}{2}\}$ for $\{1,0,\al,\bt\}$ to be distinct.

By computation, we see that if $\bt \neq -\frac{1}{4}$ the algebra has an identity and $\1 = \tfrac{a_0 +a_2 +a_{-1}+ a_1 + w}{4\bt + 1}$. When $\bt = -\tfrac{1}{4}$, the algebra has no identity and instead $\Ann A = \la a_0 +a_2 +a_{-1}+ a_1 + w \ra$.  Here we also have $\lla a_0, a_2 \rra \cong 2\B$.

\begin{lemma}
	For the axis $a_0 \in 4\J(2\bt, \bt)$, we have
	\begin{align*}
		A_1(a_0) &= \la a_0 \ra, \\	
		A_0(a_0) &= \la 4\bt a_0 - a_1 - a_{-1} - w, a_2 \ra, \\
		A_{2\bt}(a_0) &= \la a_1 +a_{-1} - w \ra, \\
		A_{\bt}(a_0) &= \la a_1 - a_{-1} \ra.
	\end{align*}
\end{lemma}

\begin{lemma}
	There are no ideals of $A := 4\J(2\bt, \bt)$ which contain axes.
\end{lemma}
\begin{proof}Since $(a_0, a_1) = \bt \neq 0$, by Corollary \ref{axisideal}, no proper ideal contains an axis.
\end{proof}

\begin{computation}
	The determinant of the Frobenius form is
	\[
	2  (2\beta - 1)^2  (4\beta + 1)
	\]
\end{computation}

Since $\bt \neq \frac{1}{2}$, we need just examine the case of $\bt = -\frac{1}{4}$.

\begin{proposition}
The ideals and quotients of $A := 4\J(2\bt, \bt)$ are given in Table $\ref{tab:4Jprops}$.
\end{proposition}

\begin{proof}
When $\bt = -\frac{1}{4}$, the radical $R = \la a_0 + a_1 + a_2 + a_{-1} + w\ra$ is $1$-dimensional. By inspection, $R = \Ann A$.  Since the quotient is of Monster type $(-\tfrac{1}{2}, -\tfrac{1}{4})$, it is not isomorphic to the other examples; it is $4\J(-\frac{1}{2}, -\frac{1}{4})^\times$.
\end{proof}

\begin{table}[h!tb]
	\renewcommand{\arraystretch}{1.5}
	\centering
	\footnotesize
	\begin{tabular}{c|c|c|c}
		\hline
		Condition & Ideals & Quotients & Dimension \\
		\hline
		$\bt = -\frac{1}{4}$ & $\Ann A = \la a_0 + a_1 + a_2 + a_{-1} + w\ra$ & $4\J(-\frac{1}{2}, -\frac{1}{4})^\times$  & 4 \\
	\end{tabular}
	\caption{Ideals and quotients of $4\J(2\bt, \bt)$}\label{tab:4Jprops}
\end{table}

Again, as $a_0 a_2 = 0$, we can consider double axes.  The analysis here is very similar to the $4\A(\frac{1}{4}, \bt)$ case.

\begin{lemma}\label{4Jdouble}
In $A = 4\J(2\bt, \bt)$, $d_0 := a_0+a_2$ and $d_1 := a_1+a_{-1}$ are double axes with
\begin{align*}
A_1(d_0) &= \la a_0, a_2 \ra \\
A_0(d_0) &= \la 4\bt d_0 - d_1 -w \ra \\
A_{4\bt}(d_0) &= \la d_1-w \ra \\
A_{2\bt}(d_0) &= \la a_1 - a_{-1} \ra
\end{align*}
If $\bt \neq \frac{1}{4}$, then these are of Monster type $(4\bt, 2\beta)$.
\end{lemma}
\begin{proof}
The proof is analogous to Lemma \ref{4Adouble}.
\end{proof}

\begin{proposition}\label{4Jdoublesub}
The subalgebra $B := \lla d_0, d_1 \rra = \la d_0,d_1,w \ra \cong 3\C(4\bt)$.  Moreover, it is the fixed subalgebra under the action of the Miyamoto group of $4\J(2\bt, \bt)$.
\end{proposition}
\begin{proof}
The proof is analogous to Proposition \ref{4Adoublesub}.  We decompose $d_1 = (d_1)_1 + (d_1)_0 + (d_1)_{4\bt} = 2\bt d_0 -\frac{1}{2}\big(4\bt -d_1-w \big) + \frac{1}{2} \big( d_1-w \big)$ with respect to $d_0$, then we get
\[
d_0d_1 = 2\bt d_0 + \tfrac{4\bt}{2}\big(d_1-w\big) = \tfrac{4\bt}{2}\big(d_0+d_1-w\big)
\]
and as $w^2=w$, we see that $\lla d_0,d_1 \rra \cong 3\C(4\bt)$.
\end{proof}
	
For the idempotents in $4\J(2\bt, \bt)$, we suppose that $\bt \neq -\frac{1}{4}$ and so the algebra has an identity.  Suppose that $\FF$ contains non-zero roots $\rt_1$, $\rt_2$, $\rt_3$, $\rt_4$ such that
\[
\rt_1^2 := 1-4\bt , \qquad \rt_2^2 := 1+4\bt, \qquad  \rt_3^2 := 1-8\bt , \qquad \rt_4^2 := 1-12\bt^2
\]
Define
\begin{align*}
	d_0 &:= a_0+a_2 \\
	x &:= \tfrac{\bt}{\rt_4}\1 -\tfrac{1}{2\rt_4}\left(w+\rt_1\sum_{i=-1}^{2}a_i\right)\\
	y &:=  \tfrac{1}{2\rt_1\rt_2}\left(\rt_3(a_0-a_2)+a_1+a_{-1}-w\right)
\end{align*}

\begin{table}[h!tb]
\renewcommand{\arraystretch}{1.5}
	\begin{tabular}{cccc}
		\hline
		Representative & orbit size & length & comment \\
		\hline
		$\1$ & $1$ &  $\tfrac{6}{4\bt +1}$ & $\bt \neq \tfrac{-1}{4}$ \\
		$w$ & $1$ &  $2$ & $\cM(4\bt, 2\bt)$-axis, $\tau_{w} = \tau_0\tau_1$ \\
		$\1-w$ & $1$ &  $\tfrac{4(1-2\bt)}{1+4\bt}$ &\begin{tabular}[t]{c}
			 $\cM(1	-4\bt, 1-2\bt)$-axis,\\ $\tau_{1-w} = \tau_0\tau_1$
		\end{tabular} \\
		$d_0$ & $2$ &  $2$ & $\cM(4\bt, 2\bt)$-axis, Double axis \\
		$\1-d_0$ & $2$ &  $\tfrac{4(1-2\bt)}{1+4\bt}$ & $\cM(1-4\bt,1- 2\bt)$-axis \\
		$a_0$ & $4$ & $1$ & $\cM(2\bt, \bt)$-axis \\
		$\1 - a_0$ & $4$ &  $\tfrac{5-4\bt}{1+4\bt}$ & $\cM(1-2\bt, 1-\bt)$-axis \\
		$\tfrac{1}{2}\1 \pm x$ & $2\times4$ & $\tfrac{3}{4\bt+1} \pm \tfrac{1-2\bt}{48\bt^3 +12\bt^2 -4\bt -1}\rt_4$ & 
		{\renewcommand{\arraystretch}{1}
		\begin{tabular}[t]{c}
			$5$ eigenspaces, \\ $C_2$-graded fusion law, \\$\tau_{\frac{1}{2}\1\pm x} = \half$
		\end{tabular}
		}   \\

		$\tfrac{1}{2}\1\pm y$ & $2\times4$ & $\tfrac{3}{4\bt+1} $ & {\renewcommand{\arraystretch}{1}
		\begin{tabular}[t]{c}
			$5$ eigenspaces, \\$C_2$-graded fusion law, \\ $\tau_{\frac{1}{2}\1\pm y} = \tau_0$
		\end{tabular}
		} \\
	\end{tabular}
	\caption{Idempotents of $4\J(2\bt, \bt)$ when $\bt \neq \tfrac{-1}{4}$}\label{tab:4Jidempotents}
\end{table}

\begin{computation}\label{4Jidems}
Suppose that $\bt \neq -\frac{1}{4}$ and $\FF$ is a field containing non-zero roots $\rt_1, \rt_2, \rt_3, \rt_4$.  Then the algebra has $31$ non-trivial idempotents as given in Table $\ref{tab:4Jidempotents}$.
\end{computation}

\begin{remark}
Note that, since we assume that $\rt_3$ is non-zero, $\bt \neq \frac{1}{8}$.  We will see later in Proposition \ref{X4:exceptionaliso} that if $\bt = \frac{1}{8}$, then $4\J(\frac{1}{4}, \frac{1}{8}) \cong 4\A(\frac{1}{4}, \frac{1}{8})$ and the double axis subalgebra is $3\C(\frac{1}{2}) \cong S(\dl)$, which has a $1$-dimensional family of idempotents.
\end{remark}

\subsection{$4\B(\al, \frac{\al^2}{2})$}

For $4\B(\al, \frac{\al^2}{2})$, we require $\al  \neq 0,1, 2, \pm \sqrt{2}$ for $\{1,0,\al,\bt\}$ to be distinct.

By computation, the algebra has an identity if and only if $\al \neq -1$, in which case $ \1 = \tfrac{a_0 +a_2 +a_{-1}+ a_1 + (1-\al)c}{\al + 1}$.  If $\al = -1$, then we will see that $A$ has a $2$-dimensional annihilator.  From the multiplication table, we see that $\lla a_0, a_2 \rra = \la a_0,a_2,c \ra \cong 3\C(\al)$.

\begin{lemma}\label{4Beigenspaces}
	For the axis $ a_0 \in 4\B(\al, \tfrac{\al^2}{2})$, we have
	\begin{align*}
		A_1(a_0) &= \la a_0 \ra, \\	
		A_0(a_0) &= \la \al a_0 - (a_2 + c), a_1 + a_{-1} -\al c \ra, \\
		A_{\al}(a_0) &= \la a_2 - c \ra, \\
		A_{\tfrac{\al^2}{2}}(a_0) &= \la a_1 - a_{-1} \ra 
	\end{align*}
\end{lemma}

\begin{lemma}
No proper ideals of $A := 4\B(\al, \frac{\al^2}{2})$ contain axes.
\end{lemma}
\begin{proof} Since $(a_0, a_1) = \frac{\al}{2} \neq 0$, by Corollary \ref{axisideal}, no proper ideal contains an axis.
\end{proof}

\begin{computation}
	The determinant of the Frobenius form is
	
	\[
	\tfrac{1}{16}  (\al - 2)^4 (\al + 1)^2
	\]
\end{computation}

Since we require $\al \neq 2$ so that $\al \neq \bt$, the algebra has a non-trivial radical when $\al = -1$.

\begin{proposition}\label{4Bideals}
	The ideals and quotients of $A := 4\B(\al, \frac{\al^2}{2})$ are given in Table $\ref{tab:4Bprops}$.
\end{proposition}

\begin{proof}
	When $\al = -1$, the radical $R = \lla r_0, r_1 \rra$ is two dimensional with $r_0 = a_0 + a_2 + c$ and $r_1 = a_1 + a_{-1} + c$. By computation, we observe that $r_0$ and $r_1$ annihilate the algebra, and since the radical is a maximal ideal it must be the annihilator $\Ann A$. So every 1-dimensional subspace is an ideal. Let $\sigma$ be the involutory automorphism which switches the two generating axes $a_0$ and $a_1$. Clearly $\sigma$ switches $r_0$ and $r_1$, hence $I_1 = \la r_0 - r_1 \ra$ and $I_2 = \la r_0 + r_1 \ra$ are the only 1-dimensional ideals which are $\sigma$-invariant, hence the only ones with symmetric quotients.

    We first consider $C_1 := A/I_1$; let $\pi \colon A \to C_1$ be the quotient map.  We now define $n:= \pi(a_0+a_2+c)=\pi(r_1)$, $z_1 := \pi(c)$, $e := \pi(a_0-a_2)$, and $f:=\pi(-a_0 +2a_1-a_2)$.  We observe that $z_1^2 = z_1$, $e\cdot z_1 = -e$, $f\cdot z_1 = -f$, $e\cdot f = 0$ and $n$ annihilates the algebra. Moreover, $\pi(a_0) = \tfrac{1}{2}(e-z_1+n)$ and $\pi(a_1) = \tfrac{1}{2}(f-z_1+n)$.   Hence, by \cite[Definition 6.1]{splitspin}, $C_1 \cong \IY_3(-1,\frac{1}{2},0) = \widehat{S}(b, -1)^\circ$.

    It is easy to see that $\Ann C_1 =  \la n \ra = \la \pi(r_1) \ra$ and, by \cite[Theorem 6.8]{splitspin}, the algebra has a unique quotient $\IY_3(-1,\frac{1}{2},0)^\times = S(b, -1)^\circ$. By the correspondence theorem, $B/R \cong C_1/\Ann C_1 \cong \IY_3(-1,\frac{1}{2},0)^\times$.

    Finally, for $C_2 := A/I_2$, we observe that $C_2$ is of dimension $4$, does not have an identity,  and has four Monster type $(-1, \tfrac{1}{2})$ axes. So, $C_1$  is not isomorphic to the other examples; it is $4\B(-1, \frac{1}{2})^\times$.
\end{proof}

\begin{table}[h!tb]
	\renewcommand{\arraystretch}{1.5}
	\centering
	\footnotesize
	\begin{tabular}{c|c|c|c}
		\hline
		Condition & Ideals & Quotients & Dimension \\
		\hline
		$\al = -1$ & $R := \Ann A = \la a_0  + a_2 + c, a_1 + a_{-1} + c\ra$  & $\IY_3(-1,\tfrac{1}{2},0)^\times$ & 3 \\
		$\al =-1$ & $\la a_0  + a_2 - a_1 - a_{-1} \ra$  & $\IY_3(-1,\tfrac{1}{2},0)$ & 4\\
		$\al =-1$ & $\la a_0  + a_2 + a_1 + a_{-1} + 2c \ra$  & $4\B(-1, \tfrac{1}{2})^\times
		$ & 4\\
		$\al =-1$ & Any $1$-dimensional subspace of $R$  & Non-symmetric, $\cM(\al, \frac{\al^2}{2})$ & 4
	\end{tabular}
	\caption{The symmetric ideals of $4\B(\al, \frac{\al^2}{2})$}\label{tab:4Bprops}
\end{table}
	
Finally, we consider the idempotents in $4\B(\al, \frac{\al^2}{2})$.  Suppose that $\al \neq -1, \frac{1}{2}$, $\al^2-4\al+1 \neq 0$ and $\al^4-2\al^3-2\al+1 \neq 0$.  Take field extensions so that $\FF$ contains roots $\rt_1$, $\rt_2$ such that
\[
\left(\al^4-2\al^3-2\al+1\right)\rt_1^2 := \al^2-4\al+1, \qquad \left(\al^4-2\al^3-2\al+1\right)\rt_2^2 := 1-2\al 
\]
Let $\1_{3\C} = \1_{\lla a_0, a_2 \rra}$ be the identity in the subalgebra $\lla a_0, a_2 \rra \cong 3\C(\al)$, which exists as $\al \neq -1$.  Note that $\half$ switches the two $3\C(\al)$ subalgebras $\lla a_0, a_2 \rra$ and $\lla a_1, a_{-1} \rra$ and so ${\1_{3\C}}^\half = \1_{\lla a_1, a_{-1} \rra}$.
Define
\begin{align*}
	x &:= \1_{3\C} - a_0 \\
	y &:= \tfrac{\rt_1}{2(\al+1)}\left(\al \sum_{-1}^{2}a_{i} - (1+\al^2)c\right) - \tfrac{\rt_2}{2}(a_0-a_2+a_1-a_{-1}) 
\end{align*}

\begin{table}[h!tb]
\renewcommand{\arraystretch}{1.3}
	\begin{tabular}{cccc}
		\hline
		Representative & orbit size & length & comment \\
		\hline
		$\1$ & $1$ &  $\tfrac{5-\al}{1+\al}$ & $\al \neq -1$ \\
		$c$ & $1$ &  $1$ & $\cJ(\al)$-axis, $\tau_{c} = \tau_0\tau_1$ \\
		$\1-c$ & $1$ &  $\tfrac{4-2\al}{1+\al}$ & $\cJ(1-\al)$-axis, $\tau_{c} = \tau_0\tau_1$ \\
		$\1_{3\C}$ & $2$ &  $\tfrac{3}{1+\al}$ & $\cJ(\al)$-axis, $\al \neq -1$ \\
		$\1_{3\C}-c = \1 - {\1_{3\C}}^\half$ & $2$ &  $\tfrac{2-\al}{1+\al}$ & $\cJ(\al)$-axis, $\al \neq -1$ \\
			$a_0$ & $4$ & $1$ & $\cM(\al,  \tfrac{\al^2}{2})$-axis\\
		$\1 - a_0$ & $4$ &  $\tfrac{4-2\al}{1+\al}$ & $\cM(1-\al, 1- \tfrac{\al^2}{2})$-axis \\
			$x$ & $4$ & $\tfrac{2-\al}{1+\al}$ &\begin{tabular}[t]{c}
				$\cM(1-\al, \al - \frac{\al^2}{2})$-axis, \\$\tau_x = \tau_0$
			\end{tabular}   \\
		$\1 - x$ & $4$ &  $\tfrac{3}{1+\al}$ & \begin{tabular}[t]{c}
			$\cM(\al, \bt-\al+1)$-axis,\\ $\tau_{\1-x} = \tau_0$ 
		\end{tabular} \\
		$\frac{1}{2}\1 \pm y$ & $2\times4$ &  $\tfrac{5-\al \mp(\al^2 - 4\al + 1)\rt_1}{2(1+\al)}$  & 
		{\renewcommand{\arraystretch}{1}
		\begin{tabular}[t]{c}
			$5$ eigenspaces,\\ $C_2$-graded fusion law \\ $\tau_{\frac{1}{2}\1\pm y} = \half$
		\end{tabular}
		}
	\end{tabular}
	\caption{Idempotents of $4\B(\al, \frac{\al^2}{2})$}\label{tab:4Bidempotents}
\end{table}

\begin{computation}
If $\al \neq -1, \frac{1}{2}$, $\al^2-4\al+1 \neq 0$ and $\al^4-2\al^3-2\al+1 \neq 0$ and $\FF$ contains the roots $\rt_1$ and $\rt_2$, then $4\B(\al, \frac{\al^2}{2})$ has $31$ non-trivial idempotents as given in Table $\ref{tab:4Bidempotents}$.
\end{computation}

We note that $x$ is a primitive $\cM(1-\al, \al-\frac{\al^2}{2})$-axis and this gives an algebra isomorphism to $4\Y(\al, \frac{1-\al^2}{2})$ as we will see in the next subsection.

\subsection{$4\Y(\al, \frac{1-\al^2}{2})$}

For $4\Y(\al, \frac{1-\al^2}{2})$, we require $\al  \neq \{0, \pm 1, \pm \sqrt{-1}, -1 \pm \sqrt{2} \}$ for $\{1,0,\al,\bt\}$ to be distinct.  By computation, we see that the algebra always has an identity and
\[
\1 = \tfrac{1}{\al}\big(a_0 +a_2 +a_{-1}+ a_1 - (2 + \al)c\big)
\]

\begin{lemma}
	For the axis $a_0 \in 4\Y(\al, \tfrac{1-\al^2}{2})$, we have
	\begin{align*}
		A_1(a_0) &= \la a_0 \ra, \\	
		A_0(a_0) &= \la (\al-1)a_0 - a_2 + c, a_1 +a_{-1} -(\al+1)c \ra, \\
		A_{\al}(a_0) &= \la a_0 +(1-\al)a_{2} -(\al+1)c\ra, \\
		A_{\frac{1-\al^2}{2}}(a_0) &= \la a_1 - a_{-1} \ra 
	\end{align*}
\end{lemma}

\begin{theorem}\label{4B4Yiso}
Let  $A:= 4\Y(\al, \frac{1-\al^2}{2})$ with axes $a_0, a_1, a_2, a_{-1}$ and $B := 4\B(\tilde{\al}, \frac{\tilde{\al}^2}{2})$ with axes $b_0, b_1, b_2, b_{-1}$.  If $\al \neq 2$, then $4\Y(\al, \frac{1-\al^2}{2}) \cong 4\B(\tilde{\al}, \frac{\tilde{\al}^2}{2})$ as algebras \textup{(}but not as axial algebras\textup{)}, where $\tilde{\al} = 1 - \al$.  The isomorphism is given by $a_i := \1_{\lla b_i, b_{i+2} \rra} - b_i$.

The algebras are not isomorphic when $\al = 2$, equivalently when $\tilde{\al} = -1$.
\end{theorem}
\begin{proof}
Since $\tilde{\al} \neq -1$, the subalgebra $\lla b_0, b_2 \rra \cong 3\C(\tilde{\al})$ has an identity $\1_{\lla b_0, b_2 \rra} = \frac{1}{\tilde{\al}+1}(b_0+b_2 + c)$.  Hence $a_0 := \1_{\lla b_0, b_{2} \rra} - b_0$ is an idempotent in $B$.  Clearly $\1_{\lla b_0, b_{2} \rra}$ has $\la b_0,b_2, c \ra$ in its $1$-eigenspace and, by calculation, we see that $b_1+b_{-1} -\tilde{\al} c$ is a $0$-eigenvector and $b_1-b_{-1}$ is an $\tilde{\al}$-eigenvector (in fact, from Table \ref{tab:4Bidempotents}, $\1_{\lla b_0, b_{2} \rra}$ is a non-primitive $\cJ(\tilde{\al})$-axis).  By Lemma \ref{4Beigenspaces}, we know the eigenspaces for $b_0$.  So, $a_0$ is semisimple, with eigenspaces $B_1 = \la a_0 \ra$, $B_0 = \la b_0, b_1+b_{-1} -\tilde{\al} c\ra$, $B_{1-\tilde{\al}} = \la b_2 - c \ra$ and $B_{\tilde{\al} - \frac{\tilde{\al}^2}{2}} = \la b_1 - b_{-1} \ra$.  Note that the fusion law can be deduced for $a_0 = \1_{3\C} - b_0$ simply by comparing the eigenspaces for $a_0$ and $b_0$.  Hence we obtain the comment in Table \ref{tab:4Bidempotents} that $a_0 = \1_{3\C} - b_0$ is a primitive Monster type $(1-\tilde{\al}, \tilde{\al} - \frac{\tilde{\al}^2}{2}) = (\al, \frac{1-\al^2}{2})$ axis.

Similarly $a_1 := \1_{\lla b_1, b_{-1} \rra} - b_1$ is a primitive Monster type $(\al, \frac{1-\al^2}{2})$ axis.  Let $C = \lla a_0, a_1 \rra$.  By above, $\tau_{a_0}$ switches $a_1$ and $a_{-1}$ and fixes $\1_{\lla b_1, b_{-1} \rra} = \frac{1}{\tilde{\al}+1}(b_1+b_{-1} + c)$. Hence $a_1^{\tau_{a_0}} = \1_{\lla b_1, b_{-1} \rra} - b_{-1} = a_{-1}$ and so $a_{-1}$, and by symmetry $a_2$, are in $C$.  Now, we know that $\lla a_1, a_{-1} \rra = \lla b_1, b_{-1} \rra \cong 3\C(\tilde{\al})$ and so $c \in \lla a_1, a_{-1} \rra$.  Therefore $\lla a_0, a_1 \rra = C = B$.  Hence $B$ is also a symmetric $2$-generated Monster type $(\al, \frac{1-\al^2}{2})$ algebra.  By checking the structure constants one sees that this is $4\Y(\al, \frac{1-\al^2}{2})$.

Finally, observe that $4\Y(\al, \frac{1-\al^2}{2})$ has an identity if $\al = 2$, but $4\B(\tilde{\al}, \frac{\tilde{\al}^2}{2})$ has no identity if $\tilde{\al} = -1$, hence they are not isomorphic.
\end{proof}

 Using the isomorphism, if $\al \neq 2$, then $\lla a_0, a_2 \rra \cong 3\C(1-\tilde\al) = 3\C(\al)$ with $\1_{\lla a_0, a_2 \rra} = \1_{3\C} = \tfrac{a_0 + a_{2}-c}{\al}$.  One can check that this also holds when $\al = 2$.  So in all cases, $\lla a_0, a_2 \rra \cong 3\C(\al) \cong \lla a_1, a_{-1} \rra$.    However, we note that these subalgebras intersect in $\la c \ra$, where $c$ is an element of Jordan type $1-\al$, not $\al$ in the subalgebras (and in fact in the whole algebra $A$).

\begin{lemma}
If $\al \neq 2$, then $A:= 4\Y(\al, \tfrac{1-\al^2}{2})$ has no proper ideal which contain axes.
\end{lemma}
\begin{proof}
There are two orbits of axes $\{a_0,a_2\}$ and $\{a_1,a_{-1}\}$ under the action of $\Miy(A)$ and $(a_0, a_1) =  \tfrac{1}{4} (\al-2)(\al+1)$.  We know that $\al \neq -1$, so if in addition $\al \neq 2$, then by Corollary \ref{axisideal}, $A$ has no proper ideals which contain axes.  
\end{proof}

Note that, when $\al = 2$, the orbit projection graph is disconnected and so there may be ideals which contain an axis.  We will see that there are indeed such ideals, but this will be easier to prove once we have considered the radical.

\begin{computation}
	The determinant of the Frobenius form is:
	
	\[
	-\frac{1}{16}  \frac{\al^4  (\al - 2)^3}{\al + 1}
	\]
\end{computation}

So, the algebra has a non-trivial radical if and only if $\al = 2$.  Note that this is also when the orbit projection graph is disconnected.

\begin{proposition}
The ideals and quotients of $4\Y(\al, \frac{1-\al^2}{2})$ are given in Table $\ref{tab:4Yaprops}$. In particular, there are two ideals which contain axes.
\end{proposition}

\begin{proof}
When $\al = 2$, the radical is $3$-dimensional with $R := \la (a_0 - a_2),  (a_1 - a_{-1}),  c\ra$. Setting $a := a_0 - a_2$, $b:= a_1 - a_{-1}$, we observe that $R$ decomposes as a $\Miy(A) \cong V_4$ module as a trivial module $\la c \ra$ plus the sum of two non-isomorphic sign modules $\la a\ra$, $\la b \ra$.  Since $2a_0 c = -a+3c$ and $2a_1 c = -b+3c$, any sub-ideal containing $c$ will generate all of $R$. On the other hand, $2a_0 a = a-3c$ and $2a_1b = b-3c$, so any ideals containing $a$, or $b$ will also contain $c$.  Hence $R$ has no non-trivial proper sub-ideals.

Observe that $\bar{a}_0 = \bar{a}_2$ and $\bar{a}_1 = \bar{a}_{-1}$ span the quotient which is $2$-dimensional.  It is now easy to see that $\bar{a}_0 \bar{a}_1 = 0$ and so $A/R \cong 2\B$.

Now suppose that $I$ is a proper ideal which contains an axis.  Since $I$ is invariant under the Miyamoto group, there are two cases, either $a_0, a_2 \in I$, or $a_1,a_{-1} \in I$.  Note that these cases are switched by the automorphism which switches $a_0$ and $a_1$, so it is enough just to consider when $a_0, a_2 \in I$.  The intersection $I \cap R$ is a sub-ideal of $R$.  However, $0 \neq a_0 - a_2 \in I \cap R$ and so $I \cap R = R$.  Now by the correspondence theorem, $I$ corresponds to an ideal of $A/R \cong 2\B$.  This has precisely two non-trivial ideals, $\la \bar{a}_0 \ra$ and $\la \bar{a}_1 \ra$; hence the result follows.
\end{proof}

\begin{table}[h!tb]
	\renewcommand{\arraystretch}{1.5}
	\centering
	\footnotesize
	\begin{tabular}{c|c|c|c}
		\hline
		Condition & Ideals & Quotients & Dimension \\
		\hline
		 $\al = 2$ & $\la a_0 - a_2, a_1-a_{-1}, c \ra $ & $2\B$ & 2 \\
		 $\al = 2$ & $\la a_0, a_2, a_1-a_{-1}, c \ra $ & $1\A$ & 1 \\
		 $\al = 2$ & $\la a_1, a_{-1}, a_0 - a_2, c \ra $ & $1\A$ & 1
	\end{tabular}
	\caption{Ideals and quotients of $4\Y(\al, \frac{1-\al^2}{2})$}\label{tab:4Yaprops}
\end{table}

	We now turn to examine the idempotents in $4\Y(\al, \frac{1-\al^2}{2})$.	Suppose that $\FF$ contains roots $\rt_1$, $\rt_2$ such that
\[
\rt_1^2 := \tfrac{\al^2+2\al-2}{\al^4-2\al^3+4\al-2} , \qquad \rt_2^2 :=\tfrac{2\al-1}{\al^4-2\al^3+4\al-2}, 
\]
As before, set $\1_{3\C} := \1_{\lla a_0, a_2 \rra}$ and define
\begin{align*}
x &:= \1_{3\C} - a_{0} \\
y &:=  \tfrac{\rt_1}{2(\al)}\left((1-\al)\sum_{-1}^{2}a_{i} + (\al^2-2)c\right) + \tfrac{\rt_2}{2}(a_0-a_2+a_1-a_{-1}) 
\end{align*}

\begin{table}[h!tb]
	\renewcommand{\arraystretch}{1.3}
	\begin{tabular}{cccc}
		\hline
		Representative & orbit size & length & comment \\
		\hline
		$\1$ & $1$ &  $\tfrac{4+\al}{1+\al}$ & \\
		$c$ & $1$ &  $\tfrac{2-\al}{1+\al}$ & $\cJ(1-\al)$-axis, $\tau_{c} = \tau_0\tau_1$ \\
		$\1-c$ & $1$ &  $2$ & $\cJ(\al)$-axis, $\tau_{c} = \tau_0\tau_1$ \\
		$\1_{3\C}$ & $2$ &  $\tfrac{3}{1+\al}$ & $\cJ(1-\al)$-axis,  $\al \neq -1$ \\
		$\1_{3\C}-c$ & $2$ &  $1$ & $\cJ(\al)$-axis, $\al \neq -1$\\
		$a_0$ & $4$ & $1$ & $\cM(\al, \frac{1-\al^2}{2})$-axis \\
		$\1 - a_0$ & $4$ &  $\tfrac{3}{1+\al}$ & $\cM(1-\al, \frac{1+\al^2}{2})$-axis   \\
		$x$ & $4$ & $\tfrac{2-\al}{1+\al}$ &\begin{tabular}[t]{c}
			$\cM(1-\al, 1-\al-\bt)$-axis,\\ $\tau_x = \tau_{0}$  
		\end{tabular} \\
		$\1 - x$ & $4$ &  $2$ & $\cM(\al, \bt)$-axis, $\tau_{\1-x} = \tau_{0}$  \\
			$\tfrac{1}{2}\1 \pm y$ & $2\times4$ &  $\tfrac{4+\al }{2(\al+1)} \mp\tfrac{2(\al-\bt) - 1}{2(1+\al)}\rt_1  $  &
			{\renewcommand{\arraystretch}{1} \begin{tabular}[t]{c}
			$5$ eigenspaces,\\ $C_2$-graded fusion law \\ $\tau_{\frac{1}{2}\1\pm y} = \half$
		\end{tabular}}  \\
	\end{tabular}
	\caption{Idempotents of $4\Y(\al, \frac{1-\al^2}{2})$}\label{tab:4Yidempotents}
\end{table}

\begin{computation}
Suppose that $\al^4-2\al^3+4\al-2 \neq 0$ and $\FF$ contains roots $\rt_1$ and $\rt_2$, then $4\Y(\al, \frac{1-\al^2}{2})$ has $31$ non-trivial idempotents given in Table $\ref{tab:4Yidempotents}$.
\end{computation}

\subsection{$4\Y(\frac{1}{2}, \bt)$}

For $4\Y(\frac{1}{2}, \bt)$, we require $\bt  \neq \{0,1, \frac{1}{2}\}$ for $\{1,0,\al,\bt\}$ to be distinct.  By computation, we see that the algebra always has an identity and
\[
\1 = \tfrac{1}{1-2\bt }\left( \tfrac{1}{2}(a_0 +a_2 +a_{-1}+ a_1) + (1-6\bt)z \right)
\]

\begin{lemma}
	For the axis $ a_0 \in 4\Y(\frac{1}{2}, \bt)$, we have
	\begin{align*}
		A_1(a_0) &= \la a_0 \ra, \\	
		A_0(a_0) &= \la (4\bt-1)a_0 + a_2 - 2(4\bt-1) z, a_1 +a_{-1} -4\bt z \ra, \\
		A_{\frac{1}{2}}(a_0) &= \la (8\bt-1)a_0 + a_2 -8 \bt w \ra, \\
		A_{\bt}(a_0) &= \la a_1 - a_{-1} \ra 
	\end{align*}
\end{lemma}

\begin{lemma}
	There are no ideals of $A := 4\Y(\frac{1}{2}, \bt)$ which contain axes.
\end{lemma}
\begin{proof}
Since $(a_0, a_1) =  4\bt^2 \neq 0$, by Corollary \ref{axisideal}, $A$ has no proper ideals which contain axes.
\end{proof}

\begin{computation}
	The determinant of the Frobenius form is
	
	\[
	2^8  \bt^2  (2\bt - 1)^6
	\]
\end{computation}

\begin{corollary}
The algebra $4\Y(\frac{1}{2}, \bt)$ has no non-trivial proper ideals or quotients.
\qed
\end{corollary}

We now examine the subalgebra structure of $4\Y(\frac{1}{2}, \bt)$. 

\begin{proposition}\label{4Ybtsub}
	The subalgebras $ \lla a_0, a_{2} \rra = \la a_0,a_{2}, z \ra$ and $ \lla a_1, a_{-1} \rra = \la a_1,a_{-1},z \ra$ are axial algebras of Jordan type $\cJ(\tfrac{1}{2})$ and
	\[
	\lla a_i, a_{i+2}\rra \cong S(\delta)
	\]
where $\delta = 64\bt^2 -32\bt + 2$.
\end{proposition}
\begin{proof}
This is very similar to Proposition \ref{4Adoublesub}.
\end{proof}

We now turn to examine the idempotents in $4\Y(\tfrac{1}{2}, \bt)$. 
 
\begin{computation}\label{4Ybtcomp}
The idempotent ideal for $4\Y(\frac{1}{2}, \bt)$ decomposes as the sum of a $0$- and a $1$-dimensional ideal.  The $0$-dimensional ideal has three non-trivial idempotents as given in the first three rows of Table $\ref{tab:4Y_btidempotents}$.  Over an algebraically closed field, the $1$-dimensional ideal comprises four infinite families of idempotents
\begin{align*}
x_0(\lm, \mu) &:= \lm a_0 + \mu a_2 + (1-\lm-\mu)z, \\
x_1(\lm, \mu) &:= \lm a_1 + \mu a_{-1} + (1-\lm-\mu)z,
\end{align*}
$\1-x_0(\lm, \mu)$ and $\1-x_1(\lm,\mu)$, where $\lm$ and $\mu$ satisfy $\lm^2 + \mu^2 + 2(1-4\bt)\lm\mu = \lm+\mu$.
\end{computation}

\begin{table}[h!tb]
\centering
	\begin{tabular}{cccc}
		\hline
		Representative & orbit size & length & comment \\
		\hline
		$\1$ & $1$ &  $3$ &  \\
		$z$ & $1$ &  $1$ & $\cJ(\tfrac{1}{2})$-axis, $\tau_{c} = \tau_0\tau_1$ \\
		$\1-z$ & $1$ &  $2$ & $\cJ(\tfrac{1}{2})$-axis, $\tau_{c} = \tau_0\tau_1$  \\
		$x_0(\lm, \mu)$ & - & 1  & infinite family \\
		$\1 - x_0(\lm, \mu)$ & - & 2  & infinite family \\
		$x_1(\lm, \mu)$ & - &  1 & infinite family \\
		$\1 - x_1(\lm, \mu)$ & - & 2  & infinite family \\
	\end{tabular}
	\caption{Idempotents of $4\Y(\tfrac{1}{2},\bt)$}\label{tab:4Y_btidempotents}
\end{table}

Note that $x_0(\lm, \mu)$ is contained in $\lla a_0, a_2 \rra$.  However, by Proposition \ref{4Ybtsub}, this is isomorphic to $S(\dl)$, which is known to have infinitely many idempotents.  This is precisely the family $x_0(\lm, \mu)$.  Note also that $a_0$, $a_2$ and $\1_{\lla a_0, a_2 \rra}$ are members of this family, so it is clear that there are several types of fusion law in this family.

\subsection{Exceptional isomorphisms}

The are some exception axial isomorphisms between some of the $X(4)$ algebras.

\begin{proposition}\label{X4:exceptionaliso}
When $(\al, \bt) = (\frac{1}{2}, \frac{3}{8})$, then $4\Y(\al, \frac{1-\al^2}{2}) \cong 4\Y(\frac{1}{2}, \bt)$.  In particular our notation for these algebras is indeed well-defined.

The only other exceptional axial isomorphisms between the algebras $4\A(\frac{1}{4}, \bt)$, $4\J(2\bt, \bt)$, $4\B(\al, \frac{\al^2}{2})$, $4\Y(\al, \frac{1-\al^2}{2})$ and $4\Y(\frac{1}{2}, \bt)$ are precisely the following:
\begin{align*}
4\A(\tfrac{1}{4}, \tfrac{1}{8}) &\cong 4\J(\tfrac{1}{4}, \tfrac{1}{8}), & 4\J(\tfrac{1}{2}, \tfrac{1}{4}) &\cong 4\Y(\tfrac{1}{2}, \tfrac{1}{4}),\\
4\B(\pm \tfrac{1}{\sqrt{2}}, \tfrac{1}{4}) &\cong 4\Y(\pm \tfrac{1}{\sqrt{2}}, \tfrac{1}{4}), & 4\B(\tfrac{1}{2}, \tfrac{1}{8}) &\cong 4\Y(\tfrac{1}{2}, \tfrac{1}{8})
\end{align*}
\end{proposition}
\begin{proof}
For $4\Y(\al, \frac{1-\al^2}{2})$ and $4\Y(\frac{1}{2}, \bt)$ to possibly be isomorphic, we require that $(\al, \bt) = (\frac{1}{2}, \frac{3}{8})$.  On the other hand, it is easy to check from the structure constants that they are indeed isomorphic in this case, with $c$ mapping to $z$.

Two algebras can only be isomorphic only if they have the same axial subalgebras.  We use Table \ref{tab:X4subalg} and check each isomorphism class of axial subalgebra in turn.  For $2\B$, there are three algebras $4\A(\tfrac{1}{4}, \bt)$, $4\J(2\bt, \bt)$ and $4\Y(\tfrac{1}{2}, \tfrac{1}{4})$.  Clearly $4\A(\tfrac{1}{4}, \bt)$ and $4\J(2\bt, \bt)$ can only be isomorphic when $(\al,\bt) = (\frac{1}{4}, \frac{1}{8})$; checking here we see that they are indeed isomorphic with $-\frac{1}{16}(\sum_{i=-1}^2 a_i + w)$ mapping to $e$.  Similarly, $4\J(\tfrac{1}{2}, \tfrac{1}{4}) \cong 4\Y(\tfrac{1}{2}, \tfrac{1}{4})$, with $\frac{1}{2}(\sum_{i=-1}^2 a_i - w)$ mapping to $z$.  For $3\C(\al)$, we have already checked that $4\Y(\al, \frac{1-\al^2}{2}) \cong 4\Y(\frac{1}{2}, \bt)$ when $(\al, \bt) = (\frac{1}{2}, \frac{3}{8})$, so it remains to check if $4\B(\al, \frac{\al^2}{2})$ and $4\Y(\al, \frac{1-\al^2}{2})$ can be isomorphic.  Note that, by Theorem \ref{4B4Yiso}, these are already isomorphic as algebras, so we just require $\frac{\al^2}{2} = \bt = \frac{1-\al^2}{2}$ for them to be isomorphic as axial algebras.  This is equivalent to $\al = \pm \frac{1}{\sqrt{2}}$.  Finally, we note that $3\C(\al)$ can be isomorphic to $S(\dl)$, provided $\al = \frac{1}{2}$.  This gives one more possibility of $4\B(\frac{1}{2}, \frac{1}{8})$ and $4\Y(\frac{1}{2}, \frac{1}{8})$ being isomorphic.  It is easy to check that they have the same structure constants and hence are isomorphic.
\end{proof}

\section{Algebras with axet $X(5)$}

There is only one algebra with axet $X(5)$, which is $A := 5\A(\al, \frac{5\al-1}{8})$.  As we will show, this algebra has an additional automorphism, so that $D_{10} = \Miy(A) < \Aut(A) = F_{20}$, the Frobenius group of order $20$.  We choose a basis, so that an element of order $4$ permutes the axes and inverts the last basis element.\footnote{Compared to Yabe's basis $w := \sum_{i=-2}^2 \hat{a}_i + \frac{4}{\bt}\hat{p}_1$.}

The computations in this section can be found in \cite[\texttt{X5 algebras.m}]{githubcode}.

\begin{table}[h!tb]
	\setlength{\tabcolsep}{4pt}
	\renewcommand{\arraystretch}{1.5}
	\centering
	\footnotesize
	\begin{tabular}{c|c|c}
		Type & Basis & Products \& form \\ \hline
		$5\A(\al, \frac{5\al-1}{8})$ & 
		\begin{tabular}[t]{c}
			$a_{-2}$, $a_{-1}$, \\
			$a_0$, $a_1$, \\
			$a_2$, $w$
		\end{tabular} &
		\begin{tabular}[t]{c}
			$a_i \cdot a_{i+1} = \frac{\bt}{4}\big( 3(a_i +a_{i+1}) - (a_{i+2} +a_{i-1}+a_{i-2}) +w\big)$ \\
			$a_i \cdot a_{i+2} = \frac{\bt}{4}\big( 3(a_i +a_{i+2}) - (a_{i+1} +a_{i-1}+a_{i-2}) -w\big)$\\
			$a_i \cdot w = (\al-\bt)\big( (a_{i+1} + a_{i-1}) -(a_{i+2}+a_{i-2}) + w\big) $\\
			$w^2 = \frac{(\al-\bt)(7\al-3)}{8\bt}(a_0+a_1+a_2+a_{-1}+a_{-2})$ \\
			$(a_i, a_{i+1}) = \frac{3\bt}{4} = (a_i, a_{i+2})$, \\
			$(a_i, w) = 0$, $(w,w) = \frac{5(\al-\bt)(7-3\al)}{8\bt}$
			\vspace{4pt}
		\end{tabular}
\end{tabular}
\caption{$2$-generated $\cM(\alpha, \beta)$-axial algebras on $X(5)$}\label{tab:2genMonster5}
\end{table}

\subsection{$5\A(\al, \frac{5\al-1}{8})$}

For $5\A(\al, \tfrac{5\al-1}{8})$, we require $\al  \neq 0, 1, -\tfrac{1}{3}, \tfrac{1}{5}, \tfrac{9}{5}$ for $\{1,0,\al,\bt\}$ to be distinct.  By computation, we see that  if $\ch \FF \neq 5$, then the algebra has an identity and $ \1 = \tfrac{1}{5(\al-\bt)}(a_0+a_1+a_{-1}+a_2+a_{-2})$.

\begin{proposition}\label{5AIY5iso}\textup{\cite{highwater5}}
If $\ch \FF = 5$, then $5\A(\al, \frac{5\al-1}{8}) = 5\A(\al, \frac{1}{2}) \cong \IY_5(\al, \frac{1}{2})$
\end{proposition}
\begin{proof}
It is straightforward to check that the basis $a_0, \dots, a_4, a_0a_1 - \frac{1}{2}(a_0+a_1)$ for $5\A(\al, \frac{1}{2})$ has the same structure constants as for $\IY_5(\al, \frac{1}{2})$.
\end{proof}

It turns out that it is more convenient to consider the algebra as $\IY_5(\al, \frac{1}{2})$ in characteristic $5$, so here we just record the results for characteristic $5$, but defer the proofs until later in Section \ref{sec:IY5}.

\begin{lemma}
	For the axis $a_0 \in 5\A(\al, \tfrac{5\al-1}{8})$, we have
	\begin{align*}
		A_1(a_0) &= \la a_0 \ra, \\	
		A_0(a_0) &= \la 3\bt a_0 - 2(a_2+a_{-2}) -\tfrac{\bt}{\al-\bt}w, \\
		& \phantom{{}= \la{}} \qquad (a_1+a_{-1}) - (a_2+a_{-2}) -\tfrac{\bt}{\al-\bt}w \ra, \\
		A_{\al}(a_0) &= \la (a_1+a_{-1}) - (a_2+a_{-2}) +w\ra, \\
		A_{\tfrac{5\al-1}{8}}(a_0) &= \la a_1 - a_{-1}, a_2 -a_{-2} \ra 
	\end{align*}
\end{lemma}

\begin{lemma}
	There are no ideals of $A := 5\A(\al, \tfrac{5\al-1}{8})$ which contain axes.
\end{lemma}
\begin{proof}
There is one orbit of axes, so by Corollary \ref{axisideal}, $5\A(\al, \frac{5\al-1}{8})$ has no proper ideals that contain axes.
\end{proof}

\begin{computation}
	The determinant of the Frobenius form is
	\[
	-\tfrac{5^6 (3\al - 7)^5 (\al -\bt)^2}{2^{23}\bt}
	\]
\end{computation}

Clearly the algebra has a non-trivial radical if and only if $\al = \frac{7}{3}$ or $\ch \FF = 5$.  We consider first the case where $\ch \FF \neq 5$ and $\al = \frac{7}{3}$.

\begin{lemma}
Let $\ch \FF \neq 5$.  Then the radical of  $5\A(\al, \frac{5\al-1}{8})$ when $\al = \frac{7}{3}$ is
\[
\la a_0 - a_1, a_1 - a_2, a_2 - a_{-2}, a_{-2} - a_{-1}, w \ra 
\]
and this has no subideals.  We have $A/R \cong 1\A$.
\end{lemma}
\begin{proof}
When $\al = \frac{7}{3}$, the characteristic polynomial of the Frobenius form is $t^5(t-5)$, so the nullspace is at most $5$-dimensional in any characteristic.  On the other hand, we see that the difference of axes and $w$ are in the nullspace.  Hence, when $\al = \frac{7}{3}$, the radical $R$ is indeed $5$-dimensional in all characteristics.  As a $\Miy(A) = D_{10}$ module, since $\ch \FF \neq 5$, this decomposes as the sum of a trivial module and a $4$-dimensional irreducible.  Since
\[
a_1(a_0 - a_{-1}) = \tfrac{4}{3}(a_0 - a_{-1}) + \tfrac{2}{3}w
\]
and
\[
a_0w = a_1 + a_{-1} - (a_2 +a_{-2}) w
\]
we see that neither irreducible is a subideal.
\end{proof}

For the case where $\ch \FF = 5$ (and $\al$ may or may not be equal to $\frac{7}{3} = -1$), we refer to the later Table \ref{tab:IY5props}.

\begin{corollary}
The ideals and quotients of $5\A(\al, \frac{5\al-1}{8})$ are given in Table $\ref{tab:5Aprops}$.
\end{corollary}

\begin{table}[h!tb]
	\renewcommand{\arraystretch}{1.5}
	\centering
	\footnotesize
	\begin{tabular}{c|c|c|c}
		\hline
		Condition & Ideals & Quotients & Dimension \\
		\hline
		 $\al = \tfrac{7}{3}$ & $\la a_0 - a_1, a_1 - a_2, a_2-a_{-2},a_{-2}-a_{-1},w \ra$ & $1\A$ & 1 \\
		 \hline
 		 \multicolumn{4}{c}{\text{When $\ch \FF = 5$}} \\ \hline
 		 & $R = \la u_4, u_3, u_2, u_1, z\ra$ & $1\A$ & 1 \\
		 & $\la u_4, u_3, u_2, z\ra$ & $S(2)^\circ$ & 2 \\
		 & $\la u_4, u_3, u_2 - z\ra$ & $\Cl$ & 3 \\
		 & $\la u_4, u_3, u_2 +\frac{1}{2\al-1}z\ra$ & $\mathrm{W}_a(\al, \frac{1}{2}, 1)$ & 3 \\
		 & $\la u_4, u_3\ra$ & $\IY_5(\al,\frac{1}{2})^{\times\times}$ & 4 \\
		 & $\Ann A = \la u_4 \ra$ & $\IY_5(\al,\frac{1}{2})^\times$ & 5
	\end{tabular}\\
	\vspace{5pt}
	where $u_k := \sum_{i=0}^k (-1)^i \binom{k}{i}a_i$ and $z := 2(w-u_4)$
	\caption{Ideals and quotients of $5\A(\al, \frac{5\al-1}{8})$}\label{tab:5Aprops}
\end{table}

Now we turn to the idempotents.  Suppose that both $3\al-7$ and $\al^3-6\al^2+1$ are non-zero and 
$\FF$ contains roots $\rt_1$, $\rt_2$, $\rt_3$ and $\rt_4$ so that
\[
\rt_1^2 = \tfrac{5(1-5\al)}{3\al-7}, \qquad \rt_2^2 = \tfrac{5}{\al}, \qquad \rt_3^2 = 2\al-1, \qquad \rt_4^2 = \tfrac{1-5\al}{\al^3-6\al^2+1}
\]
and set
\begin{align*}
v_1 &= \tfrac{1}{10(\bt-\al)} \rt_1 w\\
v_2 &= \tfrac{1}{5\al}\rt_1\left( a_1+a_{-1} - (a_2+a_{-2}) + \tfrac{\al-1}{8(\bt-\al)}w \right) \\
v_3 &= \tfrac{1}{2}a_0 + \tfrac{\rt_2}{10}\left( a_1+a_{-1} - (a_2+a_{-2})+ \tfrac{\bt}{\bt-\al}w \right) \\
v_4 &= \tfrac{\rt_4}{5} \left( \tfrac{7\al+5}{16(\bt-\al)}a_0  - \tfrac{\al}{2(\bt-\al)}\sum_{i\neq 0} a_i
                          + \tfrac{\rt_2\rt_3}{2}( a_1+a_{-1} - (a_2+a_{-2}) + w) \right)
\end{align*}

\begin{computation}
Let $\ch \FF \neq 5$ and $3\al-7 \neq 0 \neq \al^3-6\al^2+1$ and suppose that $\FF$ contains the roots $\rt_1, \dots, \rt_4$. Then $5\A(\al, \frac{5\al-1}{8})$ has $63$ non-trivial idempotents as given in Table $\ref{tab:5Aidempotents}$.
\end{computation}

\begin{table}[h!tb]
	\begin{tabular}{cccc}
		\hline
		Representative & orbit size & length & comment \\
		\hline
		$\1$ & $1$ &  $\frac{2^3}{3\al + 1}$ &  \\
		$\tfrac{1}{2}\1 \pm v_1$ & $2$ &  $\frac{2^3}{3\al + 1}$ &  $4$ eigenvalues\\
		$a_0$ & $5$ &  $1$ & $\cM(\al,\bt)$ \\
		$\1-a_0$ & $5$ &  $\frac{7-3\al}{3\al + 1}$ & $\cM(1-\al, 1-\bt)$ \\
		$\tfrac{1}{2}\1 + v_2$ & $10$ &  $\frac{2^3}{3\al + 1}$ & \begin{tabular}[t]{c} $\tfrac{1}{2}\1 - v_2$ is in the same orbit\\  $6$ eigenvalues, graded\\ 
		         fusion law, $\tau_{\frac{1}{2}\1 + v_2} = \tau_0$ \end{tabular}\\
		$\tfrac{1}{2}\1 \pm v_3$ & $2\times10$ &  $\frac{3(\al+3)}{2(3\al + 1)}$, $\frac{7-3\al}{2(3\al + 1)}$ & \begin{tabular}[t]{c} $5$ eigenvalues,\\ graded 
		         fusion law,\\ $\tau_{\frac{1}{2}\1 \pm v_3} = \tau_0$ \end{tabular}  \\
		$\tfrac{1}{2}\1 \pm v_4$ & $2\times10$ &  $\frac{1}{3\al + 1} \left( \frac{5\al - 1}{2}\rt_4 + 2^2\right)$ & \begin{tabular}[t]{c} $6$ eigenvalues,\\ graded 
		         fusion law,\\ $\tau_{\frac{1}{2}\1 \pm v_4} = \tau_0$ \end{tabular}   \\
	\end{tabular}
	\caption{Idempotents of $5\A(\al,\frac{5\al-1}{8})$}\label{tab:5Aidempotents}
\end{table}


\section{Algebras with axet $X(6)$}

There are three algebras with axet $X(6)$, which are $6\A\left(\al, \frac{-\al^2}{4(2\al-1)}\right)$, $6\J(2\bt, \bt)$ and $6\Y( \frac{1}{2}, 2)$.  The computations in this section can be found in \cite[\texttt{X6 algebras.m}]{githubcode}.

\begin{table}[h!tb]
	\setlength{\tabcolsep}{4pt}
	\renewcommand{\arraystretch}{1.5}
	\centering
	\footnotesize
	\begin{tabular}{c|c|c}
		Type & Basis & Products \& form \\ \hline
		$6\A\left(\al, \frac{-\al^2}{4(2\al-1)}\right)$ & \begin{tabular}[t]{c} $a_{-2}$, $a_{-1}$, $a_0$,\\ $a_{1}$, $a_{2}$, $a_{3}$, \\ $c$, $z$  \end{tabular} &
		\begin{tabular}[t]{c}
			$a_i \cdot a_{i+1} = \frac{\bt}{2}(a_i + a_{i+1}-a_{i+2}-a_{i+3}$ \\ \hspace{60pt} $-a_{i-1}-a_{i-2}+c+z)$\\
			$a_i \cdot a_{i+2} = \frac{\al}{4}(a_i + a_{i+2}) + \frac{\al(3\al-1)}{4(2\al-1)}a_{i+4} - \frac{\al(5\al-2)}{8(2\al-1)}z $\\
			$a_i \cdot a_{i+3} = \frac{\al}{2}(a_i + a_{i+3} - c) $\\
			$a_i \cdot c = \frac{\al}{2}(a_i - c + a_{i+3}) $\\
			$a_i \cdot z = \frac{\al(3\al-2)}{4(2\al-1)}(2a_i - a_{i-2} - a_{i+2} + z) $\\
			$c^2 = c$, $c \cdot z = 0$, $z^2 = \frac{(\al+2)(3\al-2)}{4(2\al-1)}z$ \\
			$(a_i, a_i) = (c, c) = 1, \quad (a_i, a_{i+1}) = \frac{-\alpha^2(3\alpha - 2)}{4(2\alpha - 1)^2}$ \\ 
			$(a_i, a_{i+2}) = \frac{\alpha(21\alpha^2 - 18\alpha + 4)}{(4(2\alpha - 1))^2}$, \quad $(a_i, a_{i+3}) = (a_i,c) = \frac{\alpha}{2}$ \\ 
			$(a_i, z) = \frac{\alpha(7\alpha - 4)(3\alpha - 2)}{8(2\alpha - 1)^2}$ \\ 
			$(c, z) = 0, \quad (z, z) = \frac{(\alpha + 2)(7\alpha - 4)(3\alpha - 2)}{8(2\alpha - 1)^2}$ 
			\vspace{4pt}
		\end{tabular}\\
		\hline
		$6\J(2\bt, \bt)$ & \begin{tabular}[t]{c}$a_{-2}$, $a_{-1}$, $a_0$,\\ $a_{1}$, $a_{2}$, $a_{3}$, \\ $u$, $w$ \end{tabular} &
		\begin{tabular}[t]{c}
			 $a_i \cdot a_{i+1} = \frac{\beta}{2}\left(2\left(a_i + a_{i+1}\right) - w\right)$ \\ 
			 $a_i \cdot a_{i+2} = \frac{\beta}{2}\left(a_i + a_{i+2} - a_{i+4}\right)$ \\ 
			 $a_i \cdot a_{i+3} = \frac{\alpha}{2}\left(a_i + a_{i+3} - u\right)$ \\ 
			 $a_i \cdot u = \frac{\alpha}{2}\left(a_i + u - a_{i+3}\right)$ \\ 
			 $a_i \cdot w = \frac{\alpha}{2}\left(2a_i - a_{i-1} - a_{i+1} + w\right)$ \\ 
			 $u^2 = u, \quad uw = \beta u, \quad w^2 = (8 + 1)w - \beta u$ \\ 
			 $(a_i, a_i) = (u, u) = 1, \quad (a_i, a_{i+1}) = (u, w) = \beta$ \\ 
			 $(a_i, a_{i+2}) = \frac{\beta}{2}, \quad (a_i, a_{i+3}) = (a_i, u) = \frac{\alpha}{2}$ \\ 
			 $(a_i, w) = \alpha, \quad (w, w) = \beta + 2$  
			\vspace{4pt}
		\end{tabular}\\
		\hline
		$6\Y( \frac{1}{2}, 2)$ & \begin{tabular}[t]{c} $a_0, a_2, a_4, d, z$ \\ where \\ $a_i := a_{i+3}+d$ \end{tabular} &
		\begin{tabular}[t]{c}
			$a_i\cdot a_{i+2} = (a_i + a_{i+2} - a_{i+4})$\\
			$a_i \cdot d = \frac{1}{2}d + z$ \\
			$d^2 = -2z, \quad z\cdot a_i = z \cdot d = z^2 = 0$ \\ 
			$(a_i, a_j) = 1$ \\ 
			$(d, x) = (z, x) = 0, \quad \forall x \in A$ 
			\vspace{4pt}
		\end{tabular}
	\end{tabular}
	\caption{$2$-generated $\cM(\alpha, \beta)$-axial algebras on $X(6)$}\label{tab:2genMonster6}
\end{table}

\begin{proposition}
The axial subalgebras of the algebras are given in Table $\ref{tab:X6subalg}$.
\begin{table}[h!tb]
\centering
\renewcommand{\arraystretch}{1.3}
\begin{tabular}{ccc|ccc}
\hline
Algebra & $\lla a_0, a_2 \rra$ & conditions & Algebra &  $\lla a_0, a_3 \rra$ &conditions\\
\hline
$6\A(\al, \frac{-\al^2}{4(2\al-1)})$ & $3\A(\al, \frac{-\al^2}{4(2\al-1)})$ & $\al \neq \frac{2}{5}$    & $6\A(\al, \frac{-\al^2}{4(2\al-1)})$ & $3\C(\al)$ & \\
$6\A(\al, \frac{-\al^2}{4(2\al-1)})$ & $3\C(\frac{1}{5})$ & $\al = \frac{2}{5}$    &  &  & \\
$6\J(2\bt, \bt)$  & $3\C(\bt)$ &  & $6\J(2\bt, \bt)$  & $3\C(2\bt)$ & \\
$6\Y( \frac{1}{2}, 2)$ & $3\C(2)$ &    & $6\Y( \frac{1}{2}, 2)$ & $\Cl$ & \\
\hline
\end{tabular}
\caption{Subalgebras in algebras with axet $X(6)$}\label{tab:X6subalg}
\end{table}
\end{proposition}
\begin{proof}
These follow by computation from the structure constants and using Remark \ref{idJordan}.
\end{proof}

Again, we propose different bases compared to Yabe and Rehren, choosing natural elements in the axial subalgebras if possible.  For $6\A(\al, \frac{-\al^2}{4(2\al-1)})$, we choose $c$ to be the third axis in the axial subalgebra $\lla a_0, a_3 \rra \cong 3\C(\al)$ and then $z$ to be a scalar multiple of an idempotent in $\lla a_0, a_2 \rra \cong 3\A(\al, \frac{-\al^2}{4(2\al-1)})$ (note that we cannot choose an idempotent for some values of $\al$ it becomes a nilpotent).  For $6\J(2\bt,\bt)$ we use the same basis as in \cite{doubleMatsuo} with the axes reordered into the usual order.  Here $u$ is also the third axis in the axial subalgebra $\lla a_0, a_3 \rra \cong 3\C(2\bt)$.  For $6\Y(2, \frac{1}{2})$, there are six axes, but they span a $4$-dimensional space.  Since $a_i - a_{i+3}$ is the same for all axes, we choose this difference $d$ as a basis element along with the even axes and an element spanning the annihilator.

Considering dimensions, there is only isomorphisms between these algebras can be between $6\A(\al, \frac{-\al^2}{4(2\al-1)})$ and $6\J(2\bt, \bt)$.

\begin{lemma}
If $(\al, \bt) = (\frac{2}{5}, \frac{1}{5})$, then $6\A(\al, \frac{-\al^2}{4(2\al-1)}) \cong 6\J(2\bt, \bt)$.
\end{lemma}
\begin{proof}
For both algebras to have the same $\al$ and $\bt$, it is easy to see that we need $(\al, \bt) = (\frac{2}{5}, \frac{1}{5})$.  Then taking the basis $a, 0, \dots, a_5, c, w$ for $6\A(\frac{2}{5}, \frac{1}{5})$, where $w := 2(a_0+a_1) - 10a_0a_1$, we see that the structure constants are identical to those for $6\J(\frac{2}{5}, \frac{1}{5})$.
\end{proof}

Note that $\al = \frac{1}{5}$ is precisely when the subalgebra $\lla a_0, a_2 \rra$ for $6\A(\al, \frac{-\al^2}{4(2\al-1)})$ degenerates to become $3\C(\frac{1}{5})$ as required.

\subsection{$6\A\left(\al, \frac{-\al^2}{4(2\al-1)}\right)$}

For $6\A(\al, \frac{-\al^2}{4(2\al-1)})$, we require $\al  \neq \{0,1, \frac{4}{9}, -4 \pm 2\sqrt{5}\}$ so that $\{1,0,\al,\bt\}$ are distinct; we also need $\al \neq \frac{1}{2}$ for $\bt$ to be defined.

By computation, we see that the algebra has an identity unless $12\al^2-\al-2=0$ or $\al=\frac{2}{3}$ with	
	\[
	\1 = \tfrac{1}{12\al^2-\al-2} \left(2(2\al-1)\sum_{i=-2}^{3}a_i + (5\al-2)c + \tfrac{4(2\al-1)(3\al-1)}{3\al-2}z\right).
	\]

\begin{lemma}
	For the axis $a_0 \in 6\A\left(\al,\frac{-\al^2}{4(2\al-1)}\right)$, we have
	\begin{align*}
		A_1(a_0) &= \left\la a_0 \right\ra  \\
		A_0(a_0) &= \la \al a_0 - (a_3 + c), \\
		&\phantom{{}={} \la} \quad \tfrac{\al(3\al-2)(7\al-3)}{2\al-1} a_0 - 2(3\al-2)(a_2 +a_{-2}) - 2(5\al-2)z,\\
		&\phantom{{}={} \la} \quad (3\al-2)(a_1+a_{-1} - \tfrac{\al}{2(2\al-1)}a_3) + \al z \ra\\
		A_\al(a_0) &= \la \al a_0 - 2(2\al-1)(a_2 +  a_{-2} -z), a_3 - c\ra \\
		A_{\frac{-\al^2}{4(2\al-1)}}(a_0) &= \la a_1 - a_{-1}, a_2-a_{-2} \ra 
	\end{align*}
\end{lemma}

\begin{lemma}
	There are no ideals of $A := 6\A\left(\al,\frac{-\al^2}{4(2\al-1)}\right)$ which contain axes.
\end{lemma}
\begin{proof}
Since there are two orbits of axes in $A$ and $(a_0, a_3) = \tfrac{\al}{2} \neq 0$, by Corollary \ref{axisideal}, $6\A\left(\al,\tfrac{-\al^2}{4(2\al-1)}\right)$ has no proper ideals which contain axes.
\end{proof}

\begin{computation}
The determinant of the Frobenius form is 
\[
\frac{- (\alpha - 1)^3  (3\alpha - 2)  (7\alpha - 4)^5  (12\alpha^2 - \alpha - 2)  (\alpha^2 + 4\alpha - 2)^4}{2^{15}(2\alpha - 1)^{11}}
\]
\end{computation}

So the algebra has a non-trivial radical when one of $\al = \tfrac{2}{3}$, $\al = \tfrac{4}{7}$, $12\al^2 - \al - 2 = 0$, or $\al^2 + 4\al - 2 = 0$. In order to get a clean case distinction, we first see when more than one of these conditions can be satisfied at once.

\begin{lemma}\label{6Acases}
No two of the conditions $\al = \tfrac{2}{3}$, $\al = \tfrac{4}{7}$, $12\al^2 - \al - 2 = 0$, or $\al^2 + 4\al - 2 = 0$ are simultaneously satisfied for $6\A(\al, \tfrac{-\al^2}{4(2\al-1)})$ unless $\ch \FF = 5$. In this case, the only conditions which can be simultaneously satisfied are 
\begin{enumerate}
	\item $\al = \tfrac{2}{3} = -1$ and $0 = \al^2 + 4\al - 2 =(\al +1)(\al - 2)$, or
	\item $\al = \tfrac{4}{7} = 2$ and $0 = \al^2 + 4\al - 2= (\al +1)(\al - 2)$.
\end{enumerate}
\end{lemma}

\begin{proof}
	We first suppose that $\al = \tfrac{2}{3}$; so $\ch \FF \neq 3$.  If $\al = \tfrac{4}{7}$ too, then $\ch \FF \neq 7$ and we need $\tfrac{2}{3}-\tfrac{4}{7} = \tfrac{14-12}{21} = \tfrac{2}{21} = 0$, which only holds when $\ch \FF = 2$, a contradiction. If $0 = \al^2 + 4\al - 2 = \tfrac{10}{9}$, then this can occur simultaneously only when $\ch \FF = 5$. In $\ch \FF = 5$, we have $\al^2 + 4\al - 2 = \al^2 - \al - 2 = (\al+1)(\al-2)$, and we see that $\al = \tfrac {2}{3} = -1$ and $\bt = 3$. For $0 = 12\al^2 - \al - 2 = \tfrac{8}{3}$, this can occur simultaneously only when $\ch \FF = 2$, a contradiction.

Now suppose that $\al = \tfrac{4}{7}$, then $\ch \FF \neq 3, 7$. For $0=\al^2 + 4\al- 2 = \tfrac{30}{49}$, this can only occur simultaneously when $\ch \FF = 5$. In $\ch \FF = 5$, we have $\al^2 + 4\al - 2 = (\al+1)(\al-2)$. We see that $\al= \tfrac{4}{7}=2$ and $\bt = 3$. For $12\al^2 -\al - 2 = \tfrac{66}{49} = 0$, this can occur simultaneously only when $\ch \FF = 11$,  but when $\ch \FF = 11$ then $\bt = 1$, a contradiction.
	
Finally, suppose that $12\al^2 -\al - 2 = 0$ and $\al^2 + 4\al - 2 =0$, then $0 = 12(\al^2 + 4\al - 2)-(12\al^2 -\al - 2) = 49 \al -22$, so $ \ch \FF \neq 7,11$. For $\al = \frac{22}{49}$ to be a root of both polynomials we require $\ch \FF = 3$, but then $\al = \tfrac{22}{49} =1$, a contradiction.

Thus, the only simultaneous occurrences happen when $\ch \FF = 5$.
\end{proof}

We proceed to calculate the nullspace and find the subideals in each case.  We start by assuming that the characteristic is not $5$.  For $\al = \frac{2}{3}$, finding the ideals is straightforward.

\begin{lemma}\label{6Aideal 2/3}
If $\ch \FF \neq 5$ and $\al = \frac{2}{3}$, then the radical is $1$-dimensional, equal to the annihilator and the quotient is $6\A(\frac{2}{3}, \frac{-1}{3})^\times$.
\end{lemma}

\begin{proof}
When $\al = \tfrac{2}{3}$, by computation, the characteristic polynomial for the Frobenius form is $t(t-\frac{5}{3})(t - \frac{5}{6})^2(t - \frac{1}{6})^2(t^2 - \frac{10}{3}t + \frac{5}{3})$.  Clearly, if the characteristic is not $5$, then the $0$-eigenspace is at most $1$-dimensional.  On the other hand, $z$ is in the radical $R$, hence $R$ is always $1$-dimensional.  From the multiplication table, $R = \la z \ra = \Ann A$.  Since the quotient is of Monster type $(\tfrac{2}{3}, \tfrac{-1}{3})$, it is not isomorphic to any other example; it is $6\A(\frac{2}{3}, \frac{-1}{3} )^\times$.
\end{proof}

Now we turn to $\al = \frac{4}{7}$.

\begin{lemma}\label{6Aideal 4/7}
Suppose that $\ch \FF \neq 5$ and $\al = \frac{4}{7}$, then the radical of $A := 6\A(\frac{4}{7}, -\frac{4}{7})$ is
\[
R = \la a_0 - a_2, a_0 - a_{-2}, a_3 - a_{-1}, a_3 - a_1, z \ra
\]
and there are no non-trivial proper subideals of $R$.  The quotient $A/R$ is $3\C(\frac{4}{7})$.
\end{lemma}
\begin{proof}
Note that $\al = \frac{4}{7}$, the characteristic of the field is not $3$, or $7$.  In addition, the characteristic is not $11$ either as in this case $\bt = -\frac{4}{7} = 1$, a contradiction.  By computation, we find that the characteristic polynomial of the Frobenius form is $t^5(t - \frac{15}{7})(t^2 - \frac{34}{7}t + \frac{165}{49})$.  Since the characteristic is not $3$, $7$, or $11$ and is not $5$ by assumption, $15$ and $165 = 3\cdot 5 \cdot 11$ are not zero.  Hence the nullspace $R$ is at most $5$-dimensional in any characteristic other than $5$.  On the other hand, it is easy to see that $a_0 - a_2, a_0 - a_{-2}, a_3 - a_{-1}, a_3 - a_1, z \in R$ and hence the the radical is indeed $5$-dimensional.

Since $\Miy(A) \cong S_3$, $R$ decomposes as the direct sum of irreducible $S_3$-modules.  It is easy to see that it is the sum of a trivial module $\la z \ra \cong U_1$ and two $2$-dimensional irreducible modules $\langle a_0 - a_2, a_0 - a_{-2} \rangle, \langle a_3 - a_{-1}, a_3 - a_1 \rangle \cong U_2$.  Let $I \trianglelefteq A$ be a non-trivial proper subideal of $R$. The only possible $G$-mod structure for $I$ is $U_1$, $U_2$, $U_1 \oplus U_2$, or $U_2 \oplus U_2$.

First suppose that $z \in I$.  Then we claim that $I = R$.  Indeed,
\begin{equation}\label{6A 4/7 eqn z}
-\tfrac{7}{2} a_0 \cdot z = 2a_0 - a_{-2} - a_2 + z = (a_0 - a_2) + (a_0 -a_{-2}) + z
\end{equation}
and so, as $\ch \FF \neq 7$, $(a_0 - a_2) + (a_0 -a_{-2}) \in I$.  There exists $\phi \in \Aut(A)$, such that $\phi \colon a_i \mapsto a_{i+3}$ and fixes $z$; indeed this is $\phi = \half \tau_2$.  Now, by applying $\phi$ to both side above, we get $-\tfrac{7}{2} a_3 \cdot z = (a_3 -a_{-1}) + (a_3 - a_1) +z \in I$. Hence $I = R$.  Therefore, the trivial module $\la z \ra$ cannot be a constituent of any proper ideal $I$ and in particular, $I$ cannot be $1$-dimensional.

So now suppose that $I$ contains a $2$-dimensional module $\la u , v \ra$ isomorphic to $U_2$.  By considering the action of $\Miy(A) \cong S_3$ on $\langle a_0 - a_2, a_0 - a_{-2} \rangle, \langle a_3 - a_{-1}, a_3 - a_1 \rangle$, we see any isomorphism between these is a scalar multiple of the map $\phi$.  So, there exists $\lm, \mu \in \mathbb{F}$, with $\lm$ and $\mu$ not both zero, such that $u := \lambda(a_0 - a_2) + \mu(a_3 - a_{-1})$ and $v :=  \lambda(a_0 - a_{-2}) + \mu(a_3 - a_1)$.  Now
\begin{align*}
7a_0 u &= a_0 \big( \lambda(a_0 - a_2) + \mu(a_3 - a_{-1}) \big) \\
&=  7\lambda a_0 - \lambda \left(a_0 + a_2 + 5a_{-2} - 3z \right) + 2\mu \left(a_0 + a_3 - c\right) \\
&\phantom{{}={}} + 2\mu \left(a_0 + a_{-1} - a_1 - a_2 - a_{-2} + c + z\right) \\
&= (6\lm + 4 \mu)a_0 +2\mu (a_{-1} - a_1) - (\lm + 2\mu) a_2 \\ 
&\phantom{{}={}} - (5\lm + 2\mu) a_{-2} + (3\lm + 2\mu) z
\end{align*}
In particular, if $3\lambda + 2\mu \neq 0$, then $z \in I$, and so by the previous argument $I = R$.  Since $I$ is a proper ideal of $R$, we must have $3\lambda + 2\mu = 0$.
Now, note that applying $\phi$ to the above maps $a_0$ to $a_3$ and switches the roles of $\lambda$ and $\mu$.  Hence we also require $3\mu + 2\lm = 0$.  We cannot simultaneously have $3\lm + 2\mu = 0$ and $3\mu + 2\lm = 0$ unless $\ch \FF = 5$ and $\lm = \mu$.  Therefore $R$ has no non-trivial proper subideals.

Note that the quotient $A/R$ is generated by two axes $\bar{a}_0 = \bar{a}_2 = \bar{a}_{-2}$ and $\bar{a}_1 = \bar{a}_3 = \bar{a}_{-1}$ and the $\bt$-eigenspace of $a_0$ is contained in $R$.  It is then easy to see that $A/R$ is $2$-generated of Jordan type $\frac{4}{7}$ and hence is isomorphic to $3\C(\frac{4}{7})$.
\end{proof}

We now consider $\al^2+4\al-2=0$, which has the roots $\al = -2 \pm \sqrt{6}$.

\begin{lemma}\label{6Aideal sqrt6}
Suppose that $\ch \FF \neq 5$ and $\al = -2 \pm \sqrt{6}$, then the radical of $A := 6\A(\al, \frac{1}{2})$ is
\begin{align*}
	R &= \la (a_0 -a_2)+(a_3-a_{-1}), (a_0 - a_{-2}) +  (a_3-a_1), \\
	&\phantom{{}={} \la} \quad (a_0 + a_2 + a_{-2}) - (a_1 + a_3 + a_{-1}), \sum_{i=-2}^3 a_i -3\al c +2(\al-1)z \ra
\end{align*}
and there are no non-trivial proper subideals of $R$.  The quotient $A/R$ is $\IY_3(\al, \frac{1}{2}, \frac{1}{2})$.
\end{lemma}
\begin{proof}
Let $\zeta$ be a square root of $6$; then $\al = -2+\zeta$.   By assumption, the characteristic of the field is not $2$ or $5$ and it is also not $3$ as then $\al = -2 + \zeta = -2 = 1$, a contradiction.  For this case, we perform Gaussian elimination on the Frobenius form thought of over the function field $\Q(\al, \bt)$ taking care with our scalars.  That is, we do not allow any denominators, except those whose prime factors are $2$, $3$, or $5$.  Similarly, we do not multiply any row except by numbers whose prime factors are $2$, $3$, or $5$.  Let $M$ be the resulting matrix.  Thus, the nullspace of $\phi(M)$, where $\phi \colon \Z \to \Z/p\Z$, for a prime $p$, is the natural ring homomorphism, is the nullspace of the Frobenius form of $A$ over $\FF_p(\al, \bt)$.  By computation and inspection, we find that the nullspace of $M$ and hence the Frobenius form, is $4$-dimensional in all characteristics.

Since $\Miy(A) \cong S_3$, $R$ decomposes as the direct sum of irreducible $S_3$-modules.  It is easy to see that it is the sum of a $2$-dimensional irreducible module, spanned by  $r_1 := (a_0 -a_2)+(a_3-a_{-1})$ and $r_2 := (a_0 - a_{-2}) +  (a_3-a_1)$, and two trivial modules spanned by $r_3 :=(a_0 + a_2 + a_{-2}) - (a_1 + a_3 + a_{-1})$ and $r_4  :=  \sum_{i=-2}^3 a_i -3\al c +2(\al-1)z$.

Let $I \unlhd A$ be a proper subideal of $R$.  Suppose that $I$ contains a trivial module, spanned by $\lm r_3 + \mu r_4$. By computation, we see that
\begin{align*}
2r_3^2 &= \zeta r_4 \\
2r_3 r_4 &= 3(9\zeta - 22) r_3 \\
2r_4^2 &= (5\zeta - 12) r_4
\end{align*}
Note that $9\zeta - 22$ and $5\zeta - 12$ are both never zero as the characteristic is not $2$, or $3$.  Also note that $3(9\zeta - 22)$ is never equal to $5\zeta -12$ as $\ch \FF \neq 3$.  Hence $r_3, r_4 \in (\lm r_3 + \mu r_4) \subseteq I$.  Now, by computation,
\[
2(a_0-a_2)r_3 = (1-\al) r_1
\]
so $r_1, r_2 \in I$ and hence $I = R$, a contradiction.  Hence $I$ contains no trivial modules and so the only option for a proper ideal is $\la r_1, r_2 \ra$.  However, by computation, we see that
\begin{align}
	6 r_1^2 & = -2(\al-1)r_1 + 4(\al-1)r_2 + ( 6 + \zeta) r_3 \label{sqrt6 eqn1}\\
	4(a_0 - a_3)r_1 &= (4 - \zeta) r_4 \label{sqrt6 eqn2}
\end{align}
We claim that unless $\ch \FF = 5$ and $\zeta =-1$, then at least one of the above has a non-zero coefficient for $r_3$, or $r_4$.  Indeed, suppose that both $6 +\zeta$ and $4 -\zeta$ are zero.  Then, $\zeta = 4 = -6$ and so $\ch \FF = 5$ and $\zeta = -1$.  Hence, we see that unless $\ch \FF = 5$ and $\zeta = -1$, there are no subideals of $R$.

Now, let $B = A/R$.  Set $z_1 = \bar{c}$, $z_2 := \frac{2}{2\zeta +3} \bar{z}$, $e := \frac{2}{3}(\bar{a}_0 - \bar{a}_2 + \bar{a}_0 - \bar{a}_{-2})$ and $f := \frac{2}{3}( -(\bar{a}_0 - \bar{a}_2) + 2(\bar{a}_0 - \bar{a}_{-2})$.  Then, one can check that $B \cong \IY_3(\al, \frac{1}{2}, \frac{1}{2})$ as claimed.
\end{proof}


\begin{lemma}
Suppose that $\ch \FF = 5$ and $\al = \tfrac{2}{3} = -2+\sqrt{6} = -1$, then the radical of $A := 6\A(-1, \frac{1}{2})$ is	
\begin{align*}
	R &= \la z, (a_0 -a_2)+(a_3-a_{-1}), (a_0-a_{-2}) + (a_3-a_1) \\
	&\phantom{{}={}\la} \quad (a_0+a_2+a_{-2})-(a_1+a_3+a_{-1}), \sum_{i=-2}^{3}a_i +3c +z
	 \ra
\end{align*}
and there are two non-trivial proper subideals of $R$.  These are the ideals from Lemmas $\ref{6Aideal 2/3}$ and $\ref{6Aideal sqrt6}$, which also exist in characteristic $5$ and $R$ is the sum of these.  The quotients are given in Table $\ref{tab:6Aprops}$.
\end{lemma}

\begin{proof}
Set $\zeta = 1$ which is a square root of $6 = 1$.  Then, we have $\al = -2+\zeta = -1 = \frac{2}{3}$. We find by computation that the radical $R$ is $5$-dimensional. It is easy see that just as in Lemma \ref{6Aideal 2/3}, $I_1:=\Ann(A) = \la z \ra$ is a $1$-dimensional subideal of $R$.  Given that $\Miy(A) = S_3$ and $\ch \FF = 5$, by Maschke's Theorem, $R$ decomposes as a direct sum, so it must have a $4$-dimensional submodule. Taking $I_2 := \la (a_0 -a_2)-(a_3-a_{-1}), (a_0-a_{-2}) + (a_3-a_1), (a_0+a_2+a_{-2})-(a_1+a_3+a_{-1}),\sum_{i=-2}^{3}a_i +3c +z \ra $, by computation we see that it is a subideal. Observe that, $I_2$ is isomorphic to the radical in Lemma $\ref{6Aideal sqrt6}$.  Note that, as $\zeta = 1 \neq -1$, Equations \eqref{sqrt6 eqn1} and \eqref{sqrt6 eqn2} both still have a non-trivial coefficient for $r_3$ and $r_4$, hence the subideal analysis continues to hold in our case and $I_2$ does not have any subideals. Hence $I_1$ and $I_2$ are the only subideals of $R$.
	
	The quotient $B_1 := A/I_1$ is a $2$-generated axial algebra of Monster type $(\frac{2}{3},\frac{-1}{3})$, that has dimension $7$; it is $6\A(\frac{2}{3},\frac{-1}{3})^\times$.  Consider the quotient $B_2 := A/I_2$.  Set $z_1 = \bar{c}$, $n := \frac{1}{2} \bar{z}$, $e := \frac{2}{3}(\bar{a}_0 - \bar{a}_2 + \bar{a}_0 - \bar{a}_{-2})$ and $f := \frac{2}{3}\big( -(\bar{a}_0 - \bar{a}_2) + 2(\bar{a}_0 - \bar{a}_{-2})\big)$.  Then, one can check that $B_2 \cong \IY_3(-1, \frac{1}{2}, \frac{1}{2})$ as claimed.  This has a $1$-dimensional quotient spanned by the annihilator $\la n \ra = \la \bar{z} \ra$ and so $A/R \cong \IY_3(-1, \frac{1}{2}, \frac{1}{2})^\times$.
\end{proof}

\begin{lemma}
Suppose that $\ch \FF = 5$ and $\al = 2 = \tfrac{4}{7} = -2+\zeta$, where $\zeta = -1$ is a square root of $6=1$.  The radical of $A := 6\A(2, 3)$ is
\[
	R = \la a_0 -a_1, a_1-a_2, a_2-a_3, a_3-a_{-2}, a_{-2} -a_{-1}, \sum_{i=-2}^3 a_i -c, z  \ra 
\]
and there are three non-trivial proper subideal of $R$.  Namely, the ideals $I_{\frac{4}{7}}$ and $I_{-2+\zeta}$ from Lemmas $\ref{6Aideal 4/7}$ and $\ref{6Aideal sqrt6}$, respectively, and their intersection, which is $J :=  \la (a_0 -a_2)+(a_3-a_{-1}), (a_0 - a_{-2}) +  (a_3-a_1) \ra$.  The quotients are given in Table $\ref{tab:6Aprops}$.
\end{lemma}
\begin{proof}
By inspection, the radical $R$ is $7$-dimensional and contains the above elements.  Considering $\al$ to be $\frac{4}{7}$ and $-2+\zeta$, respectively, we get the ideals in Lemmas \ref{6Aideal 4/7} and \ref{6Aideal sqrt6}, which are also ideals in characteristic $5$.

Let $I$ be a proper subideal of $R$.  Note that, as a $\Miy(A)$-module $R$ decomposes as the sum of two $2$-dimensional irreducible modules, $M_4 := \la  a_0-a_2, a_0-a_{-2}\ra\oplus \la a_3-a_{-1}, a_3-a_1\ra$ and three trivial modules.  Suppose that $I$ contains a constituent which is a $2$-dimensional irreducible.  By Lemma \ref{6Aideal 4/7}, the module $M_4$ is contained in $I_{\frac{4}{7}}$, which is $5$-dimensional.  Following the proof of that lemma, either $I_{\frac{4}{7}} \subseteq I$, or $I$ contains a single $2$-dimensional irreducible $\la (a_0-a_2) + (a_3-a_{-1}), (a_0-a_{-2}) + (a_3-a_1) \ra$ corresponding to $\lm = \mu$ in that proof.  This is precisely the intersection $I_{\frac{4}{7}} \cap I_{-2+\zeta}$ and so is an ideal.

Now suppose that $I$ contains a constituent which is a trivial module.  The trivial constituents of $R$ are $\la z \ra \oplus \la r_3, r_4 \ra$ where $r_3:=(a_0 + a_2 + a_{-2}) - (a_1 + a_3 + a_{-1})$, $r_4 := \sum_{i=-2}^3 a_i - c +2z$.  Let $\lm z + \mu r_3 + \nu r_4 \in I$.  If $\lm \neq 0$, then by Equation \eqref{6A 4/7 eqn z} on page \pageref{6A 4/7 eqn z}, $I_{\frac{4}{7}} \subseteq I$ and hence $z \in I$.  So we may now suppose that $\lm =0$.  By the proof of Lemma \ref{6Aideal sqrt6}, if either $\mu$, or $\nu$ are non-zero, then $I_{-2+\zeta} \subseteq I$ and $r_3, r_4 \in I$.  Therefore, $I_{\frac{4}{7}}$, $I_{-2+\zeta}$ and $Y$ are the only proper subideals of $R$.

The quotient $A/J$ is $6$-dimensional and of Monster type $(2, \frac{1}{2})$ and clearly $(a_0-a_{-2}) + (a_3-a_1) = a_0-a_1 + a_3-a_{-2}$ is zero in the quotient.  One can check that $A/J$ is isomorphic to the quotient of the Highwater algebra $\cH$ by the ideal $\la a_0-a_1+a_3-a_4\ra$.
\end{proof}

Finally, we turn to the case where $12\al^2-\al-2=0$, which has roots $\al = \frac{1 \pm \sqrt{97}}{24}$.

\begin{lemma}
Suppose that $\al = \frac{1 \pm \sqrt{97}}{24}$, then the radical of $A := 6\A(\al, \bt)$ is
\[
R = \left\la \sum_{i=-2}^{3}a_i - \tfrac{1\pm \sqrt{97}}{8}c + \tfrac{1\mp \sqrt{97}}{8}z \right\ra
\]
and the quotient $A/R$ is $6\A(\frac{1\pm \sqrt{97}}{24}, \frac{53 \pm 5\sqrt{97}}{192} )^\times$.
\end{lemma}

\begin{proof}
The proof is analogous to Lemma \ref{6Aideal 2/3}.
\end{proof}

\begin{table}[h!tb]
	\setlength{\tabcolsep}{4pt}
	\renewcommand{\arraystretch}{1.5}
	\centering
	\footnotesize
	\begin{tabular}{c|c|c|c}
		Condition & Ideal & Quotients & Dimension \\ \hline
		
		$\al = \tfrac{2}{3}$ &  $\Ann A = \la z \ra$ &  $6\A(\frac{2}{3}, \frac{-1}{3})^\times$   & $7$\\

		$\al = \tfrac{4}{7}$  & \begin{tabular}[t]{c} $\la a_0 -a_2, a_0-a_{-2}, a_3-a_{-1}, a_3 - a_1, z \ra $ \end{tabular} & $3\C(\tfrac{4}{7})$ & $3$ \\
		
		 $\al = -2 \pm \sqrt{6}$  & \begin{tabular}[t]{c} $\la (a_0 -a_2)+(a_3-a_{-1}),$\\$ (a_0-a_{-2})+(a_3-a_1),$ \\$
		 (a_0+a_2+a_{-2}) - (a_1+a_3+a_{-1})$ \\ $\sum_{i=-2}^3 a_i -3\al c + 2(2\al-1)z \ra$ \end{tabular}  & $\IY_3(-2\pm\sqrt{6}, \tfrac{1}{2}, \tfrac{1}{2})$ & $4$\\ 
		
	$\al = \frac{1\pm \sqrt{97}}{24}$ & $\la \sum_{i=-2}^{3}a_i - \frac{1\pm \sqrt{97}}{8}c + \frac{1\mp \sqrt{97}}{8}z \ra$ &  $6\A(\frac{1\pm \sqrt{97}}{24}, \frac{53 \pm 5\sqrt{97}}{192} )^\times$  & $7$ \\
		 
		  \hline
		 \multicolumn{4}{c}{\text{In addition, when $\ch \FF = 5$}} \\ \hline 
		\begin{tabular}[t]{c}  $\al = -1 = \frac{2}{3}$ \\
			$ \,\;\,\;\;\; =  -2+\sqrt{6} $ \end{tabular}
			 & \begin{tabular}[t]{c} $ \la z,(a_0 -a_2)+(a_3-a_{-1}), $\\$(a_0-a_{-2})+(a_3-a_1), $ \\ $ ( a_0+a_2+a_{-2})-(a_1+a_3+a_{-1}),$ \\  $\sum_{i=-2}^{3}a_i +3c +z\ra$ \end{tabular} &  $\IY_3(-1, \frac{1}{2}, \frac{1}{2})^\times$   & $3$\\
		\begin{tabular}[t]{c} $\al =\;\; 2\; = \frac{4}{7}$ \\
			$ \,\;\,\;\;\; =  -2-\sqrt{6} $ \end{tabular}  & \begin{tabular}[t]{c} $\la a_0 -a_1, a_1-a_2, a_2-a_3, a_3-a_{-2},$ \\ $ a_{-2} -a_{-1}, \sum_{i=-2}^3 a_i -c, z \ra $ \end{tabular} & $1\A$ & $1$ \\
\begin{tabular}[t]{c} $\al =\;\; 2\; = \frac{4}{7}$ \\
			$ \,\;\,\;\;\; =  -2-\sqrt{6} $ \end{tabular}  & \begin{tabular}[t]{c} $\la (a_0 -a_2)+(a_3-a_{-1}),$\\$ (a_0-a_{-2})+(a_3-a_1) \ra$  \end{tabular} & $\cH/\la a_0-a_1+a_3-a_4\ra$ & $6$ \\

	\end{tabular}
	\caption{Ideals and quotients of $6\A(\al, \frac{-\al^2}{4(2\al-1)})$}\label{tab:6Aprops}
\end{table}

We now turn to the idempotents.  Consider the algebra over the algebraic closure of the function field $\FF(\al)$ and define
\begin{align*}
z' &:= \tfrac{2^2(2\al-1)}{(3\al-2)(\al+2)}z \\
s &:= \tfrac{2(2\al-1)}{12\al^2-\al-2}\left( \sum_{i=-2}^3 a_i -3\al c -\tfrac{6\al}{\al+2}z \right)\\
u_1 &:= \tfrac{1}{7\al^2+\al-2}\left( (2-5\al)(a_0+a_4 -\al c -2\al z') + (12\al^2-\al-2)s \right) \\
u_2 &:= \tfrac{-\al}{(2\al - 1)(\al + 2)}a_0 + \tfrac{2}{(\al+2)}(a_1+a_{-1}) - \tfrac{2\bt}{(\al)}z'\\
u_3 &:= c+ u_1 -u_2 
\end{align*}

\begin{computation}\label{6Aidempotents}
Over the algebraic closure of $\FF(\al)$, $6\A(\al, \frac{-\al^2}{4(2\al-1)})$ has $255$ non-trivial idempotents.  Under the action of the group $\la \tau_0, \half \ra \cong D_{12}$, there are $8$ orbits of length $1$, $4$ of length $2$, $8$ of length $3$, $28$ of length $6$ and $4$ of length $12$.  Some of these are given in Table $\ref{tab:6Aidempotents}$.
\end{computation}

\begin{proof}
Using Magma, \texttt{IsZeroDimension} shows that the ideal is $0$-dimensional and so, by Bezout's Theorem, there are 255 non-trivial idempotents.  By including relations from the action of the group, we can find all those except those in an orbit of size $12$.  See \cite[\texttt{6A find idempotents.m}, \texttt{6A idempotents.m}]{githubcode} for all computations.
\end{proof}

\begin{table}[h!tb]
	\centering
	\footnotesize
	\begin{tabular}{cccc}
		\hline
		Representative & orbit size & length & comment \\
		\hline
		$c$ & 1 & 1 & \\
		$z'$ & 1 & $\frac{2(7\al-4)}{(3\al-2)(\al+2)}$ & \\
		$s$ & 1 & $-\frac{12(\al-1)(\al^2+4\al-2)}{(\al+2)(12\al^2-\al-2)}$ & \\
		$c+z'$ & 1 & $\frac{3(\al^2 + 6\al - 4)}{(\al^2 + 4)(9\al - 4)}$ & \\
		$c+s$ & 1& $\frac{-13\al^2 + 68\al - 28}{2\al^3 + 23\al^2 - 4\al - 4}$ & \\
		$s+z'$ & 1 & $\frac{6(-\al^3 + 78\al^2 - 67\al + 16)}{(36\al^3 - 27\al^2 - 4\al + 4)}$ & \\
		$\1$ & $1$ &  $\frac{3(43\al^2 - 46\al + 12)}{(3\al - 2)(12\al^2 - \al - 2)}$ & $\1 = c+z'+s$ \\
		$\1_{\lla a_0,a_2 \rra}$ & 2 & $\frac{71\al^2 - 76\al + 20}{(3\al - 2)(7\al^2 + \al - 2)}$ & \\
		$\1 - \1_{\lla a_0,a_2 \rra}$ & 2 & $\frac{(17\al - 8)(\al^2 + 4\al - 2)}{(12\al^2 - \al - 2)(7\al^2 + \al - 2)}$ & \\
		$\1_{\lla a_0,a_2 \rra} - z'$ & 2 & $\frac{-3(3\al^2 - 10\al + 4)}{(\al + 2)(7\al^2 + \al - 2)}$ & \\
		$\1 -(\1_{\lla a_0,a_2 \rra}-z')$ & 2 & $\frac{3(409\al^5 - 118\al^4 - 204\al^3 - 48\al^2 + 128\al - 32)}{(3\al - 2)(\al+2)(12\al^2 - \al - 2)(7\al^2 + \al - 2)}$ & \\
		$\1_{\lla a_0,a_4 \rra}$ & 3 & $\frac{3}{\al+1}$ & \\
		$\1 - \1_{\lla a_0,a_3 \rra}$ & 3 & $\frac{3(7\al - 4)(\al^2 + 4\al - 2)}{(12\al^2 - \al - 2)(3\al - 2)(\al+1)}$ & \\
		$\1_{\lla a_0,a_3 \rra}-c$ & 3 & $\frac{2-\al}{\al+1}$ & \\
		$\1 - (\1_{\lla a_0,a_3 \rra}-c)$ & 3 & $\frac{36\al^4+30\al^3+ 41\al^2 - 90\al+28}{(3 \al - 2)(\al+1)(12\al^2 - \al - 2)}$ & \\
		$u_1$ & 3 & $\frac{(4-7\al)(3\al^2 - 10\al + 4)}{(3\al - 2)(7\al^2 + \al - 2)}$ & \\
		$\1 - u_1$ & 3 & $\frac{84\al^4 + 22\al^3 + 21\al^2 - 66\al + 20}{(12\al^2 - \al - 2)(7\al^2 + \al - 2)}$ & \\
		$u_1+c$ & 3 & $\frac{71\al^2 - 76\al + 20}{(3\al - 2)(7\al^2 + \al - 2)}$ & \\
		$\1 - (u_1+c)$ & 3 & $\frac{(17\al-8)(\al^2 + 4\al - 2)}{(12\al^2 - \al - 2)(7\al^2 + \al - 2)}$ & \\
		$a_0$ & $6$ & $1$ &  \\
		$\1 - a_0$ & $6$ & $\frac{-2(18\al^3 - 78\al^2 + 67\al - 16)}{(3\al - 2)(12\al^2 - \al - 2)}$ & \\
		$\1_{\lla a_0,a_2 \rra}-a_0$ & $6$ & $\frac{(4-7\al)(3\al^2 - 10\al + 4)}{(3\al - 2)(7\al^2 + \al - 2)}$ & \\
		$\1_{\lla a_0,a_3 \rra}-a_0$ & $6$ & $\frac{2-\al}{\al +1}$ & \\
		$\1-(\1_{\lla a_0,a_2 \rra}-a_0)$ & $6$ & $\frac{84\al^4 + 22\al^3 + 21\al^2 - 66\al + 20}{(12\al^2 - \al - 2)/(7\al^2 + \al - 2)}$ & \\
		$\1-(\1_{\lla a_0,a_3 \rra}-a_0)$ & $6$ & $\frac{36\al^4 + 30\al^3 + 41\al^2 - 90\al + 28}{(3\al - 2)/(\al + 1)/(12\al^2 - \al -2)}$ & \\
		$u_2$ & $6$ & $\frac{2(7\al - 4)}{(3\al - 2)/(\al + 2)}$ & \\
		$\1 - u_2$ & $6$ & $\frac{-(13\al^2 - 68\al + 28)}{(\al + 2)(12\al^2 - \al - 2)}$ & \\
		$\1_{\lla a_0,a_2 \rra}-u_2$ & $6$ & $\frac{-3(3\al^2 - 10\al + 4)}{(\al + 2)(7\al^2 + \al - 2)}$ & \\
		$\1-(\1_{\lla a_0,a_2 \rra}-u_2)$ & $6$ & $\frac{3(409\al^5 - 118\al^4 - 204\al^3 - 48\al^2 + 128\al - 32)}{(3\al - 2)(\al +2)(12\al^2 - \al - 2)(7\al^2 + \al - 2)}$ & \\
		$u_2+a_2$ & $6$ & $\frac{3(\al^2 + 6\al - 4)}{(3\al - 2)(\al + 2)}$ & \\
		$\1-(u_2+a_2)$ & $6$ & $\frac{12(1-\al)(\al^2 + 4\al - 2)}{(\al + 2)(12\al^2 - \al - 2)}$ & \\
		$u_3$ & $6$ & $\frac{-3(3\al^2 - 10\al + 4)}{(\al + 2)(7\al^2 + \al - 2)}$ & \\
		$\1-u_3$ & $6$ & $\frac{3(409\al^5 - 118\al^4 - 204\al^3 - 48\al^2 + 128\al - 32)}{(3\al - 2)(\al + 2)(12\al^2 - \al - 2)(7\al^2 + \al - 2)}$ & \\
	\end{tabular}
	\caption{Idempotents of $6\A(\al, \frac{-\al^2}{4(2\al-1)})$}\label{tab:6Aidempotents}
\end{table}

\subsection{$6\J(2\bt, \bt)$}

For $6\J(2\bt, \bt)$, we require that $\bt  \neq \{0,1, \frac{1}{2}\}$ for $\{1,0,\al,\bt\}$ to be distinct.

By computation, we see that the algebra has an identity with $\1 = \tfrac{1}{7\bt+1} \left( \sum_{i=-2}^3 a_i +u+w \right)$ unless $\bt =-\tfrac{1}{7}$.  If $\bt = -\frac{1}{7}$, then $\la  \sum_{i=-2}^3 a_i +u+w \ra = \Ann A$.

\begin{lemma}
	For the axis $a_0 \in 6\J(2\bt,\bt)$, we have
	\begin{align*}
		A_1(a_0) &= \left\la a_0 \right\ra  \\
		A_0(a_0) &= \la 2\bt a_0 - a_3 - u, a_1 - 2a_3 +a_{-1} -2u +  w, 2 a_2 - a_3 + 2 a_{-2} - u \ra\\
		A_{2\bt}(a_0) &= \la a_1 + a_{-1} - w, a_3 - u\ra \\
		A_{\bt}(a_0) &= \la a_1 - a_{-1}, a_2-a_{-2} \ra 
	\end{align*}
\end{lemma}

\begin{lemma}
	There are no ideals of $A := 6\J(2\bt, \bt)$ which contain axes.
\end{lemma}
\begin{proof}
Since there are two orbits of axes in $A$ and $(a_0, a_1) = \bt \neq 0$, by Corollary \ref{axisideal}, $6\J(2\bt,\bt)$ has no proper ideals which contain axes.
\end{proof}

\begin{computation}
	The determinant of the Frobenius form is 
	\[
	-\tfrac{1}{16} \cdot (\bt - 2)^5 \cdot (2\bt - 1)^2 \cdot \left(7\bt + 1\right)
	\]
\end{computation}

So the algebra has a non-trivial radical when  $\bt = -\frac{1}{7}$ or $\bt = 2$.  In characteristic $3$, if $\bt = 2$, then $2\bt = 1$, a contradiction.  So $\bt = -\frac{1}{7}$ and $\bt = 2$ are simultaneously satisfied only when $\ch \FF = 5$.

\begin{proposition}
	The ideals and quotients of $6\J(2\bt, \bt)$ are given in Table $\ref{tab:6Jprops}$.
\end{proposition}

\begin{proof}
If $\ch \FF \neq 5$, then $\bt = 2$ and $\bt = -\frac{1}{7}$ are separate cases and this is proved in \cite[Lemma 5.2]{2gen2btbt}.\footnote{Note that these cases can happen at the same time, which is missed in this paper.}

Finally, if $\ch \FF = 5$, then $\bt = 2 = -\frac{1}{7}$ coincide.  In this case, the radical is $6$-dimensional and is the sum of the two ideals above and the quotient is $3\C(-1)^\times$.  There are no additional ideals.
\end{proof}

\begin{table}[h!tb]
	\renewcommand{\arraystretch}{1.5}
	\centering
	\footnotesize
	\begin{tabular}{c|c|c|c}
		\hline
		Condition & Ideals & Quotients & Dimension \\
		\hline
		$\bt = \frac{-1}{7}$ and  $\ch \FF \neq 3$   &  $R_1 :=\Ann A = \la \sum_{i=-2}^3 a_i +u+w \ra$  & $6\J\left(\tfrac{-2}{7}, \tfrac{-1}{7}\right)^\times$ & 7 \\
		$\bt = 2$ and $\ch \FF \neq 3$   & \begin{tabular}[t]{c} $R_2 :=\la a_0-a_2,a_{2}-a_{-2}, a_1-a_{-1},$ \\ $ a_1-a_3, u-\frac{1}{2} w \ra$ \end{tabular} & $3\C(4)$ & 3 \\
		 $\ch \FF =5 $ and $\bt = 2 = \tfrac{-1}{7}$   & $R_1 \oplus R_2$  & $3\C(-1)^\times$ & 2 \\
	\end{tabular}
	\caption{Ideals and quotients of $6\J(2\bt, \bt)$}\label{tab:6Jprops}
\end{table}

Turning to the idempotents, suppose that $7\bt+1$, $p := 384 \bt^5 - 960 \bt^4 + 373 \bt^3 - 21 \bt^2 - 9 \bt + 1$ and $36 \bt^4 + 140 \bt^3 - 43 \bt^2 - 2 \bt + 1$ are non-zero and $\FF$ contains roots $\rt_1, \rt_2$ and $\rt_3$, where
\begin{align*}
p\, \rt_1^2 &= 11\bt-1 \\
p\, \rt_2^2 &= (4\bt-1)(4\bt^2-7\bt+1) \\
\rt_3^2 &= \tfrac{1}{36 \bt^4 + 140 \bt^3 - 43 \bt^2 - 2 \bt + 1}
\end{align*}
Let
\begin{align*}
f &:= t^2 + 2(960\bt^5 - 1536\bt^4 + 167\bt^3 + 69\bt^2 - 3\bt - 1)t \\
&\phantom{{}={}} \quad + (608\bt^5 - 1152\bt^4 + 47\bt^3 + 77\bt^2 - 3\bt - 1)^2
\end{align*}
and suppose that $f$ has two distinct roots $\gm_0$ and $\gm_1$ and that $\FF$ contains the roots $\chi_0$ and $\chi_1$, where
\[
4(7\bt+1)^2 p \, \chi_i^2 := \gm_i
\]
and fix the signs so that $4(7\bt+1)^2 p\,  \chi_0 \chi_1 = 608\bt^5 - 1152\bt^4 + 47\bt^3 + 77\bt^2 - 3\bt - 1$.  Finally, let $q \in \Z(\bt)[t]$ be a certain cubic polynomial with
distinct roots $\rho_0, \rho_1, \rho_2$. We define $\lm_0, \lm_1, \lm_2 \in \Z(\bt)$ (see \cite[\texttt{6J idempotents.m}]{githubcode} for details) and set $\zeta \in \FF$ such that
\[
\zeta_i^2 := \sum_{k=1}^3 \lm_k \rho_i^k
\]
where $i = 0,1,2$.  We can now define the elements needed in the algebra to give the idempotents.
\begin{align*}
	w' &:= \tfrac{1}{1+\bt}(-\bt u +w) \\
	s &:= \tfrac{1}{7\bt +1}\left( \sum_{i=-2}^3 a_i -\tfrac{6\bt}{1+\bt}(u+w) \right)\\
	u_1 &:= \sum_{i=-2}^{3} \left(\tfrac{\bt(2\bt-1)}{1+7\bt}\rt_1+(-1)^i\tfrac{\rt_2}{2}\right)a_i + \tfrac{(46\bt^2-3\bt-1)}{2(1+7\bt)}\rt_1 u -\tfrac{(24\bt^2-\bt-1)}{2(1+7\bt)}\rt_1 w \\
	u_2 &:= \left(2\bt(5\bt-1)\rt_3-\tfrac{1}{2}\right)u+\tfrac{26\bt^2-9\bt+1}{2}\rt_3 w' + \tfrac{34\bt^2-3\bt-1}{2}\rt_3 s +(1-5\bt)\rt_3(a_0+a_3)\\
	u_3 &:= u_2 + u \\
	u_4 &:= -\tfrac{1}{2} a_0 + \tfrac{\rt_3}{2(7\bt+1)} \left( \bt (42 \bt^2 - 7 \bt - 1)a_0  - (34\bt^2 - 3\bt - 1)(a_1 + a_{-1} + w) \right. \\
	&\phantom{{}={}} \quad \left. + (22\bt^2-3\bt-1)(a_2+a_{-2}) + (36\bt^2-\bt-1)(a_3 + u)\right) \\
	u_5 &:= u_4 + a_0 \\
	u_6(\bar{\i}) &:= \tfrac{\chi_0+\chi_1}{4\bt(2\bt-1)}\big( (46\bt^2-3\bt-1)a_0 - (24\bt^2-\bt-1)(a_2+a_{-2}) \big)  \\
	&\phantom{{}={}} \quad + \chi_{\bar{\i}}(a_1+a_{-1} + a_3) + \chi_{\bar{\i}+\bar{1}}(u+w) \\
	u_7(\bar{\j}) &:= \zeta_{\bar{\j} + \bar{1}}( \rho_{\bar{\j} + \bar{1}}a_0 + a_2 + a_{-2})  + \zeta_{\bar{\j} + \bar{2}}(\rho_{\bar{\j} + \bar{2}}a_3 + a_1+a_{-1}) + \zeta_{\bar{\j}}(\rho_{\bar{\j}}u + w)
\end{align*}
where $\bar{\i} \in \Z/2\Z$ and $\bar{\j} \in \Z/3\Z$

\begin{computation}
Over the algebraic closure of $\FF(\bt)$, $6\J(2\bt,\bt)$ has $255$ non-trivial idempotents.  Under the action of the group $\la \tau_0, \half \ra \cong D_{12}$, there are $8$ orbits of length $1$, $4$ of length $2$, $8$ of length $3$, $28$ of length $6$ and $4$ of length $12$.  Some of these are given in Table $\ref{tab:6Jidempotents}$.
\end{computation}
\begin{proof}
This is analogous to the proof for Computation \ref{6Aidempotents}.  Again we can find all idempotents except those in the orbits of size $12$; see \cite[\texttt{6J find idempotents.m}, \texttt{6J idempotents.m}]{githubcode} for details.
\end{proof}

\begin{table}[H]
	\centering
	\renewcommand{\arraystretch}{1.5}
	\small
	\begin{tabular}{cccc}
		\hline
		Representative & orbit size & length & comment \\
		\hline
		$u$ & 1 & 1 & { \renewcommand{\arraystretch}{1}\begin{tabular}[t]{c} primitive $\cJ(2\bt)$\\ $\tau_u = \half$ \end{tabular} }\\
		$w'$ & 1 & $\frac{2-\bt}{1+\bt}$ & { \renewcommand{\arraystretch}{1}\begin{tabular}[t]{c} $5$ eigenvalues \\ $\tau_{w'} = \half$ \end{tabular} }\\
		$s$ & 1 & $\frac{6(1-2\bt)}{(1+\bt)(1+7\bt)}$ & $4$ eigenvalues\\
		$u+w'$ & 1 & $\frac{3}{1+\bt}$& \\
		$u+s$ & 1& $\frac{7\bt^2-4\bt +7}{7\bt^2+8\bt+1}$& \\
		$s+w'$ & 1 & $\frac{8-7\bt}{1+7\bt}$& $\cJ(1-2\bt)$, $\tau_{s+w'} = \half$ \\
		$\1$ & 1 & $\frac{9}{1+7\bt}$ & $\1 = u+w'+s$\\
		$\1_{\lla a_0, a_2 \rra}$ & 2 & $\frac{3}{1+\bt}$ & \\
		$\1-\1_{\lla a_0, a_2 \rra}$ & $2$ & $\frac{6(1-2\bt)}{(1+\bt)(1+7\bt)}$ & \\
		$\frac{1}{2}\1	\pm u_1$ & $2\times2$	& $\frac{9\pm (1-2\bt)(1-11\bt) \rt_1}{2(1+7\bt)}$ &  \\
		$\1_{\lla a_0, a_3 \rra}$ & $3$ & $\frac{3}{1+2\bt}$ & { \renewcommand{\arraystretch}{1}\begin{tabular}[t]{c} $5$ eigenvalues \\ $\tau_u = \tau_0$ \end{tabular} }\\
		$\1-\1_{\lla a_0, a_3 \rra}$ & $3$ & $\frac{3(2-\bt)}{(1+7\bt)(1+2\bt)}$ & \\
		$\1_{\lla a_0, a_3 \rra}-u$ & $3$ & $\frac{2(1-\bt)}{1+2\bt}$ & { \renewcommand{\arraystretch}{1}\begin{tabular}[t]{c} $7$ eigenvalues \\ $C_2 \times C_2$ graded \\ $\tau_0, \half$ \end{tabular} }\\
		$\1-(\1_{\lla a_0, a_3 \rra}-u)$ & $3$ & $\frac{14\bt^2+6\bt+7}{(1+7\bt)(1+2\bt)}$ & \\
		$\frac{1}{2}\1	\pm u_2$ & $2\times3$	&  $\frac{9 \mp(1+7\bt)\mp 3\bt(1-2\bt)(1-7\bt)\rt_3}{2(1+7\bt)}$ & \\
		$\frac{1}{2}\1	\pm u_3$ & $2\times3$       & $\frac{9 \pm(1+7\bt) \mp 3\bt(1-2\bt)(1-7\bt)\rt_3}{2(1+7\bt)}$ & \\
		$a_0$ & 6 & 1 & $\cM(2\bt,\bt)$ \\
		$\1 - a_0$ & 6 & $\frac{8-7\bt}{7\bt+1}$ & $\cM(1-2\bt,1-\bt)$ \\
		$\1_{\lla a_0, a_2 \rra} - a_0$ & 6 & $\frac{2-\bt}{\bt+1}$ &{ \renewcommand{\arraystretch}{1}\begin{tabular}[t]{c}  $7$-eigenvalues \\ $\tau_{\1_{\lla a_0, a_2 \rra} - a_0} = \tau_0$  \end{tabular} }\\
		$\1_{\lla a_0, a_3 \rra} - a_0$ & 6 & $\frac{2(1-\bt)}{2\bt+1}$ &{ \renewcommand{\arraystretch}{1}\begin{tabular}[t]{c}  $7$-eigenvalues \\ $\tau_{\1_{\lla a_0, a_2 \rra} - a_0} = \tau_0$  \end{tabular} } \\
		$\1 -(\1_{\lla a_0, a_2 \rra} - a_0)$ & 6 & $\frac{7\bt^2 -4\bt +7}{(\bt+1)(7\bt+1)}$ &  \\
		$\1 -(\1_{\lla a_0, a_3 \rra} - a_0)$ & 6 & $\frac{14\bt^2 + 6\bt + 7}{(2\bt+1)(7\bt+1)}$ &  \\
		$ \tfrac{1}{2}\1 \pm u_4$ & $ 2\times 6$ & $ \frac{9 \mp (1+7\bt) \pm 3\bt (1-2 \bt) (1-7 \bt)\rt_3}{2(1+7\bt)}$ & \\
		$ \tfrac{1}{2}\1 \pm u_5$ & $ 2\times 6$ & $ \frac{9 \pm (1+7\bt) \pm 3\bt (1-2 \bt) (1-7 \bt)\rt_3}{2(1+7\bt)}$ &  \\
		$\frac{1}{2}\1 \pm u_6(\bar{i})$ & $4 \times 6$ & $\frac{9}{2(7\bt+1)}\pm \frac{11\bt-1}{4\bt}(\chi_0+\chi_1)$ & \\
		$\frac{1}{2} \1 \pm u_7(\bar{\j})$ & $6 \times 6$ & $\frac{9}{2(7\bt+1)}\pm \frac{3(36\bt^3-6\bt^2+5\bt-1)}{64\bt^3-2\bt^2+5\bt-1}(\zeta_0+\zeta_1+\zeta_3)$ &
	\end{tabular}
	\caption{Idempotents of $6\J(2\bt, \bt)$}\label{tab:6Jidempotents}
\end{table}

\subsection{$6\Y(\frac{1}{2}, 2)$}
	
For $A:= 6\Y(\tfrac{1}{2}, 2)$, we require $\ch \FF  \neq 3$ for $\{1,0,\al,\bt\}$ to be distinct.  It is clear from the multiplication that $\la z \ra$ is the annihilator of the algebra. As $z \in \Ann(A)$, there is no $x\in A$ such that $xz=z$, hence $A$ does not have an identity.

\begin{lemma}
	For the axis $a_0 \in 6\Y(\tfrac{1}{2}, 2)$, we have
	\begin{align*}
		A_1(a_0) &= \left\la a_0 \right\ra  \\
		A_0(a_0) &= \la z, 2 a_0 - a_2 - a_4 \ra \\
		A_{\frac{1}{2}}(a_0) &= \la d+2z\ra \\
		A_{2}(a_0) &= \la a_2 - a_{4} \ra 
	\end{align*}
\end{lemma}

\begin{lemma}
No proper ideals of $6\Y(\frac{1}{2}, 2)$ contain axes.
\end{lemma}
\begin{proof}
There are two orbits of axes and $(a_0, a_1) = 1$, so the result follows.
\end{proof}

\begin{lemma}\label{6Y_radical}
The radical of $A$ is $R:= \la a_0-a_2, a_2-a_4,d,z \ra$. \qed
\end{lemma}

\begin{proposition}
	The ideals and quotients of $A = 6\Y(\tfrac{1}{2}, 2)$ are given in Table $\ref{tab:6Yprops}$.
\end{proposition}

\begin{proof}
From Lemma \ref{6Y_radical}, the radical $R = \la a_0-a_2, a_2-a_4, d, z \ra$.  Since $\ch \FF \neq 2,3$, we see that $R$ decomposes as a $\Miy A \cong S_3$ module as a sum of a $2$-dimensional irreducible module $\la a_0-a_2, a_2- a_4 \ra$ and the sum of two trivial modules $\la d \ra \oplus \la z \ra$.
	
	We start by identifying all $1$-dimensional sub-ideals; let $\la \lm d + \mu z\ra$ span such an ideal for $\lm, \mu \in \FF$. Since $d(\lm d + \mu z) = -2\lm z$, we see that we require $\lm = 0$.  It is easy to see that $\la z \ra = \Ann A$ is indeed an ideal.
	
	It is clear from the multiplication that $\la d, z\ra$ is 2-dimensional subalgebra and since $a_i d = \tfrac{1}{2} d +z $ and $d^2 = -2z$, $I_2 := \la d, z\ra$ is a $2$-dimensional sub-ideal. From the $\Miy(A)$-module structure, the only other possible $2$-dimensional ideal is $I_3:=\la a_0-a_2, a_2-a_4 \ra$, which it is easy to check is indeed an ideal.

The only possible $3$-dimensional sub-ideal is $I_4:= \la a_0-a_2,a_2-a_4, z \ra = \la a_0 - a_2, a_2 - a_4 \ra + \Ann A$ and so it is indeed an ideal. 

Now, $Q_1 := A/ \Ann A$ is $4$-dimensional with Monster type $(\tfrac{1}{2},2)$ axes.  So it is not isomorphic to any previous algebra; we call this $6\Y(\frac{1}{2}, 2)^\times$.  For $Q_2:= A/I_2$, we see that it is $3$-dimensional with Jordan type $2$ axes, hence it is isomorphic to $3\C(2)$.

The quotient $Q_3:= A/I_3$ is $3$-dimensional and $\bar{a}_0 = \bar{a_2} = \bar{a}_4$ and $\bar{a}_1 = \bar{a_3} = \bar{a}_5$ are two distinct Jordan type $\frac{1}{2}$ axes.  By computation, $\bar{a}_0 \bar{a}_1 = \bar{a}_0 + \frac{1}{2} \bar{d} + \bar{z} = \frac{1}{2}(\bar{a}_0 + \bar{a}_1) + \bar{z}$ and hence $Q_3 \cong \widehat{S}(2)^\circ$.

We know that $\widehat{S}(2)^\circ$ has a $1$-dimensional annihilator which gives the $2$-dimensional quotient $S(2)^\circ$. From the correspondence theorem, we can see that $Q_4 := A/I_4 \cong  Q_3/\la \bar{z} \ra \cong S(2)^\circ$. 	Finally, $Q_5 := A/I_5$ is just $1\A$.
\end{proof}

\begin{table}[h!tb]
	\renewcommand{\arraystretch}{1.5}
	\centering
	\footnotesize
	\begin{tabular}{c|c|c}
		\hline
		 Ideals & Quotients & Dimension \\
		\hline
		    $\la z \ra$ &  $6\Y(\frac{1}{2}, 2)^\times$ & 4 \\
		    $\la d,z \ra$ &  $3\C(2)$ & 3 \\
		    $\la a_0-a_2,a_2-a_4 \ra$ &  $\widehat{S}(2)^\circ$ & 3 \\
		    $\la a_0-a_2,a_2-a_4, z \ra$ &  $S(2)^\circ$ & 2 \\
		    $\la a_0-a_2,a_2-a_4, d,z \ra$ &  $1\A$ & 1
	\end{tabular}
	\caption{Ideals and quotients of $6\Y(\frac{1}{2}, 2)$}\label{tab:6Yprops}
\end{table}

	We now examine the subalgebra structure of $6\Y(\frac{1}{2}, 2)$.  It is clear from the multiplication table that $ B_0 := \lla a_0, a_{2} \rra \cong \la a_0,a_{2}, a_4 \ra \cong 3\C(2)$ and $ B_1 := \lla a_1, a_{3} \rra \cong \la a_1,a_{3}, a_5 \ra \cong 3\C(2)$.  Note that $3\C(2)$ has an identity and so we have $\1_{B_0} := \tfrac{1}{3}(a_0+a_2+a_4)$ and $\1_{B_1} := \tfrac{1}{3}(a_1+a_3+a_5)$.

\begin{proposition}\label{6Ysub2}
	The subalgebras $ \lla a_0, a_{3} \rra \cong \la a_0,a_3, z \ra$, $ \lla a_1, a_{4} \rra \cong \la a_1,a_4,z \ra$, $ \lla a_2, a_{5} \rra \cong \la a_2,a_5,z \ra$, and  $ \lla \1_{B_0}, \1_{B_1} \rra \cong \la \1_{B_0},\1_{B_1},z \ra$ are axial algebras of Jordan type $\cJ(\tfrac{1}{2})$ and
	\[
	\lla a_i, a_{i+3}\rra \cong \lla \1_{B_0}, \1_{B_1}\rra \cong \widehat{S}(2)^\circ
	\]
\end{proposition}
\begin{proof}
Clearly $a_i^2=a_i$ and $\1_{B_{i}}^2=\1_{B_{i}}$. We also know that $z$ is in the annihilator. Now, we calculate the product $a_i\cdot a_{i+3}$.  Observe that, by the multiplication table 
\[
a_i\cdot a_{i+3}=  a_i(a_i+d) = a_i  + \tfrac{1}{2}d +z = \tfrac{1}{2}(a_i+a_{i+3})+z.
\]
Hence $\la a_i,a_{i+3},z \ra \cong \lla a_i, a_{i+3} \rra \cong \widehat{S}(2)^\circ $. Similarly,
\[
\1_{B_0}\cdot\1_{B_1} = \1_{B_0}\cdot(\1_{B_0}+d) = \1_{B_0} + \tfrac{1}{2}d +z = \tfrac{1}{2}(\1_{B_0}+\1_{B_1})+z,
\]
and so $\lla \1_{B_0}, \1_{B_1} \rra \cong \la \1_{B_0},\1_{B_1},z \ra \cong \widehat{S}(2)^\circ$.
\end{proof}
	
We now turn to examining the idempotents of $6\Y(\frac{1}{2}, 2)$. 

\begin{computation}
The idempotent ideal for algebra $6\Y(\frac{1}{2},2)$ decomposes as the sum of a $0$- and $1$-dimensional ideal.  The $0$-dimensional ideal has three non-trivial idempotents given in the first row of Table $\ref{tab:6Yidempotents}$.  The $1$-dimensional ideal comprises four infinite families given by
\begin{align*}
x_i(\lm) &:= \lm a_i + (1-\lm) a_{i+3} + 2\lm(1-\lm)z \\
y(\lm) &:= \lm \1_{B_0} + (1-\lm) \1_{B_1} + 2\lm(1-\lm)z
\end{align*}
for $i = 0,1,2$, where $\lm \in \FF$.
\end{computation}

\begin{table}[h!tb]
\centering
	\begin{tabular}{cccc}
		\hline
		Representative & orbit size & length & comment \\
		\hline
		$\1_{B_0}-a_0$ & $3$ &  $0$ & $\cJ(-1)$-axis \\
		$x_0(\lm, \mu)$ & - & $1$  &   \\ 
		$x_1(\lm, \mu)$ & - & $1$  &   \\ 
		$x_2(\lm, \mu)$ & - &  $1$ &   \\ 
		$y(\lm, \mu)$ & - & $1$  &  \\
	\end{tabular}
	\caption{Idempotents of $6\Y(\frac{1}{2}, 2)$}\label{tab:6Yidempotents}
\end{table}

Note that the four infinite families of idempotents in $6\Y(\frac{1}{2}, 2)$ arise from the four $\widehat{S}(2)^\circ$ subalgebras identified in Proposition~\ref{6Ysub2}. We can identify these idempotents using a similar argument to Proposition \ref{4Aidems}. Note also that the axes themselves occur as members in the families $x_i(\lm)$.

\section{Algebras with axet $X(\infty)$}

There are two families of algebra $\IY_3(\al, \frac{1}{2}, \mu)$ and $\IY_5(\al, \frac{1}{2})$, which generically have infinitely many axes, but can have finitely many for certain values of the parameters, or characteristic of the field.  The structure constants are given in Table \ref{tab:2genMonsterXinf}.

\begin{table}[h!tb]
\setlength{\tabcolsep}{4pt}
\renewcommand{\arraystretch}{1.5}
\centering
\footnotesize
\begin{tabular}{c|c|c}
Type & Basis & Products \& form \\ \hline
\begin{tabular}[t]{c} $\IY_3(\al, \frac{1}{2}, \mu)$ \\ $\al \neq -1$, $\mu \neq 1$ \end{tabular}
 & \begin{tabular}[t]{c} $e$, $f$, \\ $z_1$, $z_2$ \end{tabular} &
\begin{tabular}[t]{c}
$z_i^2 = z_i$, $z_1z_2 = 0$ \\
$ez_1 = \al e$, $ez_2 = (1-\al) e$ \\
$e^2 = -z = f^2$, $ef = -\mu z$ \\
where $z:= \al(\al-2)z_1 + (\al-1)(\al+1)z_2$ \\
$(e,e) = (\al+1)(2-\al) = (f,f)$ \\
$(e,f) = \mu(\al+1)(2-\al)$, $(e,z_i) = 0 = (f,z_i)$ \\
 $(z_1,z_1) = \al+1$, $(z_2,z_2) = 2- \al$, $(z_1,z_2) = 0$
\vspace{4pt}
\end{tabular}\\
\hline
\begin{tabular}[t]{c} $\IY_3(-1, \frac{1}{2}, \mu)$ \\ $\mu \neq 1$ \end{tabular}
 & \begin{tabular}[t]{c} $e$, $f$, \\ $z_1$, $n$ \end{tabular} &
\begin{tabular}[t]{c}
$z_1^2 = z_1$, $n^2 = 0$, $z_1n = 0$ \\
$ez_1 = -e$, $en = 0$ \\
$e^2 = -z = f^2$, $ef = -\mu z$ 
where $z:= 3z_1 -2n$ \\
$(e,e) = 3 = (f,f)$ \\
$(e,f) = 3\mu$, $(e,z_1) = 0 = (f,z_1)$ \\
 $(z_1,z_1) = 1$, $(n,x) = 0$ for all $x \in A$
\vspace{4pt}
\end{tabular}\\
\hline
$\IY_3(\al, \frac{1}{2}, 1)$ & \begin{tabular}[t]{c} $a_0$, $a_1$, \\ $z$, $n$ \end{tabular} &
\begin{tabular}[t]{c}
$a_0 \cdot a_1 = \frac{1}{2}(a_0+a_1) + (\al-\frac{1}{2})z + n$ \\
$z^2 = 0 = n^2$, $a_iz = \al z$, $a_in = 0 = zn$ \\
$(a_i,a_j) = 1$, $(x,z) = 0 = (x,n)$ for all $x \in A$ 
\vspace{4pt}
\end{tabular}\\
\hline
$\IY_5(\al, \frac{1}{2})$ & \begin{tabular}[t]{c} $a_{-2}$, $a_{-1}$, $a_0$\\ $a_1$, $a_2$, $z$ \end{tabular} &
\begin{tabular}[t]{c}
$a_i \cdot a_{i+1} = \frac{1}{2}(a_i+a_{i+1}) + z$ \\
$a_i \cdot a_{i+2} = \frac{1}{2}(a_i+a_{i+2}) + \frac{1}{2}u_4 + 4z$ \\
$a_i \cdot z = \frac{2\al-1}{8}(-2a_i + (a_{i-1}+a_{i+1}) + 4z)$ \\
$z \cdot z =  \frac{(2\al-1)(2\al-3)}{2^5}u_4$\\
where $u_4 := \sum_{i=0}^4 (-1)^i \binom{4}{i}a_i$\\
and, for all $j \in \Z$, $0 = \sum_{i=0}^5 (-1)^i\binom{5}{i} a_{i+j}$ \\
$(a_i, a_j) = 1$, $(a_i, z) = 0= (z, z)$
\vspace{4pt}
\end{tabular}\\
\end{tabular}\\
\caption{$2$-generated $\cM(\alpha, \beta)$-axial algebras on $X(\infty)$}\label{tab:2genMonsterXinf}
\end{table}

\subsection{$\IY_3(\al, \frac{1}{2}, \mu)$}

This algebra has Monster type $(\al, \frac{1}{2})$ and so we require $\al \neq 1,0, \frac{1}{2}$; we also have an additional parameter $\mu \in \FF$.  If $\mu \neq 1$, then they are examples of split spin factor algebras, which were introduced by M\textsuperscript{c}Inroy and Shpectorov in \cite{splitspin}.  In that paper, conditions for simplicity and the axet are given and idempotents are classified.  Here we detail the exact ideals and quotients and quote the other results.  When $\mu=1$, we use the exceptional algebra given in \cite{forbidden}.

The algebra has an identity $\1 = z_1 + z_2$ if and only if $\al \neq -1$ and $\mu \neq 1$.  Otherwise, it has a $1$-dimensional annihilator $\Ann(A) = \la n \ra$.

\begin{proposition}
The ideals and quotients of $\IY_3(\al, \frac{1}{2}, \mu)$ are given in Table $\ref{tab:IY3props}$.
\end{proposition}
\begin{proof}
First consider the case when $\mu = 1$.  From the Frobenius form, we see that no axes lie in a proper ideal and the radical is $3$-dimensional and spanned by $a_0-a_1, z, n$ (in fact the algebra is baric).  The axes have common $0$- and $\al$-eigenspaces spanned by $n$ and $z$, respectively.  It is easy to see that these span ideals $\la n \ra$ and $\la z \ra$ and so their sum is also an ideal.  Since ideals must decompose into sums of eigenspaces and $u_\bt := a_0-a_1 + z -2n \in A_\bt(a_0)$, this is the only other candidate for a $1$-dimensional ideal and all other ideals must be sums of these.  However, 
\[
a_1(a_0-a_1 + z -2n) = \tfrac{1}{2}(a_0-a_1 + z -2n) + (2\al-1)z + 2n
\]
and so both $z, n \in (u_\bt)$.  Hence there are no other proper ideals of the radical.

Since $n$ is a common $0$--eigenspace, the quotient $\IY_3(\al, \frac{1}{2}, 1)/(n)$ is another Monster type algebra with (generically) infinitely many axes and an empty $0$-eigenspace.  We note that this algebra fits into a family found by Whybrow, namely fusion law (a) with $\bt=\frac{1}{2}$ and $x=1$ in \cite[Theorem 5.6]{whybrowgraded}\footnote{Since the author is no longer in academia, this paper is now unfortunately unlikely to be published.  However the results are correct, except for a few minor typos in some formulae and tables.}, and also subsequently by Afanasev in \cite{axial3evals}.  We denote this algebra by $\mathrm{W}_a(\al, \frac{1}{2}, 1)$.  The quotient $\IY_3(\al, \frac{1}{2}, 1)/(z)$ is of Jordan type $\frac{1}{2}$ with a non-trivial annihilator and so it is isomorphic to $\Cl$.  We see then that $\IY_3(\al, \frac{1}{2}, 1)/(z,n) \cong S(2)^\circ$.

So we may now suppose that $\mu \neq 1$.  First suppose that $\al = -1$ and so $\ch \FF \neq 3$ as otherwise $\al = -1 = \frac{1}{2} = \bt$.  By \cite[Theorem 6.9]{splitspin} the algebra is isomorphic to the cover $\widehat{S}(-1,b)^\circ$ of a split spin factor, where $b$ is the bilinear form with Gram matrix $\begin{psmallmatrix} 1 & \mu \\ \mu & 1 \end{psmallmatrix}$.  By \cite[Theorem 6.8]{splitspin} every proper ideal is contained in the radical $b^\perp \oplus \la n \ra$.  It is easy to see that $b^\perp$ is non-zero if and only if $\mu = \pm 1$.  First suppose that $\mu \neq \pm 1$.  Then, $\Ann = \la n \ra$ and $A/\la n \ra$ is an algebra of Monster type $(2, \frac{1}{2})$ with generically infinitely many axes; it is the algebra $\IY_3(-1, \frac{1}{2}, \mu)^\times$.   The case of $\mu = 1$ is dealt with above; if $\mu=-1$, then $b^\perp$ is spanned by $e+f$. Then, it is easy to see that $\la e+f \ra$ is the common $\frac{1}{2}$-eigenspace and in fact is an ideal $I_\bt$.  The quotient $A/I_\bt$ is a $3$-dimensional $2$-generated axial algebra of Jordan type $\al$ and so $A/I_\bt \cong 3\C(-1)$.  Then, the quotient $A/\la n, e+f \ra \cong 3\C(-1)^\times$.

Now suppose that $\al \neq -1$.  By \cite[Theorem 5.1]{splitspin} the algebra is isomorphic to a split spin factor algebra $S(\al,b)$, where $b$ is the bilinear form with Gram matrix $\begin{psmallmatrix} 1 & \mu \\ \mu & 1 \end{psmallmatrix}$.  If $\al \neq 2$, by \cite[Proposition 4.3]{splitspin}, there are no proper ideals containing axes.  Suppose that $\al = 2$; so $\ch \FF \neq 3$ and otherwise $2 = \frac{1}{2}$.  Let $x = \frac{1}{2}(u + 2 z_1 + 3 z_2)$ be an axis contained in an ideal $I$, where $u \in \la e,f \ra$.  By computation, $2xz_1 = 2u + 2z_1$ and $(2xz_1)z_1 = 4u + 2z_1$.  So $u, z_1 \in I$ and hence also $z_2 \in I$.  Now pick $v \in \la e,f \ra - \la u \ra$.  Then $vz_1 = 2v \in I$ and hence $I = A$.  So there are no proper ideal which contain axes.  We now turn to the radical.  If $\al \neq 2$, then by \cite[Proposition 4.3]{splitspin}, the radical of $A$ equals $b^\perp$ which is $e+f$.  Similarly to above this is the common $\frac{1}{2}$-eigenspace and so $A/I_\bt \cong 3\C(\al)$.  If $\al=2$, then by \cite[Proposition 4.1]{splitspin}, the radical equals $R  := \la e,f,z_2 \ra$.  This decomposes as a $\Miy(A)$-module as the sum of an indecomposable module $\la e,f \ra$ and a trivial module $\la z_2\ra$.  Note that $ez_2 = -e$ and $fz_2 = -f$ and so $(z_2) = R$.  It remains to consider subideals of $\la e,f \ra$.  Suppose that $re + sf$ generate a proper ideal of $R$.  We have
\[
e^2 = -3z_1 = f^2, \quad ef = -3\mu z_2
\]
and so, we see that if $I$ is proper, then $r+s\mu = 0 = r\mu + s$.  Whence, $\mu^2=1$ and so $\mu = -1$ and we can take $r=s=1$.  As before when $\mu = -1$, $I_\bt := \la e+f \ra$ is an ideal which is the common $\bt$-eigenspace and $A/I_\bt :=3\C(2)$.  Since $3\C(2)$ has only a $2$-dimensional ideal, by the correspondence theorem there are no further ideals.
\end{proof}

\begin{table}[h!tb]
\renewcommand{\arraystretch}{1.5}
\centering
\footnotesize
\begin{tabular}{c|c|c|c}
\hline
Condition  & Ideals & Quotients & Dimension\\
\hline
$\mu=1$ & $\Ann(A) = \la n \ra$ & $\mathrm{W}_a(\al, \frac{1}{2}, 1)$ & 3 \\
$\mu=1$ & $\la z \ra$ & $\Cl$ & 3\\
$\mu=1$ & $\la z, n \ra$ & $S(2)^\circ$ & 2\\
$\mu=1$ & $\la z, n, a_0-a_1 \ra$ & $1\A$ & 1\\
\hline
$\mu = -1$ & $\la e+f \ra$ & $3\C(\al)$ & 3 \\
\hline
$\mu \neq 1$, $\al = -1$ & $\Ann(A) = \la n \ra$ & $\IY_3(-1, \frac{1}{2}, \mu)^\times$ & 3 \\
$\mu =-1$, $\al = -1$ & $\la e+f, n \ra$ & $3\C(-1)^\times$ & 2 \\
\hline
$\al =2$ & $\la e,f,z_2\ra$ & $1\A$ & 1
\end{tabular}
\caption{Ideals and quotients of $\IY_3(\al, \frac{1}{2}, \mu)$}\label{tab:IY3props}
\end{table}

We now turn to the idempotents which have been completed described.  \footnote{Note that there is a typo in \cite[proof of Lemma 7.9]{forbidden} where there is an erroneous factor of $2\al-1$ in $x_1(\lm)$.}  We define
\begin{align*}
x_a(u) &:= \tfrac{1}{2}(u+ \al z_1 + (\al+1)z_2) \\
x_b(u) &:= \tfrac{1}{2}(u+ (2-\al) z_1 + (1-\al)z_2) \\
x_{-1}(u) &:= \tfrac{1}{2}(u - z_1 + n) \\
x_1(\lm) &:= \lm a_0+(1-\lm)a_1+\lm(\lm -1)z+2\lm(1 -\lm)n
\end{align*}
where $u \in \la e,f \ra$ with $b(u,u) = 1$, when $\mu \neq 1$; and $\lm \in \FF$ for $\mu = 1$.

\begin{proposition}\textup{\cite[Theorem 2, Proposition 6.2]{splitspin}\cite[Lemma 7.9]{forbidden}}
The idempotents of $\IY_3(\al, \frac{1}{2})$ are given in Table $\ref{tab:IY3idempotents}$.
\end{proposition}

\begin{table}[h!tb]
	\renewcommand{\arraystretch}{1.3}
	\begin{tabular}{ccccc}
		\hline
		Condition & Representative & orbit size & length & comment \\
		\hline
$\al \neq -1$, $\mu \neq 1$ &	$z_1$ & 1 & $\al+1$ & $\cJ(\al)$-type \\
&		$z_2$ & 1 & $2-\al$ & $\cJ(1-\al)$-type \\
&		$\1$ & 1 & $1$ &  \\
&		$x_a(u)$ & - & $\al+1$ & $\cM(\al, \frac{1}{2})$ \\
&		$x_b(u)$ & - & $2-\al$ & $\cM(1-\al, \frac{1}{2})$ \\
		\hline
$\al = -1$, $\mu \neq 1$ & $z_1$ & 1 & $0$ & $\cJ(-1)$-type \\
&		$x_{-1}(u)$ & - & $0$ & $\cM(-1, \frac{1}{2})$ \\
		\hline
$\mu = 1$ & $x_1(\lm)$ & - & $1$ & $\cM(\al,\frac{1}{2})$
	\end{tabular}
	\caption{Idempotents of $\IY_3(\al, \frac{1}{2}, \mu)$}\label{tab:IY3idempotents}
\end{table}


\subsection{$\IY_5(\al, \frac{1}{2})$}\label{sec:IY5}

For $A:= \IY_5(\al, \frac{1}{2})$, we require $\al \neq 1,0,\frac{1}{2}$ for $\{ 1,0,\al,\bt\}$ to be distinct.  I straightforward calculation shows that $\Ann A$ is $1$-dimensional and spanned by $u_4 := \sum_{i=0}^4 (-1)^i \binom{4}{i} a_i$.  As above, as $\Ann A \neq 0$, the algebra has no identity.

\begin{lemma}
	For the axis $a_0 \in \IY(\al,\frac{1}{2})$, we have
	\begin{align*}
		A_1(a_0) &= \left\la a_0 \right\ra  \\
		A_0(a_0) &= \la u_4, 2a_0 - (a_1+a_{-1}) +\tfrac{2}{2\al-1}z \ra \\
		A_{\frac{1}{2}}(a_0) &= \la 2a_0 - (a_1+a_{-1}) - 4z \ra \\
		A_{2}(a_0) &= \la a_1-a_{-1}, a_2 - a_{-2} \ra 
	\end{align*}
\end{lemma}

This algebra has infinitely many axes in characteristic $0$, but finitely many in positive characteristic.

\begin{proposition}\textup{\cite[Corollary 7.12]{forbidden}}\label{IY5axet}
The algebra $\IY_5(\al, \frac{1}{2})$ has axet $X(p)$ over a field of characteristic $p \geq 5$, $X(9)$ over a field of characteristic $3$ and axet $X(\infty)$ otherwise.
\end{proposition}

Since the Frobenius form has $(a_i,a_j) = \delta_{ij}$ and $(z,x) = 0$, for all $x \in A$, the algebra is baric and the radical of the Frobenius form has codimension $1$.  So no proper ideals of $A$ contain axes.

\begin{proposition}
The ideals and quotients of $\IY_5(\al,\frac{1}{2})$ are given in Table $\ref{tab:IY5props}$.
\end{proposition}
\begin{proof}
First we examine the $\Miy(A)$-module structure of $A$.  Clearly, $z$ is fixed by $\Miy(A)$ and $M_0 := \la a_0, \dots, a_4 \ra$ spans a $5$-dimensional modules of which $M_1:= \la a_0-a_1, \dots, a_3-a_4 \ra \subset R$ is a $4$-dimensional submodule.  So we see that $R = M_1 \oplus \la z \ra$.

We claim that $M_0$ is indecomposable and uniserial.  First, assume that $\ch \FF = 0$ and so $\Miy(A) \cong D_{\infty}$ by Proposition \ref{IY5axet}.  The element $\rho := \tau_0\tau_1$ acts on $M_0$ by translation $a_i \mapsto a_{i+2}$.  So take $a_{-2} = 15a_0-40a_1+45a_2-24a_3+5a_4, a_0, a_2, a_4, a_6 = 5a_0-24a_1+45a_2-40a_3+15a_4$ as a basis for $M_0$.  Then, we see that $M_0 \cong \FF[t, t^{-1}]/\la (t-1)^5 \ra$, where $\rho$ acts as $t^i \mapsto t^{i+1}$ and $\tau_0$ acts by $t \mapsto t^{-1}$.  By \cite{infdihedral} (or see the English survey \cite[Theorem 5]{infdihsurvey}), $M_0$ is indecomposable and it is well-known that this is uniserial (see for example, \cite{infdihuniserial}).  If $\ch \FF = p >0$, then by Proposition \ref{IY5axet}, $IY_5(\al, \frac{1}{2})$ has axet $X(p)$ (or $X(9)$ if $p=3$) and so the Miyamoto group is $D_{2p}$ (or $D_{18}$ is $p=3$).  So the Miyamoto group has a cyclic normal Sylow $p$-subgroup and hence every indecomposable module is uniserial (see, for example, \cite[pages 35-37]{alperin}).  Looking at the action of $\rho$ again, it is clear that $M_0$ is indecomposable and hence uniserial.

For $k = 1, \dots 4$, set $u_k := \sum_{i=0}^k (-1)^i \binom{k}{i}a_i$ and $M_k = \la u_k \ra$.  Then we have $M_4 \subsetneqq M_3 \subsetneqq M_2 \subsetneqq M_1$ as $\Miy(A)$-modules.  Note that, if $u,v \in M_j - M_{j+1}$, for some $j = 1,2,3$, then $u-v$ is in the codimension $1$ subspace $M_{j+1}$ of $M_j$.  Moreover, $M_{j+1}$ is spanned by such differences.  So, for each $\lm \in \FF$, $\la u_j + \lm z \ra$ generates a distinct $j$-dimensional module which contains $M_{j+1}$.  Hence the non-trivial submodules of $R$ are $M_i$, $M_i \oplus \la z \ra$, or $\la u_i + \lm z \ra$, for $\lm \in \FF^\times$, $i = 1,2,3,4$.

We can now check which submodules are in fact ideals.  First, one can check that $\Ann A = \la u_4 \ra$.  By computation, one can check that
\begin{gather*}
a_1 z = \tfrac{2\al-1}{8}(u_2+4z), \quad 
z^2 = \tfrac{(2\al-3)(2\al-1)}{32}u_4 \\
a_1u_3 = \tfrac{1}{4}(2u_3 + u_4), \quad u_1 u_3 = \tfrac{1}{2}u_4, \quad u_2 u_3 = u_3 u_3 = z u_3 = 0 \\
a_1 u_2 = \tfrac{1}{2}( u_2 + 4z), \quad 
u_1 u_2 = -\tfrac{1}{4} u_4, \quad
u_2 u_2 =  -\tfrac{1}{2} u_4, \quad
z u_2 = \tfrac{2\al-1}{8}u_4 \\
a_1 u_1 = \tfrac{1}{2}(u_1 + 2z), \quad
u_1 u_1 = -2 z, \quad
z u_1 = \tfrac{2\al-1}{8}(u_3+u_4)
\end{gather*}
Now it is easy to see that $( u_3 ) = \la u_3, u_4 \ra$, $(u_2) = \la u_2, u_3, u_4, z \ra$ and $(u_1) = R$.  Since $u_2 \in (z)$, we have $(z) = (u_2) = (u_2,z)$ and so $(u_3,z) = (u_2)$.  It remains to check whether $u_i + \lm z$ can generate a new ideal for some $i = 1, \dots, 4$ and $\lm \in \FF^\times$.

Let $x = u_1 + \lm z$ and $I = (x)$.  By above, $M_2 = \la u_2, u_3, u_4 \ra \subset I$.  Multiplying $x$ by $u_1$, we see that $u_1 \in I$ and hence $I = R$.  Now let $x = u_2 + \lm z$ and $I = (x)$; we have $u_3,u_4 \in I$.  By above, multiplying $x$ with either $u_3$ or $u_4$ gives $0$, and multiplying with $u_1, u_2$, or $z$ gives a linear combination of $u_3$ and $u_4$.  For $a_1$, we have
\[
a_1(u_2 + \lm z) = (\tfrac{1}{2} + \lm \tfrac{2\al-1}{8})( u_2 + 4z)
\]
If $\lm = -\tfrac{4}{2\al-1}$, then $a_1(u_2 + \lm z) = 0$ and so $(x) = \la u_2 -\tfrac{4}{2\al-1}z, u_3, u_4 \ra$ is $3$-dimensional.  If $\lm = 4$, then $(x) = \la u_2+4z, u_3, u_4 \ra$ is also $3$-dimensional; otherwise $(x) = (u_2) = (z)$ is $4$-dimensional.  The cases for $u_3$ and $u_4$ are similar to the first and neither result in new ideals, with $(u_3 + \lm z) = (u_2) = (u_4 + \lm z)$.

We now turn to the quotients.  The quotient $A/\Ann A$ is a $2$-generated Monster type axial algebra with (generically) infinitely may axes and so not isomorphic to any of the above; it is $\IY_5(\al, \frac{1}{2})^\times$.  In $\IY_5(\al, \frac{1}{2})^\times$, the axes still have a $2$-dimensional $\frac{1}{2}$-eigenspace, but $u_3$ is a common $\frac{1}{2}$-eigenvector for all axes.  So $\IY_5(\al,\frac{1}{2})/(u_3)$ is a $4$-dimensional axial algebra of Monster type.  This is previously unknown and we call this $\IY_5(\al,\frac{1}{2})^{\times\times}$.  In its fusion law, $0 \star 0 = 0 \star \al = \al \star \al = \emptyset$ and it is easy to see that (the images of) $u_2 - 4/(2\al-1)z$ and $u_2 + 4z$ are common $\al$- and $0$-eigenvectors, respectively, for all axes.  Hence the quotient $\IY_5(\al,\frac{1}{2})/(u_2-\frac{4}{2\al-1}z)$ is also another new Monster type axial algebra with (generically) infinitely many axes, where axes have an empty $0$-eigenspace.  This is again a member of Whybrow's fusion law (a) family with $\bt=\frac{1}{2}$ and $x=1$.  The axes in  $\IY_5(\al,\frac{1}{2})/(u_2 +4z)$ have Jordan type $\frac{1}{2}$ and the algebra is isomorphic to $\Cl$.  Finally, $\IY_5(\al,\frac{1}{2})/(u_3) \cong S(2)^\circ$.
\end{proof}

\begin{table}[h!tb]
\renewcommand{\arraystretch}{1.5}
\centering
\footnotesize
\begin{tabular}{c|c|c|c}
\hline
Condition & Ideal & Quotient & Dimension\\
\hline
& $R = \la u_1, u_2, u_3, u_4, z\ra$ & $1\A$ & 1 \\
& $\la u_2, u_3, u_4, z\ra$ & $S(2)^\circ$ & 2 \\
& $\la u_2 + 4z, u_3, u_4\ra$ & $\Cl$ & 3 \\
& $\la u_2 -\frac{4}{2\al-1}z, u_3, u_4\ra$ & $\mathrm{W}_a(\al, \frac{1}{2}, 1)$  & 3 \\
& $\la u_3, u_4\ra$ & $\IY_5(\al,\frac{1}{2})^{\times\times}$ & 4 \\
& $\Ann A = \la u_4 \ra$ & $\IY_5(\al,\frac{1}{2})^\times$ & 5
\end{tabular}\\
	\vspace{5pt}
	where $u_k := \sum_{i=0}^k (-1)^i \binom{k}{i}a_i$.
\caption{Ideals and quotients of $\IY_5(\al, \frac{1}{2})$}\label{tab:IY5props}
\end{table}

\begin{lemma}
The idempotent ideal for $\IY_5(\al, \frac{1}{2})$ decomposes as the sum of a $0$- and a $2$-dimensional ideal.  The $0$-dimensional ideal yields only the zero vector.
\end{lemma}
\begin{proof}
This is by computation \cite[\texttt{Xinf.m}]{githubcode}.
\end{proof}


\begin{thebibliography}{99}

\bibitem{axial3evals} V.A. Afanasev, On axial algebras with $3$ eigenvalues, \textit{arXiv}:2410.02034, 20 pages.

\bibitem{alperin} J.L. Alperin, \textit{Local Representation Theory: Modular Representations as an Introduction to the Local Representation Theory of Finite Groups}, Cambridge University Press, 1986.

\bibitem{infdihedral} S.D. Berman and K. Buz\'asi, Representations of the infinite dihedral group (Russian), \textit{Publ. Math. Debrecen} {\bf 28} (1981), no. 1-2, 173--187.

\bibitem{magma} W. Bosma, J. Cannon and C. Playoust, The Magma algebra system. I. The user language, \textit{J. Symbolic Comput.} {\bf 24} (1997), 235--265.

\bibitem{infdihsurvey} K. Buz\'asi, On representations of infinite groups, \textit{Algebra -- Some Current Trends}, Lecture Notes in Mathematics, vol. 1352, Springer, Berlin, 1988, 44--59.

\bibitem{Alonsoidempotents} A. Castillo-Ramirez, Idempotents of the Norton–Sakuma algebras, \textit{J. Group Theory} {\bf 16} (2013), 419--444.

\bibitem{DPSV} T. De Medts, S. F. Peacock, S. Shpectorov and M. Van Couwenberghe, Decomposition algebras and axial algebras, \textit{J. Algebra} {\bf 556} (2020), 287--314.

\bibitem{infdihuniserial} I. Dimitrov, C. Paquette, D. Wehlau and T Xu, Idempotents in the group algebra of the infinite dihedral group, \textit{arXiv}: 2310.09591, Oct 2023, 6 pages.

\bibitem{Fox} D.J.F. Fox, The commutative nonassociative algebra of metric curvature tensors, \textit{Forum Math. Sigma} {\bf 9} (2021), e79, 48 pages.

\bibitem{highwater5} C. Franchi and M. Mainardis, Classifying $2$-generated symmetric axial algebras of Monster type, \textit{J. Algebra} {\bf 596} (2022), 200--218.

\bibitem{HWquo} C. Franchi, M. Mainardis and J. M\textsuperscript{c}Inroy, Quotients of the Highwater algebra, \textit{arXiv}:2205.02200, 46 pages, May 2022.

\bibitem{2gen2btbt} C. Franchi, M. Mainardis and S. Shpectorov, $2$-generated axial algebras of Monster type $(2\bt,\bt)$, \textit{J. Algebra} {\bf 636} (2023), 123--170.

\bibitem{highwater} C. Franchi, M. Mainardis and S. Shpectorov, An infinite-dimensional $2$-generated primitive axial algebra of Monster type, \textit{Ann. Mat. Pura Appl.} (4) {\bf 201} (2022), no. 3, 1279–1293.

\bibitem{doubleMatsuo} A. Galt, V. Joshi, A. Mamontov, S. Shpectorov and A. Staroletov, Double axes and subalgebras of Monster type in Matsuo algebras, \textit{Comm. Algebra} {\bf 49}, no. 10, 4208--4248.

\bibitem{Axial1} J.I. Hall, F. Rehren and S. Shpectorov, Universal axial algebras and a theorem of Sakuma, \textit{J. Algebra} {\bf 421} (2015), 394--424.

\bibitem{Axial2} J.I. Hall, F. Rehren and S. Shpectorov, Primitive axial algebras of Jordan type, \textit{J. Algebra} {\bf 437} (2015), 79--115.

\bibitem{JoshiMRes} V. Joshi, \textit{Double axial algebras}, MRes thesis, University of Birmingham, 2018.

\bibitem{axialstructure} S.M.S. Khasraw, J. M\textsuperscript{c}Inroy and S. Shpectorov, On the structure of axial algebras, \textit{Trans.\ Amer.\ Math.\ Soc.} {\bf 373} (2020), 2135--2156.

\bibitem{AxialTools} J. M\textsuperscript{c}Inroy, Axial Tools -- a {\sc magma} package, \href{https://github.com/JustMaths/AxialTools}{\url{https://github.com/JustMaths/AxialTools}}.

\bibitem{githubcode} J. M\textsuperscript{c}Inroy and A.W. Mir, Properties of examples,  Axial algebras of Monster type -- a {\sc magma} package, \href{https://github.com/JustMaths/AxialMonsterType}{\url{https://github.com/JustMaths/AxialMonsterType}}.

\bibitem{axialconstruction} J. M\textsuperscript{c}Inroy and S. Shpectorov, An expansion algorithm for constructing axial algebras, \textit{J. Algebra} {\bf 550} (2020), 379--409.

\bibitem{survey} J. M\textsuperscript{c}Inroy and S. Shpectorov, Axial algebras of Jordan and Monster type, \textit{Groups St Andrews $2022$ in Newcastle}, London Mathematical Society Lecture Note Series {\bf 496}, CUP, 2024, 246--294.

\bibitem{forbidden} J. M\textsuperscript{c}Inroy and S. Shpectorov, From forbidden configurations to a classification of some axial algebras of Monster type, \textit{J. Algebra} {\bf 627} (2023), 58--105.

\bibitem{splitspin} J. M\textsuperscript{c}Inroy and S. Shpectorov, Split spin factor algebras, \textit{J. Algebra} {\bf 595} (2022), 380--397.

\bibitem{Maddycode} M. Pfeiffer and M. Whybrow, Constructing Majorana Representations, \textit{arXiv}:1803.10723, 19 pages, Mar 2018.

\bibitem{gendihedral} F. Rehren, Generalised dihedral subalgebras from the Monster, \textit{Trans. Amer. Math. Soc.} {\bf 369} (2017), no. 10, 6953--6986.


\bibitem{Tkachev} V.G. Tkachev, The universality of one half in commutative nonassociative algebras with identities, \textit{J. Algebra} {\bf 569} (2021), 466--510.

\bibitem{whybrowgraded} M. Whybrow,  Graded $2$-generated axial algebras, \textit{arXiv}:2005.03577, 33 pages.

\bibitem{yabe} T. Yabe, On the classification of 2-generated axial algebras of Majorana type, \textit{J. Algebra} {\bf 619} (2023), 347--382.
\end{thebibliography}
\end{document}